\documentclass[article,a4paper]{aram}

\usepackage[color]{aram}

\usepackage[category]{aramtikz}

\setupheaders

\title{Maximal degree subposets of $\nu$-Tamari lattices}
\author[A.~Dermenjian]{Aram Dermenjian}
\thanks{This research was supported by NSERC and the Heilbronn Institute for Mathematical Research}
\address[A.~Dermenjian]{The University of Manchester and Heilbronn Institute for Mathematical Research}
\email{aram.dermenjian.math@gmail.com}
\urladdr{http://dermenjian.com}

\let\max\relax
\DeclareMathOperator{\max}{max}

\DeclareMathOperator{\horizDist}{horiz}
\NewDocumentCommand{\horiz}{om}{%
  \IfNoValueTF{#1}
  {\horizDist\left( #2 \right)}
  {\horizDist_{#1}\left( #2 \right)}%
}
\DeclareMathOperator{\indeg}{in}
\DeclareMathOperator{\outdeg}{out}
\newcommand{\maxin}{{\max_{\indeg}}}
\newcommand{\maxout}{{\max_{\outdeg}}}

\newcommand{\maxnu}[2]{#1\left( \mbbT_{#2} \right)}
\newcommandx{\maxinnu}[1][1=\nu]{\maxnu{\maxin}{#1}}
\newcommandx{\maxoutnu}[1][1=\nu]{\maxnu{\maxout}{#1}}

\newcommand{\Dnu}[2]{\mcD_{ {#2}_{#1}}}
\newcommandx{\Dinnu}[1][1=\nu]{\Dnu{\maxin}{ {#1}}}
\newcommandx{\Doutnu}[1][1=\nu]{\Dnu{\maxout}{ {#1}}}
\newcommand{\Dinnm}{\Dinnu[{n,m}]}
\newcommand{\Doutnm}{\Doutnu[{n,m}]}

\newcommand{\Tnu}[2]{\mbbT_{ {#2}_{#1}}}
\newcommandx{\Tinnu}[1][1=\nu]{\Tnu{\indeg}{#1}}
\newcommandx{\Toutnu}[1][1=\nu]{\Tnu{\outdeg}{#1}}
\newcommand{\Tinnm}{\Tinnu[{n,m}]}
\newcommand{\Toutnm}{\Toutnu[{n,m}]}

\newcommand{\Tup}[2]{\tau_{#2}\left( #1 \right)}
\newcommand{\Tdown}[2]{\delta_{#2}\left( #1 \right)}
\newcommand{\Gup}[2]{\gamma_{#2}\left( #1 \right)}

\DeclareMathOperator{\LA}{LA}
\DeclareMathOperator{\RA}{RA}
\newcommand{\laf}{\b{\LA}}
\newcommand{\raf}{\b{\RA}}
\begin{document}
%
%

\begin{abstract}
    In this paper, we study two different subposets of the $\nu$-Tamari lattice: one in which all elements have maximal in-degree and one in which all elements have maximal out-degree.
    The maximal in-degree and maximal out-degree of a $\nu$-Dyck path turns out to be the size of the maximal staircase shape path that fits weakly above $\nu$.
    For $m$-Dyck paths of height $n$, we further show that the maximal out-degree poset is poset isomorphic to the $\nu$-Tamari lattice of $(m-1)$-Dyck paths of height $n$, and the maximal in-degree poset is poset isomorphic to the $(m-1)$-Dyck paths of height $n$ together with a greedy order.
    We show these two isomorphisms and give some properties on $\nu$-Tamari lattices along the way.
\end{abstract}

\maketitle

\section{Introduction}
Given a lattice path $\nu$ a $\nu$-Dyck path is a lattice path which is weakly above $\nu$ using only north (N) and east (E) steps.
The $\nu$-Tamari lattice is a partial order on the set of all $\nu$-Dyck paths which was first introduced in \cite{PV-extTam} by Pr\'eville-Ratelle and Viennot as a further generalisation to $m$-Tamari lattices from Tamari lattices.
Tamari lattices were introduced by Tamari in \cite{Tamari_1954} and are precisely the $\nu$-Tamari lattices where $\nu = (NE)^n$ for some $n$.
The $m$-Tamari lattices were then introduced by Bergeron and Pr\'eville-Ratelle in \cite{Bergeron_2012} as a way to study diagonal harmonics.
These $m$-Tamari lattices have been heavily studied in recent years due to their connections with combinatorics, algebra and geometry, see \cite{Francois_2012,BFP-noIntTam,Ceballos_2018,Chapoton_2005}.
In recent years, the $\nu$-Tamari lattices themselves have been receiving more attention, see \cite{Bell_2021_2}, \cite{Bell_2021}, \cite{Ceballos_2018}, \cite{Ceballos_2020}, \cite{Defant_2022} and \cite{Fang_2017}.
The $\nu$-Tamari lattices and their related definitions are given in \autoref{sec:tamari_orders}.

In this article we concurrently study two particular subposets of the $\nu$-Tamari lattices: the subposets whose elements have maximal in-degree or out-degree.
In some sense, the easier of the two subposets is the subposet whose elements are those with maximal out-degree.
In the case where $\nu = (NE^m)^{n}$ the subposet whose elements have maximal out-degree turns out to be poset isomorphic to the $\nu'$-Tamari lattice where $\nu' = (NE^{m-1})^n$.
On the other hand, the subposet whose elements are those with maximal in-degree turns out to be poset isomorphic to the $\nu'$-Dyck paths together with a greedy order.
This greedy order is similar to the Tamari order except we take ``as much as possible'' when going up.
In other words, we use hit points rather than touch points when ascending in order.
We first define these subposets and give some properties in \autoref{sec:subposets}.
Then in \autoref{sec:isomorphisms}, we restrict to $m$-Dyck paths of height $n$ to show the poset isomorphisms.
In the arbitrary $\nu$ case, things break down and there is unfortunately no ``nice'' results as we explain at the end of the article..

\section{Tamari orders}
\label{sec:tamari_orders}
We start by defining the $\nu$-Tamari order and give some basic definitions which we use for the latter half of the paper.

\subsection{\texorpdfstring{$\nu$}{v}-Dyck paths}
Let $\nu$ be a path from $(0,0)$ to $(s_E,s_N)$ consisting exclusively of north and east steps.
The path $\nu$ can be expressed as a word over the alphabet $\left\{ N,E \right\}$ with $s_E$ number of $E$ (east) steps and $s_N$ number of $N$ (north) steps.
A \defn{$\nu$-Dyck path} is a path from $(0,0)$ to $(s_E,s_N)$ consisting exclusively of north and east steps which is weakly above $\nu$.
The set of all $\nu$-Dyck paths is denoted by $\mcD_\nu$.
A \defn{standard Dyck path} of height $n$, denoted $\mcD_n$, is a $\nu$-Dyck path where $\nu$ is the path $(NE)^n$ from $(0,0)$ to $(n,n)$.
An \defn{$m$-Dyck path} of height $n$, denoted $\mcD_{n,m}$, is a $\nu$-Dyck path where $\nu$ is the path $(NE^m)^n$ from $(0,0)$ to $(mn, n)$.

\begin{example}
    Let $\nu$ be the path from $(0,0)$ to $(8,3)$ which is represented by the word $EEEE NEE NE NE$.
    It is then represented by the path with $8$ east steps and $3$ north steps in red in the figure below.
    Let $D$ be the $\nu$-Dyck path represented by the word $EE NEE N NEEEE$ which is weakly above $\nu$.
    It is depicted in the figure below by the black path.
    
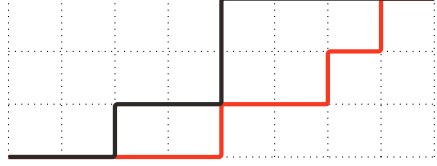
\begin{center}
        \begin{tikzpicture}[scale=0.7]
            \draw[dotted] (0, 0) grid (8, 3);
            \draw[rounded corners=1, color=BrightRed, ultra thick] (0, 0) -- (4, 0) -- (4, 1) -- (6, 1) -- (6, 2) -- (7,2) -- (7,3) -- (8,3);
            \draw[rounded corners=1, color=Black, ultra thick] (0,0) -- (2,0) -- (2,1) -- (4,1) -- (4,3) -- (8,3);
        \end{tikzpicture}
    \end{center}
\end{example}

\subsection{\texorpdfstring{$\nu$}{v}-Tamari order}
Let $D$ be an arbitrary $\nu$-Dyck path from $(0,0)$ to $(s_E,s_N)$.
Given two points $p$ and $p'$ on the path $D$, let $D_{[p,p']}$ denote the subword of $D$ from $p$ to $p'$.
For $i \in [s_N] = \left\{ 1, 2, \ldots, s_N \right\}$ we denote by $r_i^D$ the point on the path $D$ before we take the $i$-th north step.
The $r_i^D$ is referred to as the \defn{$i$-th right hand point of $D$} since on the $i$-th row, it is the point on the Dyck path which is the furthest to the right.
If $D$ is obvious, we will use $r_i$ as a shorthand for $r_i^D$.

\begin{example}
    \label{ex:nu-Dyck}
    Let $D = EE NEE NE NEEE$ as before.
    Then
    \begin{gather*}
        r_1 = (2,0) \qquad r_2 = (4,1) \qquad r_3 = (4,2)\\
        D_{[(0,0),r_1]} = EE \qquad D_{[(0,0),r_2]} = EE NEE \qquad D_{[r_2,r_3]} = N.
    \end{gather*}

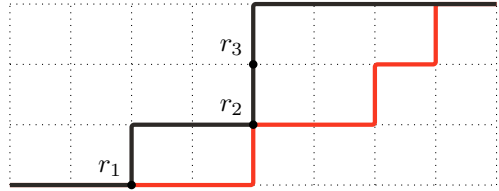
    \begin{center}
        \begin{tikzpicture}[scale=0.8]
            \draw[dotted] (0, 0) grid (8, 3);
            \draw[rounded corners=1, color=BrightRed, ultra thick] (0, 0) -- (4, 0) -- (4, 1) -- (6, 1) -- (6, 2) -- (7,2) -- (7,3) -- (8,3);
            \draw[rounded corners=1, color=Black, ultra thick] (0,0) -- (2,0) -- (2,1) -- (4,1) -- (4,3) -- (8,3);
            \draw (2,0) circle[radius=1pt];
            \fill (2,0) circle[radius=2pt] node[above left] {$r_1$};
            \draw (4,1) circle[radius=1pt];
            \fill (4,1) circle[radius=2pt] node[above left] {$r_2$};
            \draw (4,2) circle[radius=1pt];
            \fill (4,2) circle[radius=2pt] node[above left] {$r_3$};
        \end{tikzpicture}
    \end{center}
\end{example}

Let $\horizDist_\nu: D \to \N$ be the map which sends each point $p$ on the $\nu$-Dyck path $D$ to the maximum number of east steps we can take from $p$ before going past $\nu$.
If the $\nu$ is clear, we will use $\horizDist$ instead of $\horizDist_\nu$.
For a point $p$ on $D$, we call $\horiz[\nu]{p}$ the \defn{horizontal distance of $p$}.

Let $t_i^D$ denote the first point on $D$ which comes after $r_i^D$ such that $\horiz{t_i^D} = \horiz{r_i^D}$.
Let $h_i^D$ denote the first point on $D$ which comes after $r_i^D$ such that $\horiz{h_i^D} = \horiz{r_i^D}$ and where $h_i^D$ is either followed by an east step in $D$ or is the final point in $D$.
We call $t_i^D$ the \defn{$i$-th touch point of $D$} and $h_i^D$ the \defn{$i$-th hit point of $D$}.
As before, if the $D$ is unambiguous, we use $t_i$ and $h_i$ to represent $t_i^D$ and $h_i^D$ respectively.
Since $\horiz{r_i^D}$ is always positive and the final point always has horizontal distance $\horiz{(s_E, s_N)} = 0$, both $t_i^D$ and $h_i^D$ always exist.
Notice that $h_i$ is followed by an east step if and only if $\horiz{h_i} \neq 0$.

Knowing that $D$ can be viewed as an ordered set of points $(p_1, \cdots, p_{s_N + s_E + 1})$, it will be useful to view $\horizDist_\nu$ as a vector where
\[
    \horizDist_{\nu} = \left( \horiz[\nu]{p_1}, \ldots, \horiz[\nu]{p_{s_N + s_E + 1}}\right).
\]
This vector is called the \defn{horizontal distance vector of $D$}.
\begin{example}
    Given the $\nu$-Dyck path $D$ from \autoref{ex:nu-Dyck}, we label each point $p$ on $D$ with the value of $\horiz[\nu]{p}$.
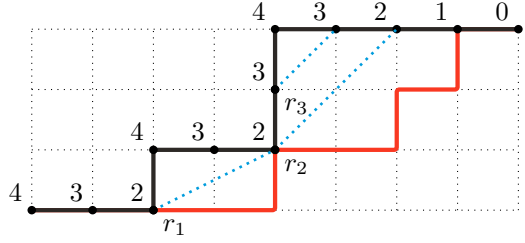
    \begin{center}
        \begin{tikzpicture}[scale=0.8]
            \draw[dotted] (0, 0) grid (8, 3);
            \draw[rounded corners=1, color=BrightRed, ultra thick] (0, 0) -- (4, 0) -- (4, 1) -- (6, 1) -- (6, 2) -- (7,2) -- (7,3) -- (8,3);
            \draw[rounded corners=1, color=Black, ultra thick] (0,0) -- (2,0) -- (2,1) -- (4,1) -- (4,3) -- (8,3);
        \draw (0,0) circle[radius=1pt];
        \fill (0,0) circle[radius=2pt] node[above left] {$4$};
        \draw (1,0) circle[radius=1pt];
        \fill (1,0) circle[radius=2pt] node[above left] {$3$};
        \draw (2,0) circle[radius=1pt];
        \fill (2,0) circle[radius=2pt] node[above left] {$2$};
        \draw (2,1) circle[radius=1pt];
        \fill (2,1) circle[radius=2pt] node[above left] {$4$};
        \draw (3,1) circle[radius=1pt];
        \fill (3,1) circle[radius=2pt] node[above left] {$3$};
        \draw (4,1) circle[radius=1pt];
        \fill (4,1) circle[radius=2pt] node[above left] {$2$};
        \draw (4,2) circle[radius=1pt];
        \fill (4,2) circle[radius=2pt] node[above left] {$3$};
        \draw (4,3) circle[radius=1pt];
        \fill (4,3) circle[radius=2pt] node[above left] {$4$};
        \draw (5,3) circle[radius=1pt];
        \fill (5,3) circle[radius=2pt] node[above left] {$3$};
        \draw (6,3) circle[radius=1pt];
        \fill (6,3) circle[radius=2pt] node[above left] {$2$};
        \draw (7,3) circle[radius=1pt];
        \fill (7,3) circle[radius=2pt] node[above left] {$1$};
        \draw (8,3) circle[radius=1pt];
        \fill (8,3) circle[radius=2pt] node[above left] {$0$};
            \draw[color=ProcessCyan, line width=1, dotted] (2,0) -- (4,1);
            \draw[color=ProcessCyan, line width=1, dotted] (4,1) -- (6,3);
            \draw[color=ProcessCyan, line width=1, dotted] (4,2) -- (5,3);
            \fill (2,0) circle[radius=2pt] node[below right] {$r_1$};
            \fill (4,1) circle[radius=2pt] node[below right] {$r_2$};
            \fill (4,2) circle[radius=2pt] node[below right] {$r_3$};
        \end{tikzpicture}
    \end{center}
    This gives us the horizontal distance vector:
    \[
        \horizDist = (4, 3, 2, 4, 3, 2, 3, 4, 3, 2, 1, 0).
    \]
    For touch and hit points we analyze each of the three rows in $D$.
    When $i = 1$ then $r_1 = (2,0)$ and $\horiz{r_1} = 2$.
    The first point after $r_1$ whose horizontal value is $2$ is the point $(4,1)$.
    Therefore $t_1 = (4,1)$ is the $1$st touch point of $D$.
    Since for the path $D$ there is not an east step after $(4,1)$ it is not the $1$st hit point of $D$.
    The next time a point on $D$ has horizontal distance of $2$ is at $(6,3)$.
    As $(6,3)$ is followed by an east step in $D$ then $h_1 = (6,3)$.
    We end up with the following:
    \begin{gather*}
        t_1 = (4,1) \quad t_2 = (6,3) \quad t_3 = (5,3)\\
        h_1 = (6,3) \quad h_2 = (6,3) \quad h_3 = (5,3)\\
    \end{gather*}
\end{example}

Let $D$ be a $\nu$-Dyck path from $(0,0)$ to $(s_E,s_N)$.
If the $i$-th right hand point $r_i^D$ is preceded by an east step, we define the $\nu$-Dyck path $\Tup{D}{i}$ in the following way.
Let $d$ denote the subword of $D_{[(0,0),r_i^D]}$ where the final $E$ has been removed.
Let $t = D_{[r_i^D, t_i^D]}$ and $f = D_{[t_i^D, (s_E, s_N)]}$, \ie $D = d E t f$.
Then we let $\Tup{D}{i}$ be the word $d t E f$.
In other words, we move the east step before $r_i^D$ to just after the $i$-th touch point $t_i^D$ keeping the rest of the path the same.

\begin{example}
    Continuing our example from before:
    
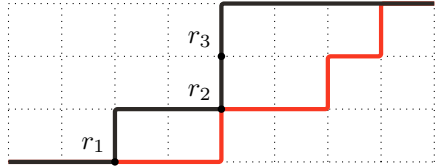
\begin{center}
        \begin{tikzpicture}[scale=0.7]
            \draw[dotted] (0, 0) grid (8, 3);
            \draw[rounded corners=1, color=BrightRed, ultra thick] (0, 0) -- (4, 0) -- (4, 1) -- (6, 1) -- (6, 2) -- (7,2) -- (7,3) -- (8,3);
            \draw[rounded corners=1, color=Black, ultra thick] (0,0) -- (2,0) -- (2,1) -- (4,1) -- (4,3) -- (8,3);
            \draw (2,0) circle[radius=1pt];
            \fill (2,0) circle[radius=2pt] node[above left] {$r_1$};
            \draw (4,1) circle[radius=1pt];
            \fill (4,1) circle[radius=2pt] node[above left] {$r_2$};
            \draw (4,2) circle[radius=1pt];
            \fill (4,2) circle[radius=2pt] node[above left] {$r_3$};
        \end{tikzpicture}
    \end{center}
    Recall that
    \[
        D_{[(0,0),r_1]} = EE \quad D_{[(0,0),r_2]} = EE NEE \quad D_{[(0,0),r_3]} = EE NEE N.
    \]
    Since $r_1$ and $r_2$ are each preceded by an east step, we can find $\Tup{D}{1}$ and $\Tup{D}{2}$, but since $r_3$ is not preceded by an east step, $\Tup{D}{3}$ is not defined.

    We first calculate $\Tup{D}{1}$.
    Recall that $r_1 = (2,0)$ and $t_1 = (4,1)$.
    Then
    \begin{gather*}
        d = E\\
        t = D_{[r_1, t_1]} = D_{[(2,0),(4,1)]} = NEE\\
        f = D_{[t_1, (s_E, s_N)]} = D_{[(4,1),(8,3)]} = NNEEEE
    \end{gather*}
    where $d$ is the subword of $D_{[(0,0),r_1]} = EE$ with the final $E$ removed.
    Therefore $\Tup{D}{1} = d t E f = E NEE E NNEEEE$.
    Drawing $D$ (dashed) and $\Tup{D}{1}$ (solid) in the figure below, we have that $\Tup{D}{1}$ is weakly above $D$ and therefore weakly above $\nu$.
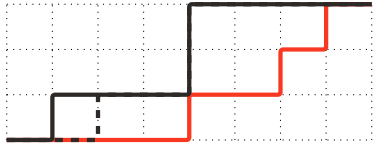
    \begin{center}
        \begin{tikzpicture}[scale=0.6]
            \draw[dotted] (0, 0) grid (8, 3);
            \draw[rounded corners=1, color=BrightRed, ultra thick] (0, 0) -- (4, 0) -- (4, 1) -- (6, 1) -- (6, 2) -- (7,2) -- (7,3) -- (8,3);
            \draw[dashed, rounded corners=1, color=Black, ultra thick] (0,0) -- (2,0) -- (2,1) -- (4,1) -- (4,3) -- (8,3);
            \draw[rounded corners=1, color=Black, ultra thick] (0,0) -- (1,0) -- (1,1) -- (4,1) -- (4,3) -- (8,3);
        \end{tikzpicture}
    \end{center}

    Let us also calculate $\Tup{D}{2}$.
    We have $d = EE NE$, $t = NNEE$, $f = EE$.
    Therefore $\Tup{D}{2} = EE NE NNEE E EE$ and, as before, we have $\Tup{D}{2}$ (solid) weakly above $D$ (dashed) as can be seen in the figure below.
    
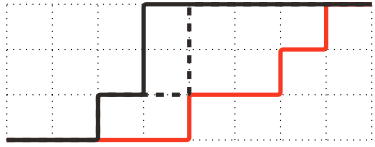
\begin{center}
        \begin{tikzpicture}[scale=0.6]
            \draw[dotted] (0, 0) grid (8, 3);
            \draw[rounded corners=1, color=BrightRed, ultra thick] (0, 0) -- (4, 0) -- (4, 1) -- (6, 1) -- (6, 2) -- (7,2) -- (7,3) -- (8,3);
            \draw[dashed, rounded corners=1, color=Black, ultra thick] (0,0) -- (2,0) -- (2,1) -- (4,1) -- (4,3) -- (8,3);
            \draw[rounded corners=1, color=Black, ultra thick] (0,0) -- (2,0) -- (2,1) -- (3,1) -- (3,3) -- (8,3);
        \end{tikzpicture}
    \end{center}
\end{example}

The \defn{$\nu$-Tamari order} is then the order on $\nu$-Dyck paths where $D$ is covered by $\Tup{D}{i}$ whenever $\Tup{D}{i}$ is defined.
We denote this poset by $\mbbT_\nu = (\mcD_\nu, \leq_T)$ and let $D \cover_T \Tup{D}{i}$ denote that $D$ is covered by $\Tup{D}{i}$ (whenever it is defined).
This order on $\nu$-Dyck paths was originally defined in \cite{PV-extTam}.
We similarly define $\mbbT_{n,m}  = \left( \mcD_{n,m}, \leq_T \right)$ and $\mbbT_{n} = \left( \mcD_n, \leq_T \right)$ to be the Tamari posets on the $m$-Dyck paths of height $n$ and the standard Dyck paths of height $n$ respectively.
The poset $\mbbT_n$ was first described in \cite{Tamari_1954} and the poset $\mbbT_{n,m}$ was first defined in \cite{Bergeron_2012}.

\begin{example}
    \label{ex:tamari}
    As an example, let $\nu = NEE NEE NEE$ be a path which gives us the $2$-Dyck paths of height $3$.
    Then $\mbbT_\nu$ is given by the following Hasse diagram:
    
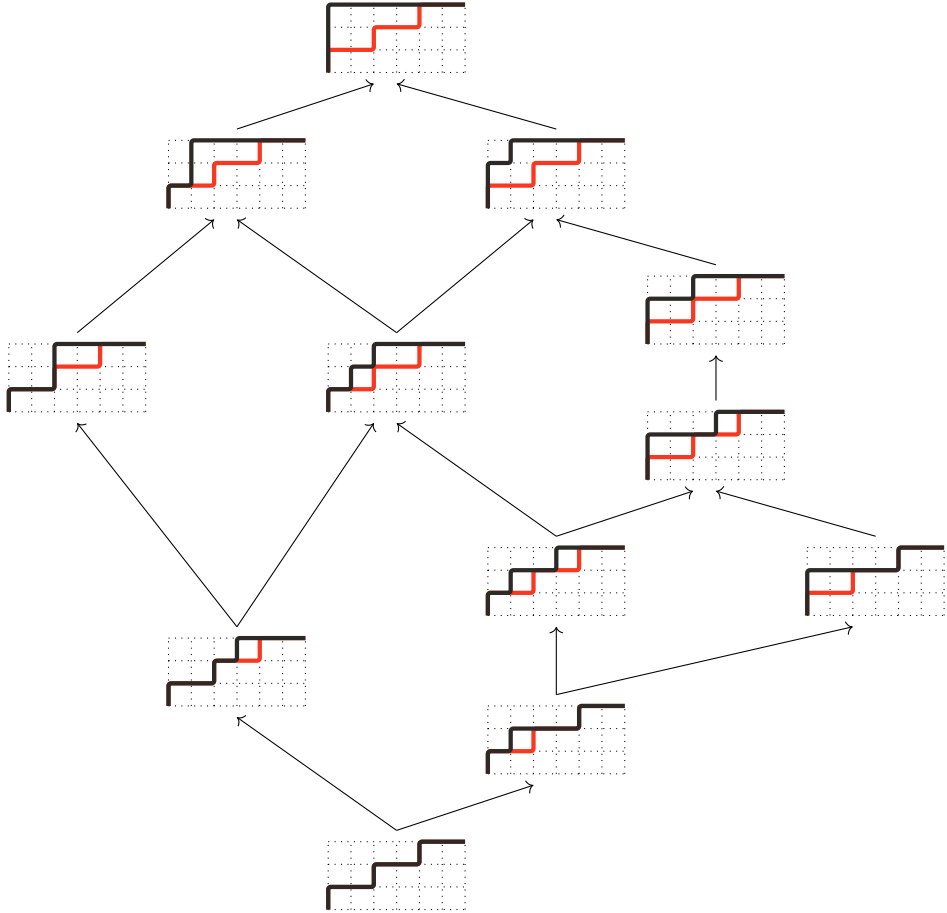
\begin{center}
        \begin{tikzpicture}[scale=0.3]
            \begin{scope}[shift={(0,0)}]
                \draw[dotted] (0, 0) grid (6, 3);
                \draw[rounded corners=1, color=BrightRed, ultra thick] (0, 0) -- (0, 1) -- (2, 1) -- (2, 2) -- (4, 2) -- (4,3) -- (6,3);
                \draw[rounded corners=1, color=Black, ultra thick] (0, 0) -- (0, 1) -- (2, 1) -- (2, 2) -- (4, 2) -- (4,3) -- (6,3);
            \end{scope}

            \begin{scope}[shift={(7,6)}]
                \draw[dotted] (0, 0) grid (6, 3);
                \draw[rounded corners=1, color=BrightRed, ultra thick] (0, 0) -- (0, 1) -- (2, 1) -- (2, 2) -- (4, 2) -- (4,3) -- (6,3);
                \draw[rounded corners=1, color=Black, ultra thick] (0, 0) -- (0, 1) -- (1, 1) -- (1, 2) -- (4, 2) -- (4,3) -- (6,3);
            \end{scope}
            \begin{scope}[shift={(-7,9)}]
                \draw[dotted] (0, 0) grid (6, 3);
                \draw[rounded corners=1, color=BrightRed, ultra thick] (0, 0) -- (0, 1) -- (2, 1) -- (2, 2) -- (4, 2) -- (4,3) -- (6,3);
                \draw[rounded corners=1, color=Black, ultra thick] (0, 0) -- (0, 1) -- (2, 1) -- (2, 2) -- (3, 2) -- (3,3) -- (6,3);
            \end{scope}

            \begin{scope}[shift={(21,13)}]
                \draw[dotted] (0, 0) grid (6, 3);
                \draw[rounded corners=1, color=BrightRed, ultra thick] (0, 0) -- (0, 1) -- (2, 1) -- (2, 2) -- (4, 2) -- (4,3) -- (6,3);
                \draw[rounded corners=1, color=Black, ultra thick] (0, 0) -- (0, 2) -- (4, 2) -- (4,3) -- (6,3);
            \end{scope}
            \begin{scope}[shift={(7,13)}]
                \draw[dotted] (0, 0) grid (6, 3);
                \draw[rounded corners=1, color=BrightRed, ultra thick] (0, 0) -- (0, 1) -- (2, 1) -- (2, 2) -- (4, 2) -- (4,3) -- (6,3);
                \draw[rounded corners=1, color=Black, ultra thick] (0, 0) -- (0, 1) -- (1, 1) -- (1, 2) -- (3, 2) -- (3,3) -- (6,3);
            \end{scope}
            \begin{scope}[shift={(0,22)}]
                \draw[dotted] (0, 0) grid (6, 3);
                \draw[rounded corners=1, color=BrightRed, ultra thick] (0, 0) -- (0, 1) -- (2, 1) -- (2, 2) -- (4, 2) -- (4,3) -- (6,3);
                \draw[rounded corners=1, color=Black, ultra thick] (0, 0) -- (0, 1) -- (1, 1) -- (1, 2) -- (2, 2) -- (2,3) -- (6,3);
            \end{scope}
            \begin{scope}[shift={(-14,22)}]
                \draw[dotted] (0, 0) grid (6, 3);
                \draw[rounded corners=1, color=BrightRed, ultra thick] (0, 0) -- (0, 1) -- (2, 1) -- (2, 2) -- (4, 2) -- (4,3) -- (6,3);
                \draw[rounded corners=1, color=Black, ultra thick] (0, 0) -- (0, 1) -- (2, 1) -- (2, 3) -- (6,3);
            \end{scope}

            \begin{scope}[shift={(14,19)}]
                \draw[dotted] (0, 0) grid (6, 3);
                \draw[rounded corners=1, color=BrightRed, ultra thick] (0, 0) -- (0, 1) -- (2, 1) -- (2, 2) -- (4, 2) -- (4,3) -- (6,3);
                \draw[rounded corners=1, color=Black, ultra thick] (0,0) -- (0,2) -- (3,2) -- (3,3) -- (6,3);
            \end{scope}
            \begin{scope}[shift={(14,25)}]
                \draw[dotted] (0, 0) grid (6, 3);
                \draw[rounded corners=1, color=BrightRed, ultra thick] (0, 0) -- (0, 1) -- (2, 1) -- (2, 2) -- (4, 2) -- (4,3) -- (6,3);
                \draw[rounded corners=1, color=Black, ultra thick] (0,0) -- (0,2) -- (2,2) -- (2,3) -- (6,3);
            \end{scope}
            \begin{scope}[shift={(7,31)}]
                \draw[dotted] (0, 0) grid (6, 3);
                \draw[rounded corners=1, color=BrightRed, ultra thick] (0, 0) -- (0, 1) -- (2, 1) -- (2, 2) -- (4, 2) -- (4,3) -- (6,3);
                \draw[rounded corners=1, color=Black, ultra thick] (0,0) -- (0,2) -- (1,2) -- (1,3) -- (6,3);
            \end{scope}
            \begin{scope}[shift={(0,37)}]
                \draw[dotted] (0, 0) grid (6, 3);
                \draw[rounded corners=1, color=BrightRed, ultra thick] (0, 0) -- (0, 1) -- (2, 1) -- (2, 2) -- (4, 2) -- (4,3) -- (6,3);
                \draw[rounded corners=1, color=Black, ultra thick] (0,0) -- (0,3) -- (6,3);
            \end{scope}
            \begin{scope}[shift={(-7,31)}]
                \draw[dotted] (0, 0) grid (6, 3);
                \draw[rounded corners=1, color=BrightRed, ultra thick] (0, 0) -- (0, 1) -- (2, 1) -- (2, 2) -- (4, 2) -- (4,3) -- (6,3);
                \draw[rounded corners=1, color=Black, ultra thick] (0,0) -- (0,1) -- (1,1) -- (1,3) -- (6,3);
            \end{scope}

            \draw[->] (3,3.5) -- (9,5.5); 
            \draw[->] (3,3.5) -- (-4,8.5); 

            \draw[->] (-4,12.5) -- (-11, 21.5); 
            \draw[->] (-4,12.5) -- (2,21.5); 
            \draw[->] (10,9.5) -- (10,12.5); 
            \draw[->] (10,9.5) -- (23,12.5); 

            \draw[->] (-11,25.5) -- (-5,30.5); 
            \draw[->] (3,25.5) -- (-4,30.5); 
            \draw[->] (3,25.5) -- (9,30.5); 
            \draw[->] (10,16.5) -- (3,21.5); 
            \draw[->] (10,16.5) -- (16,18.5); 
            \draw[->] (24,16.5) -- (17,18.5); 

            \draw[->] (-4,34.5) -- (2,36.5); 
            \draw[->] (10,34.5) -- (3,36.5); 
            \draw[->] (17,22.5) -- (17,24.5); 
            \draw[->] (17,28.5) -- (10,30.5); 
        \end{tikzpicture}
    \end{center}
\end{example}

\begin{remark}
    \label{rem:touch_to_hit}
    It is worthwhile to remark that going up in the $\nu$-Tamari order turns a touch point into a hit point.
    Indeed, given $D$ a $\nu$-Dyck path, then $\Tup{D}{i}$ exists if $r_i^D$ has a preceding $E$ step.
    Recalling that $\horiz{r_i^D} = \horiz{t_i^D}$ then $D = dEtf$ and $\Tup{D}{i} = dtEf$ imply that $\horiz{r_i^{\Tup{D}{i}}} = \horiz{t_i^{\Tup{D}{i}}} = \horiz{r_i^D} + 1$ since we put an east step after both $r_i^D$ and $t_i^D$.
    But now, since $t_i^{\Tup{D}{i}}$ has an east step following it and since it is the first point after $r_i^{\Tup{D}{i}}$ with the same horizontal distance, then $t_i^{\Tup{D}{i}} = h_i^{\Tup{D}{i}}$.
\end{remark}

\subsection{Reversing the \texorpdfstring{$\nu$}{v}-Tamari order}
Although we know how to go up in the $\nu$-Tamari order, it will be useful for us to know how to go down as well.
Recall that given a $\nu$-Dyck path $D$ then $\Tup{D}{i}$ exists if the $i$-th right hand point $r_i^D$ is preceded by an east step.
Then $D = dEtf$ and $\Tup{D}{i} = dtEf$ and by \autoref{rem:touch_to_hit} we know that in $\Tup{D}{i}$ the $i$-th touch and hit points coincide, \ie $t_i^{\Tup{D}{i}} = h_i^{\Tup{D}{i}}$.
Therefore, to go down in the $\nu$-Tamari order, it suffices to find when hit points and touch points coincide.

Let $D$ be a $\nu$-Dyck path and let $r_i^D$ be its $i$-th right hand point.
If $t_i^D = h_i^D$ and $h_i^D$ is followed by an east step, then let $d$ be the subpath of $D$ from $0$ to $r_i^D$, let $t$ be the subpath from $r_i^D$ to $h_i^D$, and let $f$ be the subpath from $h_i^D$ to the end with the first $E$ removed.
In other words $D = dtEf$.
Let $\Tdown{D}{i}$ denote the $\nu$-Dyck path where $\Tdown{D}{i} = dEtf$ if it exists.
Notice that $\Tup{\Tdown{D}{i}}{i} = D$.

\begin{lemma}
    \label{lem:down-cover}
    Let $D$ be a $\nu$-Dyck path and $\Tdown{D}{i}$ be defined as above.
    Let $r_i^D$, $t_i^D$, and $h_i^D$ be the $i$-th right hand point, touch point and hit point of $D$ respectively.
    Then $\Tdown{D}{i}$ exists and $\Tdown{D}{i} \cover_T D$ if and only if $t_i^D = h_i^D$ and $\horiz{r_i^D} \neq 0$.
\end{lemma}
\begin{proof}
    If $\Tdown{D}{i}$ exists and $\Tdown{D}{i} \cover_T D$, then by definition $\Tdown{D}{i} = dEtf$ and $D = dtEf$ where $d = D_{[(0,0), r_i^D]}$, $t = D_{[r_i^D, t_i^D]}$ and $f = D_{[t_i^D, (s_E, s_N)]}$ with the first $E$ removed.
    Since $d$ is a prefix for both $D$ and $\Tdown{D}{i}$ we know that $r_i^{\Tdown{D}{i}}$ is exactly one step to the east of $r_i^D$ (since $d$ is followed by an east step in $\Tdown{D}{i}$ but not in $D$), \ie $\horiz{r_i^{\Tdown{D}{i}}} = \horiz{r_i^D} - 1$.
    This is only true if $\horiz{r_i^D} \neq 0$.
    By \autoref{rem:touch_to_hit}, we know that $t_i^{D} = h_i^{D}$ as desired.

    In the other direction, suppose that $t_i^D = h_i^D$ and $\horiz{r_i^D} \neq 0$.
    Since $\horiz{r_i^D} \neq 0$ we know that $h_i^D$ is followed by an east step.
    Therefore, we can break $D$ down into the components $d$, $t$ and $f$ as defined above.
    Then there exists $\Tdown{D}{i}$ such that $\Tdown{D}{i} = dEtf$.
    It suffices to show that $\Tup{\Tdown{D}{i}}{i} = D$ to show that $\Tdown{D}{i} \cover_T D$, in other words, it suffices to show that $t_i^{\Tdown{D}{i}}$ and $t_i^D$ are in the same row.
    Since $t_i^D$ is the first point in $D$ such that $\horiz{r_i^D} = \horiz{t_i^D}$, then shifting all the points between $r_i^D$ and $t_i^D$ over to the east by one decreases all horizontal distances between the two by $1$.
    In other words, $t_i^{\Tdown{D}{i}}$ is the point one step to the east of $t_i^D$ and is also the first point in $\Tdown{D}{i}$ such that $\horiz{r_i^{\Tdown{D}{i}}} = \horiz{t_i^{\Tdown{D}{i}}}$.
    Therefore, we can go up in the $\nu$-Tamari order covering relation and we get $\Tdown{D}{i} \cover_T D$ as desired.
\end{proof}

\subsection{Subposets of the Tamari poset}
Let $\mcD_\nu$ be the set of $\nu$-Dyck paths and let $\mbbT_\nu = (\mcD_\nu, \leq_T)$ denote the $\nu$-Tamari poset.
For $D \in \mcD_\nu$, let $\indeg(D)$ denote the number of elements covered by $D$ in $\mbbT_\nu$.
In other words,
\[
    \indeg(D) = \order{ \set{i}{\Tup{D}{i} \text{ exists}}} \qquad \outdeg(D) = \order{ \set{i}{\Tdown{D}{i} \text{ exists}}}
\]
Similarly, let $\outdeg(D)$ denote the number of elements which cover $D$ in $\mbbT_\nu$.
We call $\indeg(D)$ the \defn{in-degree of $D$} and similarly we call $\outdeg(D)$ the \defn{out-degree of $D$}.
We set the following notation
\begin{gather*}
    \maxoutnu = \max\set{\outdeg(D)}{D \in \mcD_\nu},\\
    \maxinnu = \max\set{\indeg(D)}{D \in \mcD_\nu},\\
    \Doutnu = \set{D \in \mcD_\nu}{\outdeg(D) = \maxout(\mbbT_\nu)}, \text{ and}\\
    \Dinnu = \set{D \in \mcD_\nu}{\indeg(D) = \maxin(\mbbT_\nu)}.
\end{gather*}
where the first two describe the maximum out-degree (in-degree) of a $\nu$-Dyck path and the latter two are the sets of $\nu$-Dyck paths who have this maximal out-degree (in-degree).
Let $\Toutnu$ be the subposet of $\mbbT_\nu$ restricted to $\Doutnu$ and similarly let $\Tinnu$ be the subposet of $\mbbT_\nu$ restricted to $\Dinnu$.

\begin{example}
    Let $\nu = NEE NEE NEE$ be a path which gives us the $2$-Dyck paths of height $3$ as in \autoref{ex:tamari}.
    By an observation of the Hasse diagram, we note that $\maxinnu = \maxoutnu = 2$.
    We then have the following subposets with $\Tinnu$ on the left and $\Toutnu$ on the right.
    
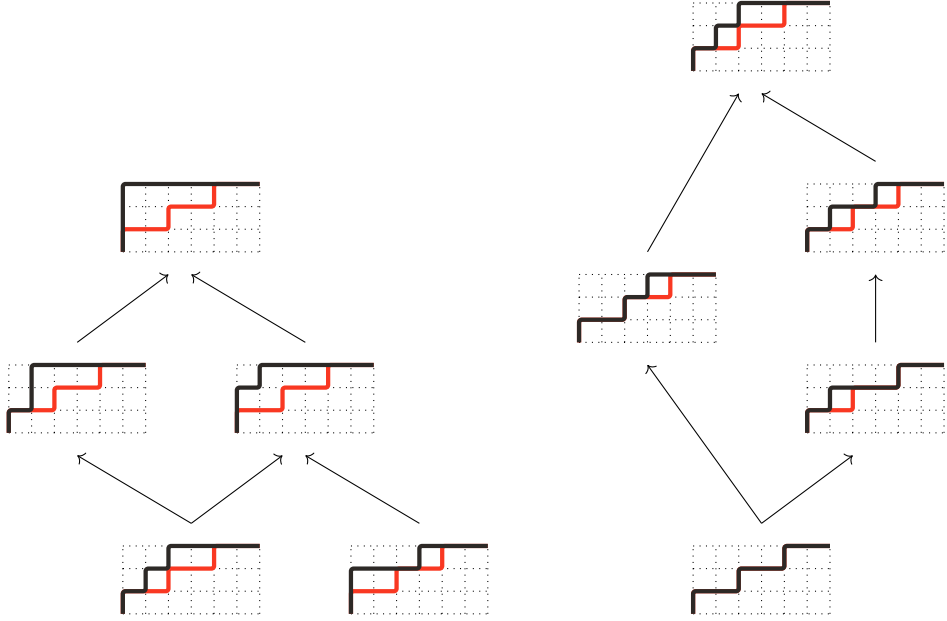
\begin{center}
        \begin{tikzpicture}[scale=0.3]
            \begin{scope}[shift={(0,0)}]
            \begin{scope}[shift={(0,0)}]
                \draw[dotted] (0, 0) grid (6, 3);
                \draw[rounded corners=1, color=BrightRed, ultra thick] (0, 0) -- (0, 1) -- (2, 1) -- (2, 2) -- (4, 2) -- (4,3) -- (6,3);
                \draw[rounded corners=1, color=Black, ultra thick] (0, 0) -- (0, 1) -- (1, 1) -- (1, 2) -- (2, 2) -- (2,3) -- (6,3);
            \end{scope}

            \begin{scope}[shift={(10,0)}]
                \draw[dotted] (0, 0) grid (6, 3);
                \draw[rounded corners=1, color=BrightRed, ultra thick] (0, 0) -- (0, 1) -- (2, 1) -- (2, 2) -- (4, 2) -- (4,3) -- (6,3);
                \draw[rounded corners=1, color=Black, ultra thick] (0,0) -- (0,2) -- (3,2) -- (3,3) -- (6,3);
            \end{scope}
            \begin{scope}[shift={(5,8)}]
                \draw[dotted] (0, 0) grid (6, 3);
                \draw[rounded corners=1, color=BrightRed, ultra thick] (0, 0) -- (0, 1) -- (2, 1) -- (2, 2) -- (4, 2) -- (4,3) -- (6,3);
                \draw[rounded corners=1, color=Black, ultra thick] (0,0) -- (0,2) -- (1,2) -- (1,3) -- (6,3);
            \end{scope}
            \begin{scope}[shift={(0,16)}]
                \draw[dotted] (0, 0) grid (6, 3);
                \draw[rounded corners=1, color=BrightRed, ultra thick] (0, 0) -- (0, 1) -- (2, 1) -- (2, 2) -- (4, 2) -- (4,3) -- (6,3);
                \draw[rounded corners=1, color=Black, ultra thick] (0,0) -- (0,3) -- (6,3);
            \end{scope}
            \begin{scope}[shift={(-5,8)}]
                \draw[dotted] (0, 0) grid (6, 3);
                \draw[rounded corners=1, color=BrightRed, ultra thick] (0, 0) -- (0, 1) -- (2, 1) -- (2, 2) -- (4, 2) -- (4,3) -- (6,3);
                \draw[rounded corners=1, color=Black, ultra thick] (0,0) -- (0,1) -- (1,1) -- (1,3) -- (6,3);
            \end{scope}

            \draw[->] (3,4) -- (-2,7); 
            \draw[->] (3,4) -- (7,7); 
            \draw[->] (-2,12) -- (2,15); 
            \draw[->] (8,12) -- (3,15); 
            \draw[->] (13,4) -- (8,7); 
        \end{scope}
        \begin{scope}[shift={(25,0)}]
            \begin{scope}[shift={(0,0)}]
                \draw[dotted] (0, 0) grid (6, 3);
                \draw[rounded corners=1, color=BrightRed, ultra thick] (0, 0) -- (0, 1) -- (2, 1) -- (2, 2) -- (4, 2) -- (4,3) -- (6,3);
                \draw[rounded corners=1, color=Black, ultra thick] (0, 0) -- (0, 1) -- (2, 1) -- (2, 2) -- (4, 2) -- (4,3) -- (6,3);
            \end{scope}

            \begin{scope}[shift={(5,8)}]
                \draw[dotted] (0, 0) grid (6, 3);
                \draw[rounded corners=1, color=BrightRed, ultra thick] (0, 0) -- (0, 1) -- (2, 1) -- (2, 2) -- (4, 2) -- (4,3) -- (6,3);
                \draw[rounded corners=1, color=Black, ultra thick] (0, 0) -- (0, 1) -- (1, 1) -- (1, 2) -- (4, 2) -- (4,3) -- (6,3);
            \end{scope}
            \begin{scope}[shift={(-5,12)}]
                \draw[dotted] (0, 0) grid (6, 3);
                \draw[rounded corners=1, color=BrightRed, ultra thick] (0, 0) -- (0, 1) -- (2, 1) -- (2, 2) -- (4, 2) -- (4,3) -- (6,3);
                \draw[rounded corners=1, color=Black, ultra thick] (0, 0) -- (0, 1) -- (2, 1) -- (2, 2) -- (3, 2) -- (3,3) -- (6,3);
            \end{scope}

            \begin{scope}[shift={(5,16)}]
                \draw[dotted] (0, 0) grid (6, 3);
                \draw[rounded corners=1, color=BrightRed, ultra thick] (0, 0) -- (0, 1) -- (2, 1) -- (2, 2) -- (4, 2) -- (4,3) -- (6,3);
                \draw[rounded corners=1, color=Black, ultra thick] (0, 0) -- (0, 1) -- (1, 1) -- (1, 2) -- (3, 2) -- (3,3) -- (6,3);
            \end{scope}
            \begin{scope}[shift={(0,24)}]
                \draw[dotted] (0, 0) grid (6, 3);
                \draw[rounded corners=1, color=BrightRed, ultra thick] (0, 0) -- (0, 1) -- (2, 1) -- (2, 2) -- (4, 2) -- (4,3) -- (6,3);
                \draw[rounded corners=1, color=Black, ultra thick] (0, 0) -- (0, 1) -- (1, 1) -- (1, 2) -- (2, 2) -- (2,3) -- (6,3);
            \end{scope}
            \draw[->] (3,4) -- (7,7); 
            \draw[->] (3,4) -- (-2,11); 
            \draw[->] (-2,16) -- (2,23); 
            \draw[->] (8,12) -- (8,15); 
            \draw[->] (8,20) -- (3,23); 
        \end{scope}
        \end{tikzpicture}
    \end{center}
    The subposet on the left is a new poset which will turn out to be poset isomorphic to Dyck paths of height $3$ together with a greedy order.
    The subposet on the right will turn out to be poset isomorphic to the standard Tamari lattice of height $3$.
\end{example}

We discuss these subposets in further detail in \autoref{sec:subposets}.

\subsection{Area functions}
Before going in depth into the subposets $\Toutnu$ and $\Tinnu$, we describe two combinatorial maps which will be useful in our later studies.
Given a path $\nu$ from $(0,0)$ to $(s_E, s_N)$ and a $\nu$-Dyck path $D$, then the \defn{left area function of $D$ with respect to $\nu$} is the function $\laf_D^\nu: [s_N] \to \N$ where $\laf(i)$ is equal to the first component in the coordinates of $r_i^D$.
On the other hand, the \defn{right area function of $D$ with respect to $\nu$} is the function $\raf_D^\nu: [s_N] \to \N$ where $\raf(i)$ is equal to the horizontal distance of $r_i^D$.
Notice that $\laf_D^\nu(i) + \raf_D^\nu(i) = \laf_\nu^\nu(i)$ for all $i \in [s_N]$.

It will be useful to view the functions $\laf_D^\nu$ and $\raf_D^\nu$ as vectors instead of functions.
Therefore, by abuse of notation, we let $\laf_D^\nu = (\laf_D^\nu(1), \ldots, \laf_D^\nu(s_N))$ and $\raf_D^\nu = (\raf_D^\nu(1), \ldots, \raf_D^\nu(s_N))$.
We will respectively call these the \defn{left area vector of $D$ with respect to $\nu$} and \defn{right area vector of $D$ with respect to $\nu$} when viewing these functions as vectors.
Similarly, if $\nu$ is evident, we will use $\laf_D$ to denote $\laf_D^\nu$.

\begin{example}
    Recall the $\nu$-Dyck path $D$ from before:
    
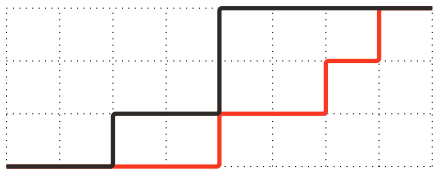
\begin{center}
        \begin{tikzpicture}[scale=0.7]
            \draw[dotted] (0, 0) grid (8, 3);
            \draw[rounded corners=1, color=BrightRed, ultra thick] (0, 0) -- (4, 0) -- (4, 1) -- (6, 1) -- (6, 2) -- (7,2) -- (7,3) -- (8,3);
            \draw[rounded corners=1, color=Black, ultra thick] (0,0) -- (2,0) -- (2,1) -- (4,1) -- (4,3) -- (8,3);
        \end{tikzpicture}
    \end{center}
    Then $\laf_D = (2,4,4)$, $\raf_D = (2, 2, 3)$ and $\laf_\nu = (4, 6, 7)$.
\end{example}

\subsection{Staircase shape \texorpdfstring{$\nu$}{v}-Dyck paths}
A particular type of $\nu$-Dyck path which is useful for this study is a staircase shaped path and it is best described using its left area function.
We say that a $\nu$-Dyck path $D$ is a \defn{staircase shape of size $n$} if $D$ is the path $N^a (EN)^n E^b$ for some $a\geq 0$ and $b\geq 0$.
The left area function for a staircase shape $\nu$-Dyck path of size $n$ is given by $(0, \ldots, 0, 1, 2, \ldots,  n)$.
For a given $\nu$, the \defn{maximal staircase shape $\nu$-Dyck path} is the $\nu$-Dyck path that is staircase shape of size $n$ where $n$ is maximal.
Let $\xi_\nu$ denote the maximal staircase shape $\nu$-Dyck path and let $\sigma_\nu$ denote the size of $\xi_\nu$.
\begin{example}
    Suppose $\nu = NNEENEEN$, then the maximal staircase shape $\nu$-Dyck path is $\xi_\nu = N^2(EN)^2 E^2$:
    
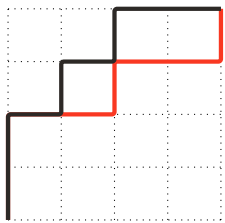
\begin{center}
        \begin{tikzpicture}[scale=0.7]
            \draw[dotted] (0, 0) grid (4, 4);
            \draw[rounded corners=1, color=BrightRed, ultra thick] (0, 0) -- (0,2) -- (2, 2) -- (2,3) -- (4,3) -- (4,4);
            \draw[rounded corners=1, color=Black, ultra thick] (0,0) -- (0,2) -- (1,2) -- (1,3) -- (2,3) -- (2,4) -- (4,4);
        \end{tikzpicture}
    \end{center}
    The left area vector for this $\nu$-Dyck path is given by $(0, 0, 1, 2)$ and therefore the maximal staircase shape $\nu$-Dyck path has size $\sigma_\nu = 2$.
\end{example}

We next provide a handy algorithm to determine the maximal staircase shape $\nu$-Dyck path using the partition of $\nu$.
\textbf{Staircase algorithm:}
Let $\Lambda = \laf_\nu$ be the left area vector of $\nu$ and start with $i = 1$, proceeding inductively.
\begin{enumerate}
    \item Find the first integer $j$ in $\Lambda$ where $i \leq j$.
    \item If no such integer exists, we're done. Else, we assume $j$ is in the $d$-th component and replace the $d$-th component with $i$.
    \item For every additional component which is equal to $i$, we replace the component with $0$ and pull it to the front of the left area vector.
    \item Let $\Lambda$ be this new left area vector and proceed inductively on $i$.
\end{enumerate}
It is clear that this algorithm maximises the size of our staircase shape.

\begin{example}
    As an example, suppose $\laf_{\nu} = (0, 2, 2, 2, 4)$.
    Then we have the following:
    \begin{enumerate}
        \item[$i = 1$:] The first integer in the left area vector which is greater than or equal to $1$ is the $2$ in the $2$nd component.
            We replace the $2$ with a $1$ to get the left area vector $\Lambda = (0, 1, 2, 2, 4)$.
            Since there are no additional components equal to $1$, we proceed inductively.
        \item[$i = 2$:] The first integer in the partition $(0, 1, 2, 2, 4)$ which is greater than or equal to $2$ is the $2$ in the $3$rd component.
            We replace the $2$ with a $2$ to get the partition $(0, 1, 2, 2, 4)$ (in other words, nothing changes).
            Since there is also a $2$ in the fourth component, we replace the fourth component with a $0$ and pull it to the front, giving us the left area vector $\Lambda = (0, 0, 1, 2, 4)$.
        \item[$i = 3$:] The first integer in the left area vector which is greater than or equal to $3$ is the $4$ in the fifth component.
            We replace the $4$ with a $3$ to get the partition $(0, 0, 1, 2, 3)$.
            Since there are no additional components equal to $3$, we continue.
        \item[$i = 4$:] No integers are greater than or equal to $4$ and so we stop.
    \end{enumerate}
    Therefore the maximal staircase shape $\nu$-Dyck path has size $3$ and is the $\nu$-Dyck path $D = N^2(EN)^3E$.
\end{example}

\section{Two subposets of \texorpdfstring{$\nu$}{v}-Tamari lattices}
\label{sec:subposets}
In this section we investigate more deeply the construction of the subposets $\Toutnu$ and $\Tinnu$.
In particular, we study how to calculate the maximal out-degree (in-degree) and whether a particular element of $\mcD_\nu$ has maximal out-degree (in-degree).

\subsection{Out-degree}
We begin by studying the subposet of $\mbbT_\nu$ whose elements have maximal out-degree.
Let $D$ be a $\nu$-Dyck path from $(0, 0)$ to $(s_E, s_N)$ and recall that $\sigma_\nu$ is the size of the maximal staircase shape $\nu$-Dyck path.
We characterise the elements in $\Doutnu$ by whether or not there are $\sigma_\nu$ number of $r_i$ which are proceeded by an east step.
\begin{lemma}
    \label{lem:maxout-size}
    Let $\nu$ be a path and let $\sigma_\nu$ be the size of the maximal staircase shape $\nu$-Dyck path.
    Then $\sigma_\nu = \maxoutnu$.
\end{lemma}
\begin{proof}
    Suppose there exists a $\nu$-Dyck path $D$ which has out-degree larger than $\sigma_\nu$.
    We produce a staircase shape from $D$ in the following way.
    For every two adjacent east edges, we move one of them to the end of $D$ until only one east edge remains in each row.
    Similarly, for every two adjacent north edges, we move one of them to the beginning of $D$ until only one north edge remains in each column.
    These operations don't change $\outdeg(D)$ since we keep the same number of $r_i$ which are proceeded by an east step and followed by a north step.
    This gives us a staircase shape with out-degree $\outdeg(D) > \sigma_\nu$ which means there is a staircase shape of higher size contradicting the fact that $\sigma_\nu$ is the size of the \emph{maximal} staircase shape $\nu$-Dyck path.
\end{proof}

\begin{proposition}
    \label{prop:out-iff-east}
    Let $D$ be a $\nu$-Dyck path from $(0, 0)$ to $(s_E, s_N)$ and let $\sigma_\nu$ be the size of the maximal staircase shape $\nu$-Dyck path.
    Then $\outdeg(D) = \sigma_\nu$ if and only if $\laf_D$ has $\sigma_\nu$ number of $r_i$ which are proceeded by an east step.
\end{proposition}
\begin{proof}
    This is a direct consequence of the definition of the $\nu$-Tamari order.
\end{proof}

For $m$-Dyck paths of height $n$, we have the following corollary.
\begin{corollary}
    \label{cor:outdeg-n}
    Let $D$ be a $m$-Dyck path of height $n$.
    Then $\outdeg(D) = n-1$ if and only if $\laf_D$ is a left area vector of $n$ distinct numbers.
\end{corollary}
\begin{proof}
    It suffices to show that there is a size $n-1$ staircase shape in the top left corner of $\nu = (NE^m)^n$ (the minimal $m$-Dyck path of height $n$).
    Since the partition $\lambda_\nu$ is given by $(0, m, 2m, \ldots, m(n-2), m(n-1))$, the staircase algorithm will end in $(0, 1, 2, \ldots, n-2, n-1)$.
    Therefore, $n-1$ is the maximal staircase size.
\end{proof}

In \autoref{ss:out-degree-poset} we will show that $\Toutnm$ is poset isomorphic to $\mbbT_{n, m-1}$

\subsection{In-degree}
In this section we aim to have a similar characterisation for $\Dinnu$.
For this it will be practical to study whether or not a particular $\nu$-Dyck path has a maximal number of in-edges.
In particular, we must first figure out what \emph{is} the maximal number of in-edges, \ie what is the value of $\maxinnu$ for an arbitrary $\nu$.

Recall that a $\nu$-Dyck path $D$ is a path weakly above $\nu$ from $(0,0)$ to $(s_E, s_N)$.
By \autoref{lem:down-cover}, the only time we can go down in the $\nu$-Tamari order is for an index $i$ such that $t_i = h_i$ and $\horiz{r_i} \neq 0$.

Suppose that $\nu$ is a path from $(0, 0)$ to $(s_E, s_N)$ and $D$ is the maximal element in the $\nu$-Tamari order, \ie $D = N^{s_N} E^{s_E}$.
We start with the simple case in which $\nu$ is a path such that its left area vector $\laf_\nu$ has no repeating entries.
Since $\laf_\nu$ has no repeating entries, the horizontal distance vector from $(0, 0)$ to $(0, s_N)$ is strictly increasing and from $(0, s_N)$ to $(s_E, s_N)$ is strictly decreasing.
Therefore, no number exists more than twice implying that $t_i = h_i$ for every $i$.
In other words, $\maxinnu$ is equal to either $s_N-1$ or $s_N$ depending on if the first entry of $\laf_\nu$ is equal to $0$ or not.
It is clear that this is maximal.

The case where $\laf_\nu$ has repeating entries is a bit more complicated, but still describable.
We first go through an example to get some intuition.

\begin{example}
    Let $\nu = EENN$ and $D$ be the maximal $\nu$-Dyck path, \ie $D = NNEE$.
    This is given in the following diagram whose points on $D$ are labelled by their horizontal distances and where we have labelled the right hand points.
    
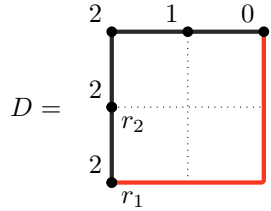
\begin{center}
        \begin{tikzpicture}[scale=1]
            \node at (-1,1) {$D = $};
            \draw[dotted] (0, 0) grid (2, 2);
            \draw[rounded corners=1, color=BrightRed, ultra thick] (0, 0) -- (2, 0) -- (2, 2);
            \draw[rounded corners=1, color=Black, ultra thick] (0,0) -- (0,2) -- (2,2);
        \draw (0,0) circle[radius=1pt];
        \fill (0,0) circle[radius=2pt] node[above left] {$2$};
        \draw (0,1) circle[radius=1pt];
        \fill (0,1) circle[radius=2pt] node[above left] {$2$};
        \draw (0,2) circle[radius=1pt];
        \fill (0,2) circle[radius=2pt] node[above left] {$2$};
        \draw (1,2) circle[radius=1pt];
        \fill (1,2) circle[radius=2pt] node[above left] {$1$};
        \draw (2,2) circle[radius=1pt];
        \fill (2,2) circle[radius=2pt] node[above left] {$0$};
            \fill (0,0) circle[radius=2pt] node[below right] {$r_1$};
            \fill (0,1) circle[radius=2pt] node[below right] {$r_2$};
        \end{tikzpicture}
    \end{center}
    First, let us find the touch points and hit points of $D$.
    We have
    \[
        t_1^D = (0,1)\quad h_1^D = t_2^D = h_2^D = (0,2).
    \]
    Since $t_1^D \neq h_1^D$ then $\Tdown{D}{1}$ doesn't exist, but since $t_2^D = h_2^D$ and $h_2^D$ is followed by an east step, $\Tdown{D}{2}$ does exist.
    Therefore we only have one path going down, which is not the maximal in-degree in the Hasse diagram.
    To get around this, we need to get rid of that final $2$, which we do by creating a staircase shape.
    In our example, what this means is we push the top row (and everything above it) over to the right by one.
    
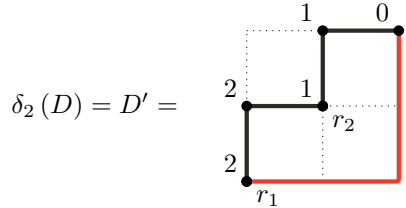
\begin{center}
        \begin{tikzpicture}[scale=1]
            \node at (-2,1) {$\Tdown{D}{2} = D' = $};
            \draw[dotted] (0, 0) grid (2, 2);
            \draw[rounded corners=1, color=BrightRed, ultra thick] (0, 0) -- (2, 0) -- (2, 2);
            \draw[rounded corners=1, color=Black, ultra thick] (0,0) -- (0,1) -- (1,1) -- (1, 2) -- (2,2);
        \draw (0,0) circle[radius=1pt];
        \fill (0,0) circle[radius=2pt] node[above left] {$2$};
        \draw (0,1) circle[radius=1pt];
        \fill (0,1) circle[radius=2pt] node[above left] {$2$};
        \draw (1,1) circle[radius=1pt];
        \fill (1,1) circle[radius=2pt] node[above left] {$1$};
        \draw (1,2) circle[radius=1pt];
        \fill (1,2) circle[radius=2pt] node[above left] {$1$};
        \draw (2,2) circle[radius=1pt];
        \fill (2,2) circle[radius=2pt] node[above left] {$0$};
            \fill (0,0) circle[radius=2pt] node[below right] {$r_1$};
            \fill (1,1) circle[radius=2pt] node[below right] {$r_2$};
        \end{tikzpicture}
    \end{center}
    Now we have $t_1^{D'} = h_1^{D'} = (0,1)$ and $t_2^{D'} = h_2^{D'} = (1,2)$ giving us two ways to go down in the $\nu$-Tamari order.
    This turns out to be maximal.

    It is worthwhile to describe exactly why we can't go down in the first row.
    Suppose that we tried to go down in the first row.
    We would shift the $E$ after $h_1^{D}$ to before $r_1^{D}$ giving us the following $\nu$-Dyck path:
    
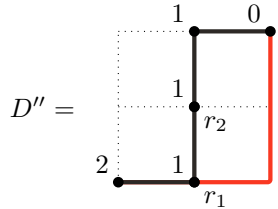
\begin{center}
        \begin{tikzpicture}[scale=1]
            \node at (-1,1) {$D'' = $};
            \draw[dotted] (0, 0) grid (2, 2);
            \draw[rounded corners=1, color=BrightRed, ultra thick] (0, 0) -- (2, 0) -- (2, 2);
            \draw[rounded corners=1, color=Black, ultra thick] (0,0) -- (1,0) -- (1, 2) -- (2,2);
        \draw (0,0) circle[radius=1pt];
        \fill (0,0) circle[radius=2pt] node[above left] {$2$};
        \draw (1,0) circle[radius=1pt];
        \fill (1,0) circle[radius=2pt] node[above left] {$1$};
        \draw (1,1) circle[radius=1pt];
        \fill (1,1) circle[radius=2pt] node[above left] {$1$};
        \draw (1,2) circle[radius=1pt];
        \fill (1,2) circle[radius=2pt] node[above left] {$1$};
        \draw (2,2) circle[radius=1pt];
        \fill (2,2) circle[radius=2pt] node[above left] {$0$};
            \fill (1,0) circle[radius=2pt] node[below right] {$r_1$};
            \fill (1,1) circle[radius=2pt] node[below right] {$r_2$};
        \end{tikzpicture}
    \end{center}
    But now, we only have one way to go up: at the first right hand point.
    Trying to go up at $r_1^{D''}$ gives us the following $\nu$-Dyck path:
    \begin{center}
        \begin{tikzpicture}[scale=1]
            \node at (-1,1) {$D' = $};
            \draw[dotted] (0, 0) grid (2, 2);
            \draw[rounded corners=1, color=BrightRed, ultra thick] (0, 0) -- (2, 0) -- (2, 2);
            \draw[rounded corners=1, color=Black, ultra thick] (0,0) -- (0,1) -- (1,1) -- (1, 2) -- (2,2);
        \end{tikzpicture}
    \end{center}
    which is not our original $\nu$-Dyck path, implying we don't have a downward covering relation here, \ie $D'' \not\cover_T D$.
\end{example}

In order to find a $\nu$-Dyck path with maximal in-degree, we follow a similar approach as above.
We increase the number of times touch points and hit points coincide by creating staircase-like shapes whenever $\lambda_\nu$ has repeated components.
This pushing of rows to make staircases will maximise the number of times hit and touch points coincide.
We encode this in the following algorithm.

\textbf{Dyck path algorithm:}
Start with the maximal $\nu$-Dyck path $D = N^{s_N} E^{s_E}$.
\begin{enumerate}
    \item Let $i\in [s_N]$ be maximal such that $t_i^D \neq h_i^D$ and $\horiz{r_i^D} \neq 0$.
        If no such $i$ exists, we are done.
        Since $t_i^D$ is not followed by an east step, it must be followed by a north step meaning that $t_i^D = r_j^D$ for some $j > i$.
        As $i$ is maximal, this implies $t_j^D = h_j^D$ and $\horiz{r_j^D} = \horiz{r_i^D} \neq 0$.
    \item Then we go down in the $\nu$-Tamari order at $j$ since $h_i^D$ must be followed by an east step (as $\horiz{h_i^D} \neq 0$).
    \item Let $D$ be the $\nu$-Dyck path $\Tdown{D}{j}$ obtained from going down and repeat previous steps.
\end{enumerate}

We can describe the Dyck path algorithm using right area vectors as well.

\textbf{Area algorithm:}
Start with the maximal $\nu$-Dyck path $D = N^{s_N} E^{s_E}$.
\begin{enumerate}
    \item Let $i \in [s_N]$ be maximal such that there exists $j$ where $(\raf_D)_j = (\raf_D)_i \neq 0$ and for all $t$ such that $j > t > i$, then $(\raf_D)_t > (\raf_D)_i$.
        If no such $i$ exists, we are done.
        Let $h > j$ then be the minimal index such that $(\raf_D)_h < (\raf_D)_i$. If no such $h$ exists let $h = s_N + 1$.
    \item Then let $\raf = \raf_D - (0, \ldots, 0, 1, \ldots, 1, 0, \ldots 0)$ where there is a $1$ in all components between $h - 1$ and $j$ inclusively.
    \item Let $D$ be the $\nu$-Dyck path whose area vector is equal to $\raf$ and repeat previous steps.
\end{enumerate}

We say that an \defn{$i$ satisfies the conditions of the area algorithm} if $i$ is maximal such that there exists $j$ where $(\raf_D)_j = (\raf_D)_i \neq 0$ and for all $t$ such that $j > t > i$, then $(\raf_D)_t > (\raf_D)_i$.
We define satisfying the conditions of the Dyck path algorithm similarly. 

These two algorithms produce the same $\nu$-Dyck path.
The main difference is that the area algorithm forgets about the paths themselves and operates on vectors directly.
We first prove this before giving an example of the algorithm(s).
\begin{lemma}
    \label{lem:same-algo}
    Given a path $\nu$, both the Dyck path algorithm and the  area algorithm produce the same $\nu$-Dyck path.
\end{lemma}
\begin{proof}
    In both algorithms we start with $D = N^{s_N}E^{s_E}$.

    Recall that $\horiz{r_i}$ is the value of the $i$-th component of $\raf_D$.
    This means that for the Dyck path algorithm, finding an $i$ that is maximal in $D$ such that $t_i \neq h_i$ and $\horiz{r_i} \neq 0$ is equivalent to saying that there exists a $j > i$ such that $\horiz{r_j} = \horiz{t_i}$ with the horizontal distances of all $r_t$ between $r_j$ and $r_i$ being greater.
    This is equivalent to saying $i$ is maximal in $[s_N]$ such that there exists a $j > i$ where $(\raf_D)_j = (\raf_D)_i \neq 0$ and for all $t$ where $j > t > i$, then $(\raf_D)_t > (\raf_D)_i$.
    The point $h_i = h_j$ in the Dyck path algorithm, and the fact that it must be followed by an east step is equivalent to finding a minimal $h > j$ (the row on which $h_i$ lies) such that $(\raf_D)_h < (\raf_D)_i$.
    The minimality condition comes from the fact that this $h_i$ is a hit point for $r_i$.

    Finally, going down in the $\nu$-Tamari order takes the $E$ after $h_j (=h_i)$ and moves it before $r_j$.
    This pushes the path between $h_j$ and $r_j$ over to the east by one step.
    In other words, the area vector decreases by one for the $j$-th row and every row up until the $h$-th row.
    In other words, we decrease $\raf_D$ in all components between $j$ and $h-1$ inclusively as stated in the area algorithm as desired.
\end{proof}

Note that we are not going down arbitrary paths in the $\nu$-Tamari order.
We are only going down for particular right hand points which have touch points that do not coincide with hit points.
Let us do an example with a couple of repeated components to understand the two algorithms better.

\begin{example}
    Suppose $\nu = EENNENN$.
    We start with the maximal $\nu$-Dyck path $NNNNEEE$.
    
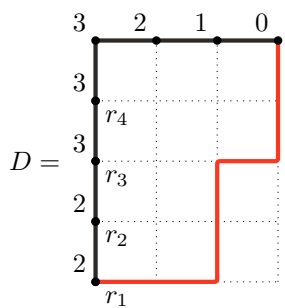
\begin{center}
        \begin{tikzpicture}[scale=0.8]
            \node at (-1, 2) {$D = $};
            \draw[dotted] (0, 0) grid (3, 4);
            \draw[rounded corners=1, color=BrightRed, ultra thick] (0, 0) -- (2, 0) -- (2, 2) -- (3,2) -- (3,4);
            \draw[rounded corners=1, color=Black, ultra thick] (0,0) -- (0,4) -- (3,4);
        \draw (0,0) circle[radius=1pt];
        \fill (0,0) circle[radius=2pt] node[above left] {$2$};
        \draw (0,1) circle[radius=1pt];
        \fill (0,1) circle[radius=2pt] node[above left] {$2$};
        \draw (0,2) circle[radius=1pt];
        \fill (0,2) circle[radius=2pt] node[above left] {$3$};
        \draw (0,3) circle[radius=1pt];
        \fill (0,3) circle[radius=2pt] node[above left] {$3$};
        \draw (0,4) circle[radius=1pt];
        \fill (0,4) circle[radius=2pt] node[above left] {$3$};
        \draw (1,4) circle[radius=1pt];
        \fill (1,4) circle[radius=2pt] node[above left] {$2$};
        \draw (2,4) circle[radius=1pt];
        \fill (2,4) circle[radius=2pt] node[above left] {$1$};
        \draw (3,4) circle[radius=1pt];
        \fill (3,4) circle[radius=2pt] node[above left] {$0$};
            \fill (0,0) circle[radius=2pt] node[below right] {$r_1$};
            \fill (0,1) circle[radius=2pt] node[below right] {$r_2$};
            \fill (0,2) circle[radius=2pt] node[below right] {$r_3$};
            \fill (0,3) circle[radius=2pt] node[below right] {$r_4$};
        \end{tikzpicture}
    \end{center}
    We have labelled $D$ by the horizontal distance vector $\horizDist$.
    Here we have
    \[
        \laf_D = (0, 0, 0, 0) \quad \text{and}\quad \raf_D = (2, 2, 3, 3)
    \]
    Note that there are only two $\nu$-Dyck paths weakly below $D$ as we can only go down in the $2$nd and $4$th rows.
    To increase this, we use the algorithms.

    In terms of the Dyck path algorithm, we first let $i = 3$ as $i$ is the maximal $i$ in which $t_i^D \neq h_i^D$ and $\horiz{r_i^D} \neq 0$.
    So we find a $j > i$ with the same horizontal distance such that $t_j^D = h_j^D$ per the algorithm and push it over one.
    Here $j = 4$ and so we push the top row over by one giving us the $\nu$-Dyck path $D'$ as shown next.
    
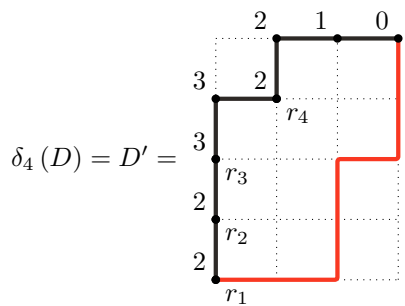
\begin{center}
        \begin{tikzpicture}[scale=0.8]
            \node at (-2, 2) {$\Tdown{D}{4} = D' = $};
            \draw[dotted] (0, 0) grid (3, 4);
            \draw[rounded corners=1, color=BrightRed, ultra thick] (0, 0) -- (2, 0) -- (2, 2) -- (3,2) -- (3,4);
            \draw[rounded corners=1, color=Black, ultra thick] (0,0) -- (0,3) -- (1,3) -- (1, 4) -- (3,4);
        \draw (0,0) circle[radius=1pt];
        \fill (0,0) circle[radius=2pt] node[above left] {$2$};
        \draw (0,1) circle[radius=1pt];
        \fill (0,1) circle[radius=2pt] node[above left] {$2$};
        \draw (0,2) circle[radius=1pt];
        \fill (0,2) circle[radius=2pt] node[above left] {$3$};
        \draw (0,3) circle[radius=1pt];
        \fill (0,3) circle[radius=2pt] node[above left] {$3$};
        \draw (1,3) circle[radius=1pt];
        \fill (1,3) circle[radius=2pt] node[above left] {$2$};
        \draw (1,4) circle[radius=1pt];
        \fill (1,4) circle[radius=2pt] node[above left] {$2$};
        \draw (2,4) circle[radius=1pt];
        \fill (2,4) circle[radius=2pt] node[above left] {$1$};
        \draw (3,4) circle[radius=1pt];
        \fill (3,4) circle[radius=2pt] node[above left] {$0$};
            \fill (0,0) circle[radius=2pt] node[below right] {$r_1$};
            \fill (0,1) circle[radius=2pt] node[below right] {$r_2$};
            \fill (0,2) circle[radius=2pt] node[below right] {$r_3$};
            \fill (1,3) circle[radius=2pt] node[below right] {$r_4$};
        \end{tikzpicture}
    \end{center}
    In terms of the area algorithm, we have that the last two components match, so we have $i = 3$, $j = 4$ and $h = 5\; (=s_N + 1)$.
    Therefore we have $\raf = (2, 2, 3, 3)- (0, 0, 0, 1) = (2, 2, 3, 2)$ which is precisely $\raf_{D'}$.
    We now have
    \[
        \laf_{D'} = (1, 0, 0, 0) \quad \text{and}\quad \raf_{D'} = (2, 2, 3, 2)
    \]

    The $\nu$-Dyck path $D'$ again only has two $\nu$-Dyck paths weakly below it as we can only go down in the $3$rd and $4$th rows.
    From here, notice we have another $i$ for which $t_i^{D'} \neq h_i^{D'}$ and $\horiz{r_i^{D'}} \neq 0$ at $i = 2$.
    Therefore, $j = 4$ and so we push the fourth row over giving us the following $\nu$-Dyck path.
    
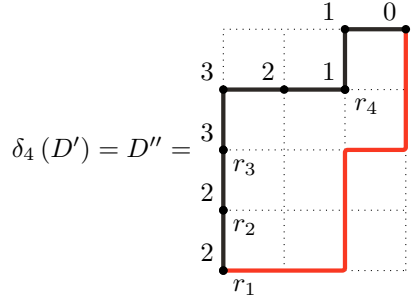
\begin{center}
        \begin{tikzpicture}[scale=0.8]
            \node at (-2, 2) {$\Tdown{D'}{4} = D'' = $};
            \draw[dotted] (0, 0) grid (3, 4);
            \draw[rounded corners=1, color=BrightRed, ultra thick] (0, 0) -- (2, 0) -- (2, 2) -- (3,2) -- (3,4);
            \draw[rounded corners=1, color=Black, ultra thick] (0,0) -- (0,3) -- (2,3) -- (2, 4) -- (3,4);
        \draw (0,0) circle[radius=1pt];
        \fill (0,0) circle[radius=2pt] node[above left] {$2$};
        \draw (0,1) circle[radius=1pt];
        \fill (0,1) circle[radius=2pt] node[above left] {$2$};
        \draw (0,2) circle[radius=1pt];
        \fill (0,2) circle[radius=2pt] node[above left] {$3$};
        \draw (0,3) circle[radius=1pt];
        \fill (0,3) circle[radius=2pt] node[above left] {$3$};
        \draw (1,3) circle[radius=1pt];
        \fill (1,3) circle[radius=2pt] node[above left] {$2$};
        \draw (2,3) circle[radius=1pt];
        \fill (2,3) circle[radius=2pt] node[above left] {$1$};
        \draw (2,4) circle[radius=1pt];
        \fill (2,4) circle[radius=2pt] node[above left] {$1$};
        \draw (3,4) circle[radius=1pt];
        \fill (3,4) circle[radius=2pt] node[above left] {$0$};
            \fill (0,0) circle[radius=2pt] node[below right] {$r_1$};
            \fill (0,1) circle[radius=2pt] node[below right] {$r_2$};
            \fill (0,2) circle[radius=2pt] node[below right] {$r_3$};
            \fill (2,3) circle[radius=2pt] node[below right] {$r_4$};
        \end{tikzpicture}
    \end{center}
    In terms of the area algorithm, we have $i = 2$, $j = 4 $ and $h = 5$.
    Therefore $\raf = (2, 2, 3, 2) - (0, 0, 0, 1) = (2, 2, 3, 1)$ and we have
    \[
        \laf_{D''} = (2, 0, 0, 0) \quad \text{and}\quad \raf_{D''} = (2, 2, 3, 1)
    \]

    Again, we're not done as $h_1^{D''} \neq t_1^{D''}$ and $\horiz{r_1^{D''}} \neq 0$.
    Letting $j = 2$ we push the second row over to get the following $\nu$-Dyck path.
    
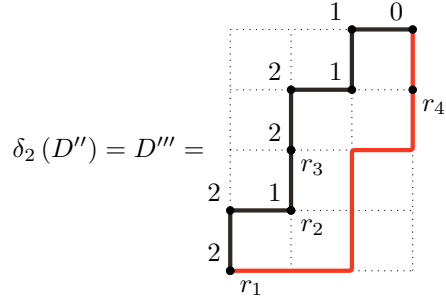
\begin{center}
        \begin{tikzpicture}[scale=0.8]
            \node at (-2, 2) {$\Tdown{D''}{2} = D''' = $};
            \draw[dotted] (0, 0) grid (3, 4);
            \draw[rounded corners=1, color=BrightRed, ultra thick] (0, 0) -- (2, 0) -- (2, 2) -- (3,2) -- (3,4);
            \draw[rounded corners=1, color=Black, ultra thick] (0,0) -- (0,1) -- (1,1) -- (1, 3) -- (2,3) -- (2,4) -- (3, 4);
        \draw (0,0) circle[radius=1pt];
        \fill (0,0) circle[radius=2pt] node[above left] {$2$};
        \draw (0,1) circle[radius=1pt];
        \fill (0,1) circle[radius=2pt] node[above left] {$2$};
        \draw (1,1) circle[radius=1pt];
        \fill (1,1) circle[radius=2pt] node[above left] {$1$};
        \draw (1,2) circle[radius=1pt];
        \fill (1,2) circle[radius=2pt] node[above left] {$2$};
        \draw (1,3) circle[radius=1pt];
        \fill (1,3) circle[radius=2pt] node[above left] {$2$};
        \draw (2,3) circle[radius=1pt];
        \fill (2,3) circle[radius=2pt] node[above left] {$1$};
        \draw (2,4) circle[radius=1pt];
        \fill (2,4) circle[radius=2pt] node[above left] {$1$};
        \draw (3,4) circle[radius=1pt];
        \fill (3,4) circle[radius=2pt] node[above left] {$0$};
            \fill (0,0) circle[radius=2pt] node[below right] {$r_1$};
            \fill (1,1) circle[radius=2pt] node[below right] {$r_2$};
            \fill (1,2) circle[radius=2pt] node[below right] {$r_3$};
            \fill (3,3) circle[radius=2pt] node[below right] {$r_4$};
        \end{tikzpicture}
    \end{center}
    In terms of the area algorithm we have $i = 1 $, $j = 2$ and $h = 4$.
    Therefore $\raf = (2, 2, 3, 1) - (0, 1, 1, 0) = (2, 1, 2, 1)$ and we have
    \[
        \laf_{D'''} = (2, 1, 1, 0) \quad \text{and}\quad \raf_{D'''} = (2, 1, 2, 1)
    \]

    Again, we're not done as $h_2^{D'''} \neq t_2^{D'''}$ and $\horiz{r_2^{D'''}} \neq 0$.
    Letting $j = 4$ we push the second row over to get the following $\nu$-Dyck path.
    
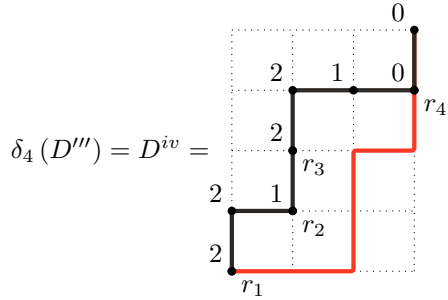
\begin{center}
        \begin{tikzpicture}[scale=0.8]
            \node at (-2, 2) {$\Tdown{D'''}{4} = D^{iv} = $};
            \draw[dotted] (0, 0) grid (3, 4);
            \draw[rounded corners=1, color=BrightRed, ultra thick] (0, 0) -- (2, 0) -- (2, 2) -- (3,2) -- (3,4);
            \draw[rounded corners=1, color=Black, ultra thick] (0,0) -- (0,1) -- (1,1) -- (1, 3) -- (3,3) -- (3, 4);
        \draw (0,0) circle[radius=1pt];
        \fill (0,0) circle[radius=2pt] node[above left] {$2$};
        \draw (0,1) circle[radius=1pt];
        \fill (0,1) circle[radius=2pt] node[above left] {$2$};
        \draw (1,1) circle[radius=1pt];
        \fill (1,1) circle[radius=2pt] node[above left] {$1$};
        \draw (1,2) circle[radius=1pt];
        \fill (1,2) circle[radius=2pt] node[above left] {$2$};
        \draw (1,3) circle[radius=1pt];
        \fill (1,3) circle[radius=2pt] node[above left] {$2$};
        \draw (2,3) circle[radius=1pt];
        \fill (2,3) circle[radius=2pt] node[above left] {$1$};
        \draw (3,3) circle[radius=1pt];
        \fill (3,3) circle[radius=2pt] node[above left] {$0$};
        \draw (3,4) circle[radius=1pt];
        \fill (3,4) circle[radius=2pt] node[above left] {$0$};
            \fill (0,0) circle[radius=2pt] node[below right] {$r_1$};
            \fill (1,1) circle[radius=2pt] node[below right] {$r_2$};
            \fill (1,2) circle[radius=2pt] node[below right] {$r_3$};
            \fill (3,3) circle[radius=2pt] node[below right] {$r_4$};
        \end{tikzpicture}
    \end{center}
    In terms of the area algorithm we have $i = 2$, $j = 4$ and $h = 5$.
    Therefore $\raf = (2, 1, 2, 1) - (0, 0, 0, 1) = (2, 1, 2, 0)$ and we have
    \[
        \laf_{D^{iv}} = (3, 1, 1, 0) \quad \text{and}\quad \raf_{D^{iv}} = (2,1,2,0)
    \]

    At this point, notice that in both algorithms we can no longer continue and therefore we are done.

    This $\nu$-Dyck path has exactly three $\nu$-Dyck paths weakly below it which we claim gives the maximal number for the in-degree.
    Notice that the previous two paths \emph{also} have three $\nu$-Dyck paths weakly below it.
    The difference between the previous paths and the final one is that in the final path for every $i$ we have $t_i^{D^{iv}} = h_i^{D^{iv}}$ or $\horiz{r_i^{D^{iv}}} = 0$ which is not true in the previous ones.
\end{example}

We must show that this procedure does produce a $\nu$-Dyck path with maximal in-degree.
Before proving this, we show that the components of $\laf_\nu$ which are equal are special in our algorithm.
\begin{lemma}
    \label{lem:algo-east}
    Let $\nu$ be some path from $(0,0)$ to $(s_E, s_N)$.
    In terms of either algorithm, if a value $j$ exists then $(\laf_\nu)_{j} = (\laf_\nu)_{j-1}$.
\end{lemma}
\begin{proof}
    Let $D_0, D_1,\,D_2,\,\cdots,\,D_n$ be the chain of Dyck paths we produce using either algorithm.
    We proceed by (strong/complete) induction on the indices.

    For $D_0 = N^{s_N}E^{s_E}$, then since $D_0$ goes strictly north and then strictly east, the only time a value appears more than twice on the horizontal distance vector is if it appears at least two times in the strictly north portion of $D_0$.
    This only happens when there are two adjacent points with the same horizontal distance.
    Therefore $i = j - 1$ in the area algorithm.
    Since $\raf_{D_0} = \laf_\nu - \laf_{D_0} = \laf_\nu$ and $(\raf_{D_0})_{j-1} = (\raf_{D_0})_i = (\raf_{D_0})_j$ then  $(\laf_\nu)_{j} = (\laf_\nu)_{j-1}$.

    Suppose $D_m$ is our current $\nu$-Dyck path and that there exists an $i$ and a $j > i$ such that $t_i^{D_m} \neq h_i^{D_m}$, $\horiz{r_i^{D_m}} = \horiz{r_j^{D_m}} \neq 0$, and $t_j^{D_m} = h_j^{D_m}$ as required in the Dyck path algorithm.
    If $r_j^{D_m}$ is preceded by a north step, then $r_i^{D_m}$ must be the adjacent point below $r_j^{D_m}$ and, by a similar argument as the $m = 1$ case, we are done.
    Otherwise, $r_j^{D_m}$ is preceded by an east step.
    In particular, $r_j^{D_m}$ is preceded by an east step if at some $D_r$ with $r < m$ the point $r_j^{D_r}$ had an east step placed before it to go to $D_{r+1}$.
    By our induction, this only occurs if $(\laf_\nu)_{j} = (\laf_\nu)_{j-1}$ as desired.
\end{proof}

\begin{corollary}
    \label{cor:east-only-same}
    Let $\nu$ be some path from $(0, 0)$ to $(s_E, s_N)$ and let $D$ be the $\nu$-Dyck path obtained from the Dyck path algorithm.
    If a row $j$ has an east step then $(\laf_\nu)_{j} = (\laf_\nu)_{j-1}$.
\end{corollary}
\begin{proof}
    Direct result of the previous lemma as we only add east steps on a row $j$ when $(\laf_\nu)_{j} = (\laf_\nu)_{j-1}$.
\end{proof}

We say that a left area vector $\laf_D$ is \defn{shifted down} to $\laf_D'$ if we duplicate the last component of $\laf_D$ and remove the first component of $\laf_D$ to get $\laf_D'$.
For example, the left area vector $(0, 2, 3, 4)$ is shifted down to $(2, 3, 4, 4)$.
We say that a $\nu$-Dyck path $D$ is \defn{shifted down} to $D'$ if its left area vector $\laf_D$ is shifted down to the left area vector $\laf_{D'}$ of $D'$.
Intuitively, shifting down refers to decreasing the $y$ coordinate of every point by $1$ and adding a north step after the final north step.
If our $\nu$-Dyck path begins with a north step, then this is the same thing as moving a north step from the beginning of the path to after the final north step.

\newpage
\begin{lemma}
    \label{lem:algo-stair}
    Let $\nu$ be the path $E^aN^b$ whose maximal staircase shape $\nu$-Dyck path has size $\sigma_\nu$.
    Then the area algorithm (or the Dyck path algorithm) terminates at the staircase shape $\nu$-Dyck path $D$ of size $\sigma_\nu - 1$ (shifted down $b-a$ times if  $a<b$).
    In particular,
    \[
        \raf_D = \begin{cases}
            (a, \ldots, a - b+1)& \text{if }a \geq b,\\
            (a, \ldots, 2, 1, 0, \ldots, 0) & \text{if } a < b,
        \end{cases}
    \]
    and
    \[
        \laf_D = \begin{cases}
            (0, 1, 2, \ldots, b-2, b-1) & \text{if } a \geq b,\\
            (0, 1, 2, \ldots, a-1, a, \ldots, a) & \text{if } a < b.
        \end{cases}
    \]
\end{lemma}
\begin{proof}
    Let $\nu = E^aN^b$ and $D_0 = N^bE^a$ be our initial $\nu$-Dyck path.
    Then $\raf_{D_0}$ is given by $(a, \ldots, a)$ with $b$ entries.
    Let $D_i$ be the $i$-th $\nu$-Dyck path obtained through the area algorithm.
    Going through the area algorithm, we get the following chain for $\raf_{D_i}$:
    \begin{align*}
        \raf_{D_0} = (a, \ldots, a, a) &\to (a, \ldots, a, a-1) &( = \raf_{D_1})\\
        &\to (a, \ldots, a, a-1, a-1) &( = \raf_{D_2})\\
        &\to (a, \ldots, a, a-1, a-2) &( = \raf_{D_3})\\
        &\to (a, \ldots, a, a-1, a-1, a-2) &( = \raf_{D_4})\\
        &\to (a, \ldots, a, a-1, a-2, a-2) &( = \raf_{D_5})\\
        &\to (a, \ldots, a, a-1, a-2, a-3) &( = \raf_{D_6})\\
        &\to \cdots
    \end{align*}

    Notice that this is just creating a staircase shape and that this terminates when we have run out of duplicate (nonzero) numbers.
    Let $D$ be the final $\nu$-Dyck path obtained through the area algorithm.
    If $a \geq b$ then $\raf_D = (a, \ldots, a-b+1)$ giving us $b$ unique entries.
    If $a < b$ then $\raf_D = (a, a-1, \ldots, 2, 1, 0, \ldots, 0)$.
    To find $\laf_D$ it suffices to recall that $\raf_D = \laf_\nu - \laf_D$.

    It remains to show that this staircase shape has size $\sigma_\nu - 1$.
    Let $x$ and $y$ be the following
    \[
        x = \begin{cases}
            b - a& \text{if }a \leq b\\
            0 & \text{otherwise}
        \end{cases}
        \qquad
        y = \begin{cases}
            a-b& \text{if }a \geq b\\
            0 & \text{otherwise}
        \end{cases}
    \]
    The maximal staircase shape $\nu$-Dyck path for $\nu=(E^aN^b)$ is $N^x(EN)^{\min(a, b)}E^y$ implying $\sigma_\nu = \min(a, b)$.
    Since $\raf_D$ always begins with $a$, $D$ is a staircase shape $\nu$-Dyck path beginning with a north step.
    As each component of $\raf_D$ is one less than the previous one and as all the components of $\laf_\nu$ are equal then we have $D = (NE)^{\min(a, b)}N^xE^y$.
    
    If $a \geq b$ then $D = (NE)^{b}E^{a-b}$ implying that the size of the staircase is $b - 1 = \sigma_\nu - 1$ as $(NE)^{b}E^{a-b} = N(EN)^{b-1}E^{a-b+1}$.
    If $a < b$ then $D = (NE)^{a}N^{b-a}$.
    This $D$ is identical to $D' = N^{b-a}(NE)^{a}=N^{b-a+1}(EN)^{a-1}E$ shifted down $b-a$ times.
    Therefore, $D$ is the staircase shape $\nu$-Dyck path $D'$ (of size $a - 1 = \sigma_\nu - 1$) shifted down $b-a$ times as desired.
\end{proof}

We now ask ourselves, what happens when we have multiple chains of duplicate entries in $\laf_\nu$.
For this, we start with an example with two chains to help gather some intuition.
\begin{example}
    Let $\nu = EEENNENNN$ and we proceed using the area algorithm.
    As always, we start with $D_0 = NNNNNEEEE$ and proceed by going down in the $\nu$-Tamari order.
    We have labelled each $\nu$-Dyck path with the horizontal distance for each right hand point.

    By \autoref{lem:algo-stair}, we know that since the last three steps in $\nu$ are $N$, we will start by constructing a staircase shape near the top.
    \begin{center}
        {\small
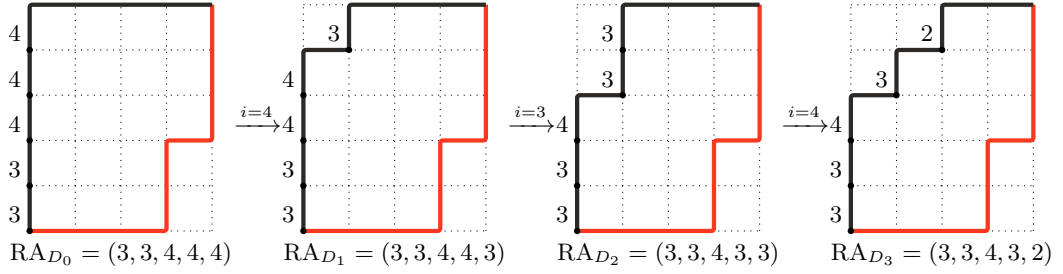
        \begin{tikzpicture}[scale=0.6]
            \begin{scope}[shift={(0,0)}]
                \draw[dotted] (0, 0) grid (4, 5);
                \draw[rounded corners=1, color=BrightRed, ultra thick] (0, 0) -- (3, 0) -- (3, 2) -- (4,2) -- (4,5);
                \draw[rounded corners=1, color=Black, ultra thick] (0,0) -- (0,5) -- (4,5);
                \draw (0,0) circle[radius=1pt];
                \fill (0,0) circle[radius=2pt] node[above left] {$3$};
                \draw (0,1) circle[radius=1pt];
                \fill (0,1) circle[radius=2pt] node[above left] {$3$};
                \draw (0,2) circle[radius=1pt];
                \fill (0,2) circle[radius=2pt] node[above left] {$4$};
                \draw (0,3) circle[radius=1pt];
                \fill (0,3) circle[radius=2pt] node[above left] {$4$};
                \draw (0,4) circle[radius=1pt];
                \fill (0,4) circle[radius=2pt] node[above left] {$4$};
                \node at (2, -0.5) {$\raf_{D_0} = (3, 3, 4, 4, 4)$};
            \end{scope}
            \node at (5, 2.5) {$\xrightarrow{i=4}$};
            \begin{scope}[shift={(6,0)}]
                \draw[dotted] (0, 0) grid (4, 5);
                \draw[rounded corners=1, color=BrightRed, ultra thick] (0, 0) -- (3, 0) -- (3, 2) -- (4,2) -- (4,5);
                \draw[rounded corners=1, color=Black, ultra thick] (0,0) -- (0,4) -- (1,4) -- (1, 5) -- (4,5);
                \draw (0,0) circle[radius=1pt];
                \fill (0,0) circle[radius=2pt] node[above left] {$3$};
                \draw (0,1) circle[radius=1pt];
                \fill (0,1) circle[radius=2pt] node[above left] {$3$};
                \draw (0,2) circle[radius=1pt];
                \fill (0,2) circle[radius=2pt] node[above left] {$4$};
                \draw (0,3) circle[radius=1pt];
                \fill (0,3) circle[radius=2pt] node[above left] {$4$};
                \draw (1,4) circle[radius=1pt];
                \fill (1,4) circle[radius=2pt] node[above left] {$3$};
                \node at (2, -0.5) {$\raf_{D_1} = (3, 3, 4, 4, 3)$};
            \end{scope}
            \node at (11, 2.5) {$\xrightarrow{i=3}$};
            \begin{scope}[shift={(12,0)}]
                \draw[dotted] (0, 0) grid (4, 5);
                \draw[rounded corners=1, color=BrightRed, ultra thick] (0, 0) -- (3, 0) -- (3, 2) -- (4,2) -- (4,5);
                \draw[rounded corners=1, color=Black, ultra thick] (0,0) -- (0,3) -- (1,3) -- (1, 5) -- (4,5);
                \draw (0,0) circle[radius=1pt];
                \fill (0,0) circle[radius=2pt] node[above left] {$3$};
                \draw (0,1) circle[radius=1pt];
                \fill (0,1) circle[radius=2pt] node[above left] {$3$};
                \draw (0,2) circle[radius=1pt];
                \fill (0,2) circle[radius=2pt] node[above left] {$4$};
                \draw (1,3) circle[radius=1pt];
                \fill (1,3) circle[radius=2pt] node[above left] {$3$};
                \draw (1,4) circle[radius=1pt];
                \fill (1,4) circle[radius=2pt] node[above left] {$3$};
                \node at (2, -0.5) {$\raf_{D_2} = (3, 3, 4, 3, 3)$};
            \end{scope}
            \node at (17, 2.5) {$\xrightarrow{i=4}$};
            \begin{scope}[shift={(18,0)}]
                \draw[dotted] (0, 0) grid (4, 5);
                \draw[rounded corners=1, color=BrightRed, ultra thick] (0, 0) -- (3, 0) -- (3, 2) -- (4,2) -- (4,5);
                \draw[rounded corners=1, color=Black, ultra thick] (0,0) -- (0,3) -- (1,3) -- (1, 4) -- (2,4) -- (2,5) -- (4,5);
                \draw (0,0) circle[radius=1pt];
                \fill (0,0) circle[radius=2pt] node[above left] {$3$};
                \draw (0,1) circle[radius=1pt];
                \fill (0,1) circle[radius=2pt] node[above left] {$3$};
                \draw (0,2) circle[radius=1pt];
                \fill (0,2) circle[radius=2pt] node[above left] {$4$};
                \draw (1,3) circle[radius=1pt];
                \fill (1,3) circle[radius=2pt] node[above left] {$3$};
                \draw (2,4) circle[radius=1pt];
                \fill (2,4) circle[radius=2pt] node[above left] {$2$};
                \node at (2, -0.5) {$\raf_{D_3} = (3, 3, 4, 3, 2)$};
            \end{scope}
        \end{tikzpicture}
        }
    \end{center}
    In other words, the algorithm converted $4, 4, 4$ to $4, 3, 2$ to have a staircase shape for the first chain.
    Continuing the algorithm, we have
    \begin{center}
        {\small
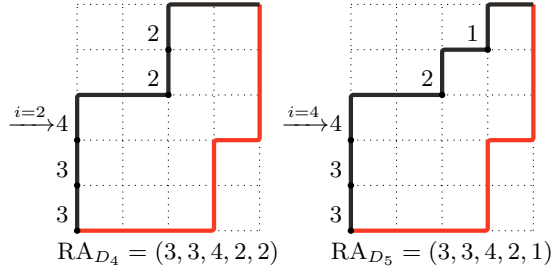
        \begin{tikzpicture}[scale=0.6]
            \node at (-1, 2.5) {$\xrightarrow{i=2}$};
            \begin{scope}[shift={(0,0)}]
                \draw[dotted] (0, 0) grid (4, 5);
                \draw[rounded corners=1, color=BrightRed, ultra thick] (0, 0) -- (3, 0) -- (3, 2) -- (4,2) -- (4,5);
                \draw[rounded corners=1, color=Black, ultra thick] (0,0) -- (0, 3) -- (2, 3) -- (2, 5) -- (4, 5);
                \draw (0,0) circle[radius=1pt];
                \fill (0,0) circle[radius=2pt] node[above left] {$3$};
                \draw (0,1) circle[radius=1pt];
                \fill (0,1) circle[radius=2pt] node[above left] {$3$};
                \draw (0,2) circle[radius=1pt];
                \fill (0,2) circle[radius=2pt] node[above left] {$4$};
                \draw (2,3) circle[radius=1pt];
                \fill (2,3) circle[radius=2pt] node[above left] {$2$};
                \draw (2,4) circle[radius=1pt];
                \fill (2,4) circle[radius=2pt] node[above left] {$2$};
                \node at (2, -0.5) {$\raf_{D_4} = (3, 3, 4, 2, 2)$};
            \end{scope}
            \node at (5, 2.5) {$\xrightarrow{i=4}$};
            \begin{scope}[shift={(6,0)}]
                \draw[dotted] (0, 0) grid (4, 5);
                \draw[rounded corners=1, color=BrightRed, ultra thick] (0, 0) -- (3, 0) -- (3, 2) -- (4,2) -- (4,5);
                \draw[rounded corners=1, color=Black, ultra thick] (0,0) -- (0, 3) -- (2, 3) -- (2, 4) -- (3, 4) --  (3, 5) -- (4, 5);
                \draw (0,0) circle[radius=1pt];
                \fill (0,0) circle[radius=2pt] node[above left] {$3$};
                \draw (0,1) circle[radius=1pt];
                \fill (0,1) circle[radius=2pt] node[above left] {$3$};
                \draw (0,2) circle[radius=1pt];
                \fill (0,2) circle[radius=2pt] node[above left] {$4$};
                \draw (2,3) circle[radius=1pt];
                \fill (2,3) circle[radius=2pt] node[above left] {$2$};
                \draw (3,4) circle[radius=1pt];
                \fill (3,4) circle[radius=2pt] node[above left] {$1$};
                \node at (2, -0.5) {$\raf_{D_5} = (3, 3, 4, 2, 1)$};
            \end{scope}
        \end{tikzpicture}
        }
    \end{center}
    Since the second set of $N$ steps in $\nu$ is such that the top most right hand point has a hit point inside the first staircase, we must grab an east step from this staircase, which then propagates this east step so that it is as if it was pulled from the top row.
    
    Finishing up the algorithm, we have the final three steps
    \begin{center}
        {\small
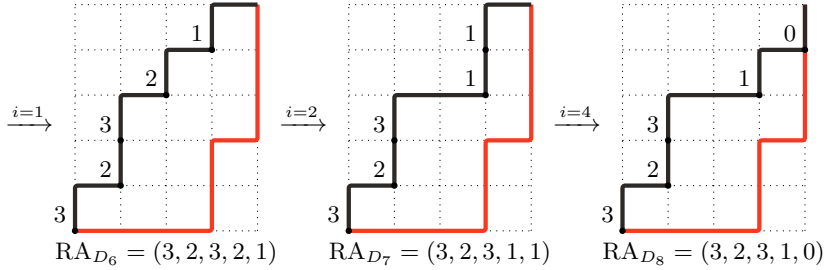
        \begin{tikzpicture}[scale=0.6]
            \node at (-1, 2.5) {$\xrightarrow{i=1}$};
            \begin{scope}[shift={(0,0)}]
                \draw[dotted] (0, 0) grid (4, 5);
                \draw[rounded corners=1, color=BrightRed, ultra thick] (0, 0) -- (3, 0) -- (3, 2) -- (4,2) -- (4,5);
                \draw[rounded corners=1, color=Black, ultra thick] (0,0) -- (0, 1) -- (1, 1) -- (1, 3) -- (2, 3) -- (2, 4) -- (3, 4) --  (3, 5) -- (4, 5);
                \draw (0,0) circle[radius=1pt];
                \fill (0,0) circle[radius=2pt] node[above left] {$3$};
                \draw (1,1) circle[radius=1pt];
                \fill (1,1) circle[radius=2pt] node[above left] {$2$};
                \draw (1,2) circle[radius=1pt];
                \fill (1,2) circle[radius=2pt] node[above left] {$3$};
                \draw (2,3) circle[radius=1pt];
                \fill (2,3) circle[radius=2pt] node[above left] {$2$};
                \draw (3,4) circle[radius=1pt];
                \fill (3,4) circle[radius=2pt] node[above left] {$1$};
                \node at (2, -0.5) {$\raf_{D_6} = (3, 2, 3, 2, 1)$};
            \end{scope}
            \node at (5, 2.5) {$\xrightarrow{i=2}$};
            \begin{scope}[shift={(6,0)}]
                \draw[dotted] (0, 0) grid (4, 5);
                \draw[rounded corners=1, color=BrightRed, ultra thick] (0, 0) -- (3, 0) -- (3, 2) -- (4,2) -- (4,5);
                \draw[rounded corners=1, color=Black, ultra thick] (0,0) -- (0, 1) -- (1, 1) -- (1, 3) -- (3, 3) -- (3, 5) -- (4, 5);
                \draw (0,0) circle[radius=1pt];
                \fill (0,0) circle[radius=2pt] node[above left] {$3$};
                \draw (1,1) circle[radius=1pt];
                \fill (1,1) circle[radius=2pt] node[above left] {$2$};
                \draw (1,2) circle[radius=1pt];
                \fill (1,2) circle[radius=2pt] node[above left] {$3$};
                \draw (3,3) circle[radius=1pt];
                \fill (3,3) circle[radius=2pt] node[above left] {$1$};
                \draw (3,4) circle[radius=1pt];
                \fill (3,4) circle[radius=2pt] node[above left] {$1$};
                \node at (2, -0.5) {$\raf_{D_7} = (3, 2, 3, 1, 1)$};
            \end{scope}
            \node at (11, 2.5) {$\xrightarrow{i=4}$};
            \begin{scope}[shift={(12,0)}]
                \draw[dotted] (0, 0) grid (4, 5);
                \draw[rounded corners=1, color=BrightRed, ultra thick] (0, 0) -- (3, 0) -- (3, 2) -- (4,2) -- (4,5);
                \draw[rounded corners=1, color=Black, ultra thick] (0,0) -- (0, 1) -- (1, 1) -- (1, 3) -- (3, 3) -- (3, 4) -- (4,4) -- (4, 5);
                \draw (0,0) circle[radius=1pt];
                \fill (0,0) circle[radius=2pt] node[above left] {$3$};
                \draw (1,1) circle[radius=1pt];
                \fill (1,1) circle[radius=2pt] node[above left] {$2$};
                \draw (1,2) circle[radius=1pt];
                \fill (1,2) circle[radius=2pt] node[above left] {$3$};
                \draw (3,3) circle[radius=1pt];
                \fill (3,3) circle[radius=2pt] node[above left] {$1$};
                \draw (4,4) circle[radius=1pt];
                \fill (4,4) circle[radius=2pt] node[above left] {$0$};
                \node at (2, -0.5) {$\raf_{D_8} = (3, 2, 3, 1, 0)$};
            \end{scope}
        \end{tikzpicture}
        }
    \end{center}
\end{example}

For ease of notation, for the rest of this section we let $D^{(k)}$ denote the \emph{first} $\nu$-Dyck path obtained from the area (or Dyck path) algorithm such that for all $j > k$ then either $(\raf_{D^{(k)}})_j \neq (\raf_{D^{(k)}})_k$ or $(\raf_{D^{(k)}})_j = 0$ or there exists $t$ such that $j > t > k$ for which $(\raf_{D^{(k)}})_t < (\raf_{D^{(k)}})_k$.
In other words, $D^{(k)}$ is the first $\nu$-Dyck path obtained in the algorithm in which every component (weakly) greater than the $k$-th row doesn't satisfy the conditions of the area algorithm.
If $k = s_N$  then $D^{(s_N)}$ is our starting $\nu$-Dyck path and by $D^{(1)}$ we mean the $\nu$-Dyck path obtained at the end of the area algorithm in which no component satisfies the conditions.
In our example above, we can read the $D^{(k)}$ from the diagram where $D^{(k)}$ is the $\nu$-Dyck path right before the first arrow labelled with $i = a$ where $a<k$:
\[
    D^{(5)} = D_0,\, D^{(4)} = D_1,\, D^{(3)} = D_3,\, D^{(2)} = D_5,\, \text{and } D^{(1)} = D_8.
\]

We next describe a formula on the chain of $D^{(k)}$ depending on our choice of $\nu$.
Before giving the formula, we give a technical lemma which will help in future proofs.
\begin{lemma}
    \label{lem:ti-is-hi-on-top}
    Let $\nu$ be a path and let $D^{(k+1)}$ be the first $\nu$-Dyck path obtained from the area (or Dyck path) algorithm where the conditions of the algorithm are not satisfied for all indices weakly greater than $k+1$.
    If $(\laf_\nu)_k > 0$ and $t_k^{D^{(k+1)}} = h_k^{D^{(k+1)}}$ then $h_k^{D^{(k+1)}}$ must be on the top row.
\end{lemma}
\begin{proof}
    We show by contradiction that $h_k^{D^{(k+1)}}$ must be on the top row.
    Recall that $h_k^{D^{(k+1)}}$ must have an east step after it by definition since the horizontal distance is non-zero, \ie $(\laf_\nu)_{k} > 0$.
    If $h_k^{D^{(k+1)}}$ is not on the top row, then at some point earlier in the algorithm an east step was added after $h_k^{D^{(k+1)}}$.
    Therefore, at some iteration (say going from $D$ to $D'$)  in the Dyck path algorithm there was $i$ and $j$ such that $t_i^D \neq h_i^D = h_j^D = t_j^D$ with the row of $r_j^D$ being the same row as $h_k^{D^{(k+1)}}$ and where $\horiz{t_i^D} = (\laf_\nu)_k$.
    Since $\horiz{t_i^D} = (\laf_\nu)_k$, therefore $h_k^{D^{(k+1)}} = t_i^D$.
    Furthermore, by construction of $D^{(k+1)}$ and since this iteration must have happened before acquiring $D^{(k+1)}$ then $i > k+1$.
But then $r_i^D$ is a point on the $\nu$-path $D^{(k+1)}$ and, in particular, $r_i^D = r_i^{D^{(k+1)}}$ such that $\horiz{r_{i}^{D^{(k+1)}}} = \horiz{r_{k}^{D^{(k+1)}}}$.
    Moreover, $i > k+1$ implies that $t_k^{D^{(k+1)}} \neq h_k^{D^{(k+1)}}$ as $r_i^{D^{(k+1)}}$ is a point with the same horizontal distance strictly between $r_k^{D^{(k+1)}}$ and $h_k^{D^{(k+1)}}$, a contradiction.
    Therefore $h_k^{D^{(k+1)}}$ is on the top row.
\end{proof}

For ease of notation let $(01)_i$ denote the vector $(0,\, \ldots,\, 0,\,1,\,\ldots,\,1)$ where the first $1$ occurs in the $i$-th entry.
The length of $(01)_i$ is given by the formula it is contained in, and if $i$ is greater than the length then $(01)_i$ is the all zero vector.
For example:
\begin{align*}
    (3,2,1) + (01)_2 &= (3,2,1) + (0,1,1) = (3,3,2)\\
    (2,4,8,1,1,3) + (01)_4 &= (2,4,8,1,1,3) + (0,0,0,1,1,1) = (2,4,8,2,2,4)\\
    (7,2,3,8,1) + (01)_9 &=  (7,2,3,8,1) + (0,0,0,0,0) = (7,2,3,8,1)
\end{align*}

\begin{lemma}
    \label{lem:many-stairs}
    Let $\nu$ be a path and let $D^{(k+1)}$ be the first $\nu$-Dyck path obtained from the area (or Dyck path) algorithm where the conditions of the algorithm are not satisfied for all indices weakly greater than $k+1$.
    Let $t$ be the row containing $t_k^{D^{(k+1)}}$ and let $z$ be the first row (weakly) after the $t$-th row whose right hand point has horizontal distance $0$.
    Then
    \begin{equation}
        \label{eq:go-down-algo}
        \laf_{D^{(k)}} = \laf_{D^{(k+1)}} + (01)_t - (01)_z
    \end{equation}
\end{lemma}
\begin{proof}
    This is easily verified for $k = s_N$.
    Start with $D^{(k+1)}$.

    If $(\laf_\nu)_k = 0$, then the area algorithm conditions aren't satisfied at $k$.
    Therefore $D^{(k)} = D^{(k+1)}$.
    To show that this coincides with our equation, it suffices to show that $z = t$.
    But this is true since $(\laf_\nu)_k = 0$ implies $\horiz{t_k^{D^{(k+1)}}} = 0$ and therefore $z = t$.

    We now assume that $(\laf_\nu)_k > 0$.
    If $t_k^{D^{(k+1)}} = h_k^{D^{(k+1)}}$ then by \autoref{lem:ti-is-hi-on-top} $h_k^{D^{(k+1)}}$ is on the top row and thus $h = t = s_N+1$.
    Thus there is only one other point between $r_k^{D^{(k+1)}}$ and $(s_E, s_N)$ with the same value implying that $D^{(k)} = D^{(k+1)}$.
    We must therefore show (again) that $z = t$.
    Since $h_k^{D^{(k+1)}} = t_k^{D^{(k+1)}}$ is on the top row then $z$ must also be $s_N + 1$ as desired.

    If on the other hand $t_k^{D^{(k+1)}} \neq h_k^{D^{(k+1)}}$, by the Dyck path algorithm, we move the east step after $h_k^{D^{(k+1)}}$ to before $t_k^{D^{(k+1)}}$ to get the $\nu$-Dyck path $D^{(k+1)'}$ whose left area vector is equal to
    \[
        \laf_{D^{(n+1)'}} = \laf_{D^{(n+1)}} + (01)_t - (01)_h.
    \]
    If $h \neq s_N+1$ then we are pulling an east step from a row strictly between $t$ and $s_N + 1$.
    Since $(\laf_\nu)_{h} = (\laf_\nu)_{h-1}$ (by \autoref{cor:east-only-same}), then by pushing the $t$-th row over to the east by one step, we have a new set of three points which are strictly above $k$ that satisfy the conditions of the algorithm and we can bring an east step down from even further above.
    Notice also that this operation decreases the horizontal distance by $1$ each time.
    Therefore continuing in this way, we keep pulling east steps from higher rows until either we end up pulling a final east step from the top row, or our horizontal distance is reduced to $0$ to get $D^{(k)}$.
    We check each of these two cases separately.
    \begin{enumerate}
        \item If we end up pulling an east step down from the top row then every $r_j^{D^{(k+1)}}$ where $j \geq t$ gets shifted to the east by one step.
            Furthermore, the horizontal distance of every $r_j^{D^{(k+1)}}$ is strictly greater than $0$ implying $z = s_N + 1$.
            Then $\laf_{D^{(k)}} = \laf_{D^{(k+1)}} + (01)_t$ as desired.
        \item If we end up pulling $(\laf_\nu)_k$ east steps and thus get to a point where the horizontal distance is zero then we know this happens only if we pull an east step from a row whose right hand point has horizontal distance $0$.
            Let $z$ be this row.
            In this case, we cannot alter the path further above the row $z$.
            In other words, we are pushing east every right hand point between the $t$-th row and the $(z-1)$-th row east by one.
            Therefore,
            \[
                \laf_{D^{(k)}} = \laf_{D^{(k+1)}} + (01)_t - (01)_z
            \]
            as desired.
    \end{enumerate}
\end{proof}

We next extend the results of \autoref{lem:algo-stair} to when we have different entries in $\laf_\nu$ in order to have an in-degree version of \autoref{prop:out-iff-east}.

\begin{proposition}
    \label{prop:indeg-is-max-size}
    Let $\nu$ be a path from $(0,0)$ to $(s_E, s_N)$ and let $D$ be the path produced by the area algorithm.
    Then $\indeg(D)$ is equal to the size $\sigma_\nu$ of the maximal staircase shape $\nu$-Dyck path $\xi_\nu$.
\end{proposition}
\begin{proof}
    We show that $\indeg(D)$ is equal to $\sigma_\nu$ by constructing $\xi_\nu$ one step at a time while going down in the $\nu$-Tamari order.
    We do this using the $D^{(k)}$ from \autoref{lem:many-stairs}:
    \[
        \laf_{D^{(k)}} = \laf_{D^{(k+1)}} + (01)_t - (01)_z
    \]
    where $t$ is the row containing $t_k^{D^{(k+1)}}$ and $z$ is the first row weakly after the $t$-th row whose right hand point has horizontal distance $0$.
    We inductively proceed from $s_N +1 $ to $1$ using \autoref{lem:many-stairs} where for each $D^{(k)}$ we associate a $\nu$-Dyck path $\xi_k$ based on $\xi_{k+1}$.
    This induction is given by the formula:
    \[
        \laf_{\xi_k} = \begin{cases}
            \laf_{\xi_{k+1}} + (01)_{k} - (01)_{z'} & \text{if } (\laf_\nu)_k \neq 0\\
            \laf_{\xi_{k+1}} & \text{otherwise}
        \end{cases}
    \]
    where $z'$ is the first row above the $k$-th row whose right hand point has horizontal distance $0$ (in other words, we never get negative numbers).
    We let $\xi_{s_N + 1}$ be the $\nu$-Dyck path $N^{s_N}E^{s_E}$.
    Note that this implies that $\xi_{s_N + 1} = D^{(s_N)} = N^{s_N}E^{s_E}$ and that $\xi_{s_N}$ will be the staircase shape $\nu$-Dyck path of size $1$ since the first row will always contribute to the in-degree for $D^{(s_N)}$.
    Furthermore, this implies $(\laf_{\xi_{k}})_i \geq (\laf_{D^{(k)}})_i$ for all $i \geq k$, where the inequality is strict when $i = k$.

    To keep track of the in-degree, we let $c_k$ be the number of components $(\laf_{D^{(k)}})_i$ where $i \geq k$ such that $t_i^{D^{(k)}} = h_i^{D^{(k)}}$ and whose horizontal distance is non-zero.
    In other words, $c_k$ is the number of components $i \geq k$ which contribute to the in-degree.
    Note that $c_{s_{N} + 1} = 0$, $c_{s_N} = 1$ and $c_1 = \indeg(D)$.
    Therefore, it suffices to show that $c_1 = \sigma_\nu$.

    First, note that if $(\laf_\nu)_k = 0$ then $D^{(k)} = D^{(k+1)}$, $\xi_k = \xi_{k+1}$ and $c_k = c_{k+1}$.
    Next, notice that the formula for the left area vector of $\xi_k$ is almost identical to that of $D^{(k)}$ except that we push east everything starting from the $k$-th row up until the first $0$ (whose row we denoted by $z'$).
    In particular, if we take $\xi_1$ and for each $N^a$ with $a > 1$ we take all $N$ except one and move them to the beginning of the path, we get $\xi_\nu$.
    
    Therefore, to prove that $c_1 = \sigma_\nu$ it suffices to show that we increase $c_k$ by one whenever we increase the number of steps in $\xi_k$.
    We show that the following are equivalent
    \begin{enumerate}
        \item \label{enum:in:one}$c_k = c_{k+1} + 1$,
        \item \label{enum:in:two}$\laf_{D^{(k)}}$ does not have a new $0$ in a component as compared to $\laf_{D^{(k+1)}}$ 
        \item \label{enum:in:three}$z' = s_N + 1$,
        \item \label{enum:in:four}$\xi_k$ has one more step than $\xi_{k+1}$.
    \end{enumerate}

    \textbf{\autoref{enum:in:one} $\iff$ \autoref{enum:in:two}:}
    If $(\laf_\nu)_k = (\laf_\nu)_{k+1}$ then we know that $t = k+1$ since for each $i \geq k+1$ its associated $t$ is greater than $i$, meaning that we only ever added east steps in rows strictly higher than $k+1$.
    Therefore the row $k+1$ hasn't changed keeping the same horizontal distance.
    By construction, the $k$-th component now contributes to the in-degree.
    If $\laf_{D^{(k)}}$ does not have a new $0$ in a component as compared to $\laf_{D^{(k+1)}}$, we know $z = s_N + 1$ and therefore the area vector is just decreased by $1$ at every component after the $k$-th component.
    In other words, $i \geq k+1$ contributes to the in-degree of $D^{(k+1)}$ if and only if it does for $D^{(k)}$, \ie $c_k = c_{k+1} + 1$.
    If on the other hand $\laf_D^{(k)}$ does have a new $0$ in a component as compared to $\laf_{D^{(k+1)}}$, this component is no longer contributing to the in-degree, but all other components are.
    Therefore $c_k = c_{k+1}$.

    If $(\laf_\nu)_k \neq (\laf_\nu)_{k+1}$, then $t > k+1$. 
    Since all the points between $t_k^{D^{(k+1)}}$ and $r_k^{D^{(k+1)}}$ must have horizontal distance greater than one, and by a similar argument as in the $(\laf_\nu)_k = (\laf_\nu)_{k+1}$ case for all points after $t_k^{D^{(k+1)}}$, we have a similar result.
    In other words, if $\laf_{D^{(k)}}$ does not have a new $0$ in a component as compared to $\laf_{D^{(k+1)}}$, then $c_k = c_{k+1} + 1$, otherwise $c_k = c_{k+1}$.

    \textbf{\autoref{enum:in:two} $\iff$ \autoref{enum:in:three}}
    First note that $z'= s_N + 1$ implies $z = s_N + 1$.
    Furthermore, since $(\laf_{\xi_{k}})_i \geq (\laf_{D^{(k)}})_i$ for all $i \geq k$, we know that all components after the $k$-th component in $\laf_{D^{(k+1)}}$ are greater than $1$ (or have a $0$ between themselves and the $k$-th component).
    In other words, $\laf_{D^{(k)}}$ does not have a new $0$ in a component as compared to $\laf_{D^{(k+1)}}$.
    Using the same fact (that $(\laf_{\xi_{k}})_i \geq (\laf_{D^{(k)}})_i$ for all $i \geq k$), the converse is true as well.

    \textbf{\autoref{enum:in:three} $\iff$ \autoref{enum:in:four}}
    To finish the proof it suffices to show that $z' = s_N + 1$ if and only if $\xi_k$ has one more step than $\xi_{k+1}$.
    If $z' = s_N + 1$ then we are pulling an east step from the top row to the $k$-th row in $\xi_k$ therefore increasing the number of steps by one.
    Conversely, if we increased the number of steps by one, then we must have pulled from a row with multiple east steps.
    Since we only ever add one east step to any row, this means this east step must have come from the top row and therefore $z' = s_N + 1$.
    
\end{proof}

\begin{corollary}
    \label{cor:indeg_is_N-1}
    Let $D$ be an $m$-Dyck path of height $n$.
    Then $\indeg(D) = n-1$ if and only if $t_i = h_i$ and $\horiz{r_i} \neq 0$ for all $1 < i \leq n$.
\end{corollary}

\begin{corollary}
    Let $\nu$ be a path.
    Then $\maxinnu = \maxoutnu$.
\end{corollary}
\begin{proof}
    This is a direct consequence of \autoref{lem:maxout-size} and \autoref{prop:indeg-is-max-size}
\end{proof}

\section{Isomorphisms}
\label{sec:isomorphisms}
We now analyse the subposets $\Toutnu$ and $\Tinnu$ using ``smaller'' $\nu$-Dyck paths.
We begin by studying $m$-Dyck paths of height $n$ and discuss arbitrary $\nu$ later.
Recall that the $m$-Dyck paths of height $n$ are the $\nu$-Dyck paths where $\nu = (NE^m)^n$ for some $m \geq 1$ and $n \geq 1$.

\subsection{The out-degree poset for \texorpdfstring{$m$}{m}-Dyck paths of height \texorpdfstring{$n$}{n}}
\label{ss:out-degree-poset}
In this section we study the subposet $\Toutnu$ which contains only $\nu$-Dyck paths with maximal out-degree.
By \autoref{cor:outdeg-n} for $m$-Dyck paths of height $n$, this happens when there are exactly $n-1$ covering relations.

Let $\nu$ be a path and recall that $\xi_\nu$ is the associated maximal staircase shape $\nu$-Dyck path.
Recall further that addition of two partitions is defined to be component-wise addition.
Given a $\nu$-Dyck path $D$, we define a new $\nu$-Dyck path using $\xi_\nu$.
Let $D^-$ be the $\nu$-Dyck path whose partition is given by $\laf_{D^-} = \laf_D - \laf_{\xi_\nu}$.

\begin{example}
    As an example, let $\nu$ be the path in red and $\xi_\nu$ to be the maximal staircase shape $\nu$-Dyck path in black.

    ~
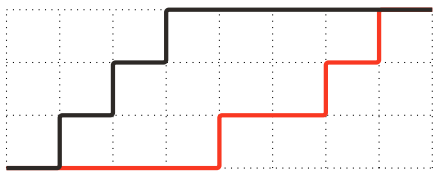
    \begin{center}
        \begin{tikzpicture}[scale=0.7]
            \draw[dotted] (0, 0) grid (8, 3);
            \draw[rounded corners=1, color=BrightRed, ultra thick] (0, 0) -- (4, 0) -- (4, 1) -- (6, 1) -- (6, 2) -- (7,2) -- (7,3) -- (8,3);
            \draw[rounded corners=1, color=Black, ultra thick] (0, 0) -- (1, 0) -- (1, 1) -- (2, 1) -- (2, 2) -- (3,2) -- (3,3) -- (8,3);
        \end{tikzpicture}
    \end{center}
    ~

    Then $\laf_\nu = (4, 6, 7)$ and $\laf_{\xi_\nu} = (1, 2, 3)$.
    Therefore,
    \[
        \laf_\nu - \laf_{\xi_\nu} = (4, 6, 7) - (1, 2, 3) = (3,4,4) = \laf_{\nu^-}
    \]
\end{example}

Before proving our main theorem, we show that going up in the $\nu$-Tamari order weakly decreases the out-degree.
\begin{lemma}
    \label{lem:outdeg-decreases}
    Suppose $\nu$ is a path such that $\laf_\nu$ has no repeating entries.
    Let $D$ and $D'$ be $\nu$-Dyck paths such that $D \leq_T D'$.
    Then $\outdeg(D) \geq \outdeg(D')$.
\end{lemma}
\begin{proof}
    Recall that we can go up in the $\nu$-Tamari order if and only if $r_i$ is proceeded by an east step.
    Therefore, it suffices to show that once $r_i^D$ no longer has east steps proceeding it, another east step cannot be put in front of it.
    Suppose $r_i^D$ has no east step proceeding it.
    If $i = 1$ then there is no way to put an east step before it.
    Otherwise, $i > 1$ and $r_i^D$ is proceeded by a north step.
    Since $\laf_\nu$ has no repeated entries, this implies $\horiz{r_i^D} > \horiz{p}$ where $p$ is the point directly before $r_i^D$ (on the other side of the north step).
    But this implies that it is impossible for $r_i^D$ to be the touch point for any $j < i$ .
    Therefore no east step can be placed before $r_i^D$.
    As each cover relation weakly decreases the out-degree, we have our lemma.
\end{proof}

\begin{theorem}
    The following is a ($\nu$-Tamari) order preserving bijection:
    \[
        \varphi: \Doutnm \to \mcD_{n, m-1}
    \]
    where $\laf_{\varphi(D)} = \laf_{D^-}$.
\end{theorem}
\begin{proof}
    Let $\xi$ be the maximal staircase shape $m$-Dyck path of height $n$.
    We start by showing we have a bijection between sets.
    By \autoref{prop:out-iff-east} and \autoref{cor:outdeg-n} we know that $D \in \Doutnm$ if and only if for every $i > 1$ then $r_i$ is proceeded by an east step.
    In other words, if and only if $\laf_{D} - \laf_\xi$ has all non-negative entries and is a partition.
    But this is true if and only if $\laf_D - \laf_\xi$ is a partition for a $m-1$-Dyck path of height $n$.
    Therefore, $\laf_{\varphi(D)} = \laf_{D^-}$ and $\varphi(D)$ is the $(m-1)$-Dyck path of height $n$ obtained from $D^-$ by removing the final $n$ east steps.
    
    We next show that this bijection is order preserving of the $\nu$-Tamari order.
    Note that in the $m = 1$ case, this is trivially true as there is only one element.
    Let $D$ and $D'$ be elements in $\Doutnm$ such that $D \leq_T D'$ and suppose $m \geq 2$.
    By \autoref{lem:outdeg-decreases}, we can further assume that $D \cover_T D'$.
    In particular we can assume that $D' = \Tup{D}{i}$ for some $i$ such that $D = dEtf$ and $\Tup{D}{i} = dtEf$ where $d = D_{[(0,0), r_i^D]}$ (with final $E$ removed), $t = D_{[r_i^D, t_i^D]}$ and $f = D_{[t_i^D, (nm, n)]}$.
    Note that applying $\varphi$ to $D$ and $D'$ pushes the $j$-th row east by $j-1$.
    Then the horizontal distance doesn't change under the $\varphi$ map, \ie $\horiz[n,m]{r_j^D} = \horiz[n,m-1]{r_j^{\varphi(D)}}$.
    Therefore $t_i^D$ and $t_i^{\varphi(D)}$ are on the same row as every $r_j^D$ keeps its horizontal distance.
    Furthermore, noting that since $\Tup{D}{i} = D'$ is in $\Doutnm$, this implies the $i$-th row of $\Tup{D}{i}$ has $i-1$ east steps that come before it implying $D$ has $i$ east steps which come before the $i$-th row.
    Therefore $\varphi(D)$ has an east step before $r_i^{\varphi(D)}$.
    Thus we can also go up in the $\nu$-Tamari order in the exact same way in $\varphi(D)$.
    Putting all this together, we have $D \cover_T \Tup{D}{i}$ implies $\varphi(D) \cover_T \Tup{\varphi(D)}{i} = \varphi(\Tup{D}{i})$ as desired.

    Conversely, we do the same procedure backwards.
    In other words, we view $D \cover_T D'$ in $\mcD_{n, m-1}$.
    Then the reverse map $\varphi^{-1}$ where we add an east step in each row keeps horizontal distances of all right hand points and touch points.
    Therefore $\varphi^{-1}(D) \cover_T \varphi^{-1}(D')$.
    Furthermore, since we added an east step in each row, both of these elements live in $\Doutnm$ as desired.
\end{proof}

\subsection{The in-degree poset for \texorpdfstring{$m$}{m}-Dyck paths of height \texorpdfstring{$n$}{n}}
\label{ss:in-degree-poset}
In this section we study the subposet $\Tinnu$ which contains only $\nu$-Dyck paths with maximal in-degree.
For $m$-Dyck paths of height $n$, by \autoref{cor:indeg_is_N-1} this happens exclusively when there are exactly $n-1$ covers.
As in the previous section, we begin by defining a new $\nu$-Dyck path out of the old one.

We say that a point $(x,y)$ \defn{comes before (after)} a point $(a,b)$ if $x < a$ ($x > a$).
Given a $\nu$-Dyck path $D$, we define a new $\nu$-Dyck path $\widehat{D}$ by its partition $\laf_{\widehat{D}}$ component-wise:
\[
    (\laf_{\widehat{D}})_i = (\laf_D)_i - \abs{\set{j \in [y]}{h_j^D \text{ comes before } r_i^D}}
\]
Note that since we are decreasing the partition, the path $\widehat{D}$ is also a $\nu$-Dyck path.

\begin{example}
    As an example, suppose we have the following $\nu$-Dyck path $D$:

    ~
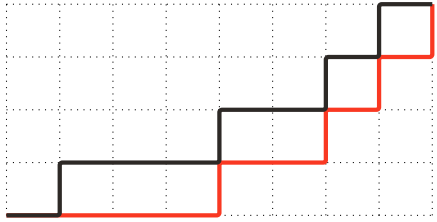
    \begin{center}
        \begin{tikzpicture}[scale=0.7]
            \draw[dotted] (0, 0) grid (8, 4);
            \draw[rounded corners=1, color=BrightRed, ultra thick] (0, 0) -- (4, 0) -- (4, 1) -- (6, 1) -- (6, 2) -- (7,2) -- (7,3) -- (8,3) -- (8, 4);
            \draw[rounded corners=1, color=Black, ultra thick] (0,0) -- (1,0) -- (1,1) -- (4,1) -- (4,2) -- (6,2) -- (6,3) -- (7,3) -- (7,4) -- (8,4);
        \end{tikzpicture}
    \end{center}
    For this example, the associated left area vector is given by $\laf_D = (1, 4, 6, 7)$.
    To calculate $\laf_{\widehat{D}}$, we need to find the hit points and the right hand points.
    
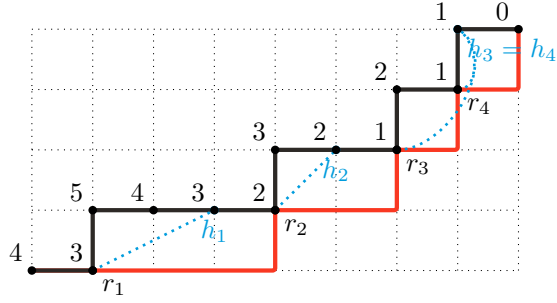
\begin{center}
        \begin{tikzpicture}[scale=0.8]
            \draw[dotted] (0, 0) grid (8, 4);
            \draw[rounded corners=1, color=BrightRed, ultra thick] (0, 0) -- (4, 0) -- (4, 1) -- (6, 1) -- (6, 2) -- (7,2) -- (7,3) -- (8,3) -- (8,4);
            \draw[rounded corners=1, color=Black, ultra thick] (0,0) -- (1,0) -- (1,1) -- (4,1) -- (4,2) -- (6,2) -- (6,3) -- (7,3) -- (7,4) -- (8,4);
            \draw[color=ProcessCyan, line width=1, dotted] (1,0) -- (3,1);
            \draw[color=ProcessCyan, line width=1, dotted] (4,1) -- (5,2);
            \draw[color=ProcessCyan, line width=1, dotted] (6,2) to[out=0, in=330] (7,4);
            \draw[color=ProcessCyan, line width=1, dotted] (7,3) to[out=0, in=330] (7,4);
            \fill[color=ProcessCyan] (3,1) circle[radius=2pt] node[below] {$h_1$};
            \fill[color=ProcessCyan] (5,2) circle[radius=2pt] node[below] {$h_2$};
            \fill[color=ProcessCyan] (7,4) circle[radius=2pt] node[below right] {$h_3 = h_4$};
            \fill (1,0) circle[radius=2pt] node[below right] {$r_1$};
            \fill (4,1) circle[radius=2pt] node[below right] {$r_2$};
            \fill (6,2) circle[radius=2pt] node[below right] {$r_3$};
            \fill (7,3) circle[radius=2pt] node[below right] {$r_4$};
            \fill (0,0) circle[radius=2pt] node[above left] {$4$};
            \fill (1,0) circle[radius=2pt] node[above left] {$3$};
            \fill (1,1) circle[radius=2pt] node[above left] {$5$};
            \fill (2,1) circle[radius=2pt] node[above left] {$4$};
            \fill (3,1) circle[radius=2pt] node[above left] {$3$};
            \fill (4,1) circle[radius=2pt] node[above left] {$2$};
            \fill (4,2) circle[radius=2pt] node[above left] {$3$};
            \fill (5,2) circle[radius=2pt] node[above left] {$2$};
            \fill (6,2) circle[radius=2pt] node[above left] {$1$};
            \fill (6,3) circle[radius=2pt] node[above left] {$2$};
            \fill (7,3) circle[radius=2pt] node[above left] {$1$};
            \fill (7,4) circle[radius=2pt] node[above left] {$1$};
            \fill (8,4) circle[radius=2pt] node[above left] {$0$};
        \end{tikzpicture}
    \end{center}
    Then we have
    \begin{align*}
        \left( \laf_{\widehat{D}} \right)_1 &= \left( \laf_D \right)_1 - \abs{ \set{j \in \left[ 4 \right]}{h_j^D \text{ comes before } r_1^D}} = 1 - 0 = 1\\
        \left( \laf_{\widehat{D}} \right)_2 &= \left( \laf_D \right)_2 - \abs{ \set{j \in \left[ 4 \right]}{h_j^D \text{ comes before } r_2^D}} = 4 - 1 = 3\\
        \left( \laf_{\widehat{D}} \right)_3 &= \left( \laf_D \right)_3 - \abs{ \set{j \in \left[ 4 \right]}{h_j^D \text{ comes before } r_3^D}} = 6 - 2 = 4\\
        \left( \laf_{\widehat{D}} \right)_4 &= \left( \laf_D \right)_4 - \abs{ \set{j \in \left[ 4 \right]}{h_j^D \text{ comes before } r_4^D}} = 7 - 2 = 5\\
    \end{align*}
    In other words, $\laf_{\widehat{D}} = (1, 3, 4, 5)$ giving us the following $\nu$-Dyck path.
    
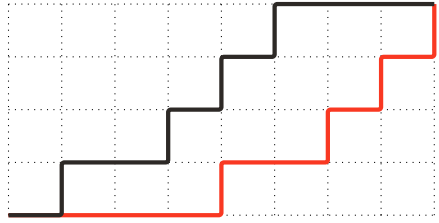
\begin{center}
        \begin{tikzpicture}[scale=0.7]
            \draw[dotted] (0, 0) grid (8, 4);
            \draw[rounded corners=1, color=BrightRed, ultra thick] (0, 0) -- (4, 0) -- (4, 1) -- (6, 1) -- (6, 2) -- (7,2) -- (7,3) -- (8,3) -- (8, 4);
            \draw[rounded corners=1, color=Black, ultra thick] (0,0) -- (1,0) -- (1,1) -- (3,1) -- (3,2) -- (4,2) -- (4,3) -- (5,3) -- (5,4) -- (8,4);
        \end{tikzpicture}
    \end{center}
\end{example}

Restricting our attention to $m$-Dyck paths of height $n$, it turns out the number of elements in the subposet $\Tinnm$ is equal to the $(m-1)$-Dyck paths of height $n$.
Before discussing the order, we show this bijection.
\begin{lemma}
    \label{lem:bij_nm}
    The following is a (set) bijection
    \[
        \phi: \mcD_{ {n,m}_\maxin} \to \mcD_{n,m-1}
    \]
    where $\laf_{\phi(D)} = \laf_{\widehat{D}}$.
\end{lemma}
\begin{proof}
    Let $\nu = (NE^m)^n$ and recall $\maxin(\mbbT_\nu) = n-1$ by \autoref{cor:indeg_is_N-1}.

    Suppose that $D \in \mcD_{{n,m}_{\maxin}}$ is a $\nu$-Dyck path which covers exactly $n-1$ $m$-Dyck paths of height $n$ in the $\nu$-Tamari order.
    By \autoref{cor:indeg_is_N-1}  for every $1 < i \leq n$ then $h_i^D = t_i^D$ and $\horiz[\nu]{r_i^D} \neq 0$.
    Let $D^{1} = D$ and let $D^{i}$ be the path obtained from $D^{i-1}$ by removing the east step directly after $h_i^{D^{i-1}}$ and placing it at the end of the path.
    By construction, $D^n = \widehat{D}$.
    An example of this can be seen on the right hand side of \autoref{fig:phi-ex}.
    Note that this operation keeps all $h_j^{D^{i-1}}$ static when $h_j^{D^{i-1}}$ comes before $h_i^{D^{i-1}}$ and moves $h_j^{D^{i-1}}$ one step to the west otherwise.
    Notice that $\widehat{D}$ is weakly above $(NE^{m-1})^nE^n$ if and only if $\horiz[\nu]{r_i^{\widehat{D}}} \geq (i-1)$ for every $i$.
    But $\horiz[\nu]{r_i^{\widehat{D}}} = \horiz[\nu]{r_i^D} + \#\set{h_j^D}{h_j^D \text{ comes before } r_i^D}$.
    For the hit points of $D$ for $j < i$, we remove one east step for every $h_j^D$ that comes before $r_i^D$ and for each $h_j^D$ that doesn't come before $r_i^D$ then $\horiz[\nu]{r_i^D} > \horiz[\nu]{r_j^D}$.
    Therefore every $j < i$ contributes at least one to the horizontal distance and $\horiz[\nu]{r_i^D} + \#\set{h_j^D}{h_j^D \text{ comes before } r_i^D} \geq i-1$ as desired.
    Therefore $\widehat{D}$ is weakly above $(NE^{m-1})^nE^m$.
    Letting $\phi(D)$ be $\widehat{D}$ with the final $m$ east steps removed implies that $\phi(D)$ lives in $\mcD_{n, m-1}$ and $\laf_{\phi(D)} = \laf_{\widehat{D}}$.

    In the other direction, let $\nu_{i-1}$ be the path obtained from $\nu_{i}$ with an east step added in the $i$-th row where $\nu_{n+1}$ is the path $(NE^{m-1})^{n}$.
    Then $\nu_1$ is the path $(NE^{m})^n = \nu$.
    Let $D_{n+1} \in \mcD_{n,m-1} = \mcD_{\nu_{n+1}}$ and inductively, let $D_{i-1}$ be obtained from $D_{i}$ by adding an east step after $h_{i-1}^{D_i}$ in $D_i$.
    Notice that $D_{i-1} \in \mcD_{\nu_{i-1}}$ since we are adding at most $j - i  - 1$ east steps before the $j$-th row where $j > i$.
    It remains to show that $D_{1} \in \Dinnm$, \ie $D_{1}$ covers exactly $n-1$ $m$-Dyck paths of height $n$.
    By construction, for $i = 1$ we have $\horiz[\nu_1]{r_1^{D_1}} = 0$ implying we can't go down for $i = 1$.
    Therefore we must show for all $i > 1$ that $h_i^{D_{1}} = t_i^{D_{1}}$ and $\horiz[\nu_1]{r_i^{D_1}} \neq 0$.

    Therefore, we suppose $i > 1$.
    Let $h$ be the row containing the point $h_{i-1}^{D_{i}}$.
    Recall that going from $D_{i}$ to $D_{i-1}$ pushes every point after $h_{i-1}^{D_{i}}$ to the east (as we add an east step in the $h$-th row) and also adds an east step to the $i$-th row of $\nu_i$ to obtain the $\nu_{i-1}$-Dyck path $D_{i-1}$.
    In other words, the horizontal distance of every point stays the same unless the point lies strictly between $r_{i-1}^{D_{i}}$ and $h_{i-1}^{D_{i}}$, in which case the horizontal distance increases by exactly one.

    We will show that if $t_j^{D_{i}} = h_j^{D_{i}}$ and $\horiz[\nu_{i}]{r_j^{D_i}} \neq 0$ then $t_j^{D_{i-1}} = h_{j}^{D_{i-1}}$ and $\horiz[\nu_{i-1}]{r_j^{D_{i-1}}} \neq 0$ for all $j$.
    We break this down into four cases.
    \begin{itemize}
        \item If $j < i-1$ then there are two cases to consider.
            If $h_{i-1}^{D_i}$ comes before $h_j^{D_i}$ then $h_j^{D_i}$ gets shifted over to the east by one step to become $h_j^{D_{i-1}}$ (and similarly with $t_i^{D_i}$).
            But since $\nu_{i-1}$ adds an east step in the $i$-th row of $\nu_i$, the horizontal distance doesn't change.
            If $h_j^{D_i}$ comes before $h_{i-1}^{D_i}$ then, in particular, it comes before $r_{i-1}^{D_i}$.
            Therefore the horizontal distances, hit and touch points aren't altered.
        \item If $i-1 = j$ then $t_{i-1}^{D_{i-1}} = h_{i-1}^{D_{i-1}}$ and $\horiz[\nu_{i-1}]{r_{i-1}^{D_{i-1}}} > \horiz[\nu_i]{r_{i-1}^{D_i}} > 0$ since the east step is placed after $h_{i-1}^{D_i}$ and a new east step is added on the $i$-th row of $\nu_i$, as desired.
        \item If $i-1 < j < h$, then every point stays where they are, but the horizontal distance increases by $1$ since we are adding an east step in the $i$-th row when going to $\nu_{i-1}$.
            Therefore, $t_j^{D_{i-1}} = h_{j}^{D_{i-1}}$ and $\horiz[\nu_{i-1}]{r_j^{D_{i-1}}} > \horiz[\nu_i]{r_j^{D_i}} > 0$ as desired.
        \item If $h \leq j$, then every point moves to the east by one step and also the horizontal distances stay the same (since we added an east step in the $i$-th row).
            Therefore, $t_j^{D_{i-1}} = h_{j}^{D_{i-1}}$ and $\horiz[\nu_{i-1}]{r_j^{D_{i-1}}} = \horiz[\nu_i]{r_j^{D_i}} > 0$ as desired.
    \end{itemize}
    Therefore, in all cases we have exactly that  if $t_j^{D_{i}} = h_j^{D_{i}}$ and $\horiz[\nu_{i}]{r_j^{D_{i}}} \neq 0$ then $t_j^{D_{i-1}} = h_{j}^{D_{i-1}}$ and $\horiz[\nu_{i-1}]{r_j^{D_{i-1}}} \neq 0$ for all $j$.

    It remains to show that going from $D_{i}$ to $D_{i-1}$ makes it so that if $t_{i-1}^{D_{i}} \neq h_{i-1}^{D_{i}}$ then $t_{i-1}^{D_{i-1}} = h_{i-1}^{D_{i-1}}$ and $\horiz[\nu_{i-1}]{r_{i-1}^{D_{i-1}}} > 0$.
    Indeed, for every point between $r_{i-1}^{D_{i}}$ and $h_{i-1}^{D_{i}}$ the horizontal distance increases by one as we are adding an east step on the $i$-th row of $\nu_i$.
    Furthermore, since we are adding an east step after $h_{i-1}^{D_i}$, then $h_{i-1}^{D_{i-1}}$ becomes the point directly to the east of $h_{i-1}^{D_i}$.
    As all other points had their horizontal distance increased by one, there are no points between $r_{i-1}^{D_{i-1}}$ and $h_{i-1}^{D_{i-1}}$ which has the same horizontal distance.
    Therefore, $t_{i-1}^{D_{i-1}} = h_{i-1}^{D_{i-1}}$ and the horizontal distance is greater than $0$ since $\horiz[\nu_{i-1}]{r_{i-1}^{D_{i-1}}} = \horiz[\nu_{i}]{r_{i-1}^{D_{i}}} > 0$.

    Combining these facts means that in $D_{1}$ for every $i> 1$ we have that $h_i^{D_{1}} = t_{i}^{D_{1}}$ and $\horiz{r_i^{D_1}} \neq 0$ since going from $n+1$ to $1$ forces touch and hit points to coincide one row at a time while keeping touch and hit points of all other rows not changed if they already coincide.
    Therefore every touch point and hit point coincide and we can go down in the $\nu$-Tamari order for every $i > 1$ and $D_{1} \in \mcD_{ {n,m}_{\maxin}}$, giving us the reverse bijection.
\end{proof}

As a nice corollary, this bijection implies that the number of elements with maximal in-degree is equal to the number of elements with maximal out-degree.
\begin{corollary}
    \label{cor:in-eq-out}
    The number of maximal in-degree $m$-Dyck paths of height $n$ is equal to the number of out-degree $m$-Dyck paths of height $n$, \ie $\order{\Dinnm} = \order{\Doutnm}$.
\end{corollary}

The bijection above is exclusively on sets and, in particular, it does not preserve the $\nu$-Tamari order.
Therefore, we denote by $\leq_D$ the order on $\Dinnm$ which gives us $\Tinnm$, \ie $\Tinnm = \left( \Dinnm, \leq_D \right)$.
This order extends naturally to arbitrary $\nu$ giving us $\Tinnu = \left( \Dinnu, \leq_D \right)$.
It turns out that this poset is poset isomorphic to another partial order on $\mcD_{n, m-1}$ which we describe next.

Let $D$ be a $\nu$-Dyck path from $(0, 0)$ to $(s_E, s_N)$.
If the $i$-th right hand point $r_i$ is preceded by an east step and followed by a north step, we define the $\nu$-Dyck path $\Gup{D}{i}$ in the following way.
Let $d$ denote the subword of $D_{\left[ (0,0), r_i \right]}$ where the final $E$ has been removed.
Let $h = D_{\left[ r_i, h_i \right]}$ and $f = D_{\left[h_i, (s_E, s_N)\right]}$.
In other words $D = dEhf$.
Then $\Gup{D}{i}$ is the word $dhEf$.
In other words, we move the east step before $r_i$ to just after the $i$-th hit point $h_i$.
The \defn{$\nu$-Greedy order} is then the order on $\nu$-Dyck paths where $D$ is covered by $\Gup{D}{i}$ whenever $\Gup{D}{i}$ is defined.
We denote this poset by $\mbbG_{\nu} = \left( \mcD_\nu, \leq_{G} \right)$.
This poset was first defined in \cite[Section 7.2]{Chapoton_2020} as it pertains to Dyck paths.

\begin{remark}
    Notice that the $\nu$-Greedy order is almost identical to the $\nu$-Tamari order except that we use hit points instead of touch points.
\end{remark}

It turns out that $\Tinnm = \left( \Dinnm, \leq_D \right)$ is (poset) isomorphic to $\mbbG_{n,m-1}$; the $\nu$-Greedy order for $m-1$-Dyck paths of height $n$.
\begin{theorem}
    \label{thm:iso-dexter-greedy}
    The following is a poset isomorphism:
    \[
        \phi: \left( \Dinnm, \leq_D \right) \to \left( \mcD_{ {n, m-1}}, \leq_{G} \right)
    \]
    where $\laf_{\phi(D)} = \laf_{\widehat{D}}$.
\end{theorem}

We break this theorem up into the two directions to make it more simple.

\subsubsection{\texorpdfstring{$\mbbG_{n,m-1}$ to $\Tinnm$}{G to T}}
We start by showing that if two elements are comparable in the $\nu$-Greedy order, then they are comparable in the same way in $\left( \Dinnu, \leq_D \right)$ after applying the reverse of the $\phi$ function from  \autoref{lem:bij_nm}.
\begin{proposition}
    \label{prop:greedy-to-dex}
    Let $\phi$ be the map from \autoref{lem:bij_nm}.
    For $D, D' \in \mcD_{{n, m-1}}$ such that $D \leq_G D'$ then $\phi^{-1}(D) \leq_D \phi^{-1}(D')$.
\end{proposition}
\begin{proof}
    Let $\nu = (NE^{m-1})^n$ and $\nu' = (NE^m)^n$.
    Suppose that $D, D' \in \mcD_{n, m-1}$ such that $D \cover_G D'$ in the $\nu$-Greedy order.
    Then for $D$, there exists an $i$ such that $D' = \Gup{D}{i}$.
    In other words, we go up in the $\nu$-Greedy order such that $D = dEhf$ and $D' = dhEf$ where
    \begin{align*}
        d &= D_{\left[ (0,0), r_i \right]} \text{ with the final $E$ removed },\\
        h &= D_{\left[ r_i, h_i \right]}, \text{ and }\\
        f &= D_{\left[h_i, (s_E, s_N)\right]}.
    \end{align*}
    We now apply $\phi^{-1}$ to both $D$ and $D'$ and aim to show $\phi^{-1}(D) \leq_D \phi^{-1}(D')$.
    Since $\leq_D$ is a restriction of $\leq_T$, it suffices to show $\phi^{-1}(D) \leq_T \phi^{-1}(D')$.

    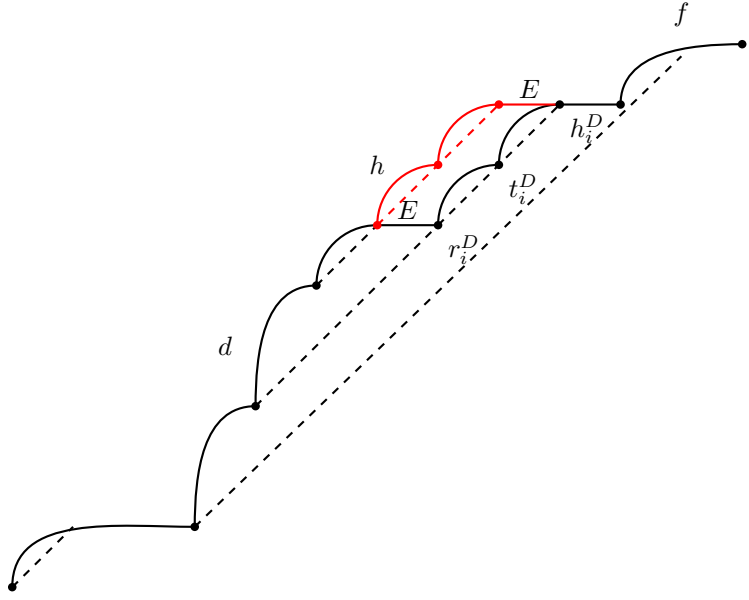
\begin{figure}[t]
        \begin{center}
            \begin{tikzpicture}[thick,scale=0.8
    ]
    \draw (0,0) to[out=90, in=180] (3,1);
    \draw (3,1) to[out=90, in=180] (4,3);
    \draw (4,3) to[out=90, in=180] (5,5);
    \draw (5,5) to[out=90, in=180] (6,6);
    \draw (6,6) -- (7, 6);
    \node at (6.5, 6.25) {$E$};
    \draw (7,6) to[out=90, in=180] (8,7);
    \draw (8,7) to[out=90, in=180] (9,8);
    \draw (9,8) -- (10,8);
    \draw (10,8) to[out=90, in=180] (12,9);

    \draw[dashed] (0,0) -- (1,1);
    \draw[dashed] (3,1) -- (11,8.8);
    \draw[dashed] (4,3) -- (7,6);
    \draw[dashed] (7,6) -- (9,8);
    \draw[dashed] (5,5) -- (6,6);

    \draw[red] (6,6) to[out=90, in=180] (7,7);
    \draw[red] (7,7) to[out=90, in=180] (8,8);
    \draw[red] (8,8) -- (9,8);
    \node at (8.5, 8.25) {$E$};
    \draw[red, dashed] (6,6) -- (8,8);
    

    \fill (0,0) circle[radius=2pt];
    \fill (3,1) circle[radius=2pt];
    \fill (4,3) circle[radius=2pt];
    \fill (5,5) circle[radius=2pt];
    \node at (3.5,4) {$d$};
    \fill[red] (6,6) circle[radius=2pt];
    \fill (7,6) circle[radius=2pt] node[below right] {$r_i^D$};
    \fill (8,7) circle[radius=2pt] node[below right] {$t_i^D$};
    \fill (9,8) circle[radius=2pt] node[below right] {$h_i^D$};
    \fill (10,8) circle[radius=2pt];
    \node at (11,9.5) {$f$};
    \fill (12,9) circle[radius=2pt];
    \fill[red] (7,7) circle[radius=2pt];
    \node at (6,7) {$h$};
    \fill[red] (8,8) circle[radius=2pt];
\end{tikzpicture}
            \caption{The path on the bottom is $D$ and the path on top is $\Gup{D}{i}$ where $D \cover_G \Gup{D}{i}$. The dashed lines indicate points with the same horizontal distance. The path $\nu$ is not shown in this figure.}
            \label{fig:greedy-to-dex}
        \end{center}
    \end{figure}

    Let $j > i$ be the row in which $h_i^D$ lies.
    We first analyse what happens to the points $d$, $h$, and $f$ under the map $\phi^{-1}$ in both $D$ and $D'$.
    Recall that applying $\phi^{-1}$ to a $\nu$-Dyck path adds an east step before each $h_k$ inductively from $k = n$ to $k = 1$.
    See \autoref{fig:greedy-to-dex} for a graphical view of the difference between the hit points of $D$ and $D'$.
    \begin{itemize}
        \item[f:] In both $D$ and $D'$, $f$ contains the same points.
            Since when applying $\phi^{-1}$ we proceed inductively from highest to lowest, for each $r_k^D$ with $k \geq j$, then $h_k^D = h_k^{D'}$ and each $k$-th right hand point in $D$ and $D'$ contribute an east step to the same row under $\phi^{-1}$.
        \item[h:] In both $D$ and $D'$, $h$ is associated to the rows between $i$ and $j$.
            Let $i \leq k < j$.
            Then $h_k^D$ is between $r_i^D$ and $h_i^D$ and therefore we know that $r_k^D$ contributes an east step only to a row strictly above $i$ and weakly below $j$.
            Going from $D$ to $D'$ only shifts these points one point to the east and therefore doesn't alter which rows each $r_k^{D'}$ contributes an east step to under the map $\phi^{-1}$.
            Therefore, each $k$-th right hand point in $D$ and $D'$ contribute an east step to the same row under $\phi^{-1}$.
        \item[d:] Finally, $d$ is the same initial path in $D$ and $D'$ and thus contains the same points.
            Let $k < i$ and there are four cases to consider.
            In all cases we have that $r_k^{D} = r_{k}^{D'}$ since $k < i$ and $D$ and $D'$ coincide on the subpath $d$.
            These cases can be viewed as the four dashed lines in \autoref{fig:greedy-to-dex}.
            \begin{itemize}
                \item If $h_k^D$ is on the path $d$ (the bottom left dashed line in \autoref{fig:greedy-to-dex}) and since $d$ coincides in $D$ and $D'$, then $h_k^{D} = h_k^{D'}$.
                    Therefore, under the map $\phi^{-1}$, $r_k^{D} = r_{k}^{D'}$ will contribute an east step to the same row.
                \item If $h_k^D$ is on the path $f$ (the bottom right dashed line in \autoref{fig:greedy-to-dex}) and since $f$ coincides in $D$ and $D'$, then $h_k^D = h_k^{D'}$.
                    Therefore, under the map $\phi^{-1}$, $r_k^{D} = r_{k}^{D'}$ will contribute an east step to the same row.
                \item If $h_k^D = h_i^D$ (the middle dashed line in \autoref{fig:greedy-to-dex}) then $h_k^{D} = h_k^{D'}$ as an east step was placed after $h_i^{D}$, the point $h_k^D$ still exists on $D'$ and no point between $r_k^D$ and $h_k^D$ has a smaller horizontal distance by construction.
                    Therefore, under the map $\phi^{-1}$, $r_k^{D} = r_{k}^{D'}$ will contribute an east step to the same row.
                \item Finally, if $h_k^D = r_i^{D'}$ (the top dashed line in \autoref{fig:greedy-to-dex}) is the point directly to the west of $r_i^D$ then the point $h_k^D$ becomes a touch point (but not necessarily $t_k^{D'}$) in $D'$ and the hit point then coincides with the hit point of $r_i^{D'}$, \ie $h_k^{D'} = h_i^{D'}$.
            \end{itemize}
            In other words, under the $\phi^{-1}$ map, if $h_k^D = r_i^{D'}$, then the $k$-th row contributes an east step in the $i$-th row for $D$ and in the $j$-th row for $D'$.
            In all other cases the $k$-th right hand point contributes an east step to the same row in both $D$ and $D'$ under $\phi^{-1}$.
    \end{itemize}
    Putting everything together, then $\phi^{-1}$ adds an east step to the same rows for both $D$ and $D'$ except for if $k < i$ such that $h_k^D = r_i^{D'}$ (which then contributes to the $i$-th row in $D$ and $j$-th row in $D'$).

    Next, using left area vectors, we have that $\laf_D$ and $\laf_{D'}$ are identical in every component except for in the $i$ through $j$ components (in which they are off by one).
    Combining this fact with the above, then
    \[
        \left( \laf_{\phi^{-1}(D)} \right)_k = \begin{cases}
            \left( \laf_{\phi^{-1}(D')} \right)_k & \text{if } k < i \text{ or } k \geq j\\
            \left(  \laf_{\phi^{-1}(D')}\right)_k - 1 - \#\set{p < i}{h_p^D = r_i^{D'}} & \text{if } i \leq k < j \\
        \end{cases}
    \]
    
    It remains to show that $\phi^{-1}(D) \leq_T \phi^{-1}(D')$ in the $\nu$-Tamari order.
    Note that since $j$ is the row containing $h_i^{D}$ and since $\phi$ doesn't alter the rows where touch and hit points exist, there is a point $p$ on the $j$-th row of $\phi^{-1}(D)$ such that $\horiz[\nu']{r_i^{\phi^{-1}(D)}} = \horiz[\nu']{p}$.
    Start with $\phi^{-1}(D)$ and do the following algorithm (which runs $1 + \#\set{p < i}{h_p^D = r_i^{D'}}$ times):
    \begin{enumerate}
        \item Let $p$ be the point on the $j$-th row such that $\horiz[\nu']{r_i} = \horiz[\nu]{p}$.
        \item While $t_i \neq p$, let $k > i$ be such that $t_k = p$ and go up in the $\nu$-Tamari order at $k$.
        \item Once no such $k$ exists, go up in the $\nu$-Tamari order at $i$.
        \item Continue the algorithm until $p = h_i^{\phi^{-1}(D')}$.
    \end{enumerate}
    In this manner we obtain $\phi^{-1}(D')$ as desired.
\end{proof}

\subsubsection{\texorpdfstring{$\Tinnm$ to $\mbbG_{n,m-1}$}{T to G}}
We next show that the reverse bijection is true.
For this, we need some machinery which we lay out next.

First, we consider our $m$-Dyck paths of height $n$ as Dyck paths of height $mn$.
To do this, we change our $\nu$ to be the staircase shape $(NE)^{mn}$ and we convert a $m$-Dyck path of height $n$ by making each north step into $m$ north steps.
\begin{example}
    Suppose we have the following $2$-Dyck path of height $3$.
    \begin{center}
        \begin{tikzpicture}[scale=0.8]
            \draw[dotted] (0, 0) grid (6, 3);
            \draw[rounded corners=1, color=BrightRed, ultra thick] (0, 0) -- (0, 1) -- (2, 1) -- (2, 2) -- (4, 2) -- (4,3) -- (6,3);
            \draw[rounded corners=1, color=Black, ultra thick] (0,0) -- (0,2) -- (2,2) -- (2,3) -- (6,3);
        \end{tikzpicture}
    \end{center}
    Then we can convert this into a Dyck path of height $6\; (=2 \cdot 3)$ by making each north step become length $2$.
    \begin{center}
        \begin{tikzpicture}[scale=0.5]
            \draw[dotted] (0, 0) grid (6, 6);
            \draw[rounded corners=1, color=BrightRed, ultra thick] (0, 0) -- (0, 1) -- (1, 1) -- (1,2) -- (2, 2) -- (2, 3) -- (3, 3) -- (3, 4) -- (4, 4) -- (4, 5) -- (5, 5) -- (5, 6) -- (6,6);
            \draw[rounded corners=1, color=Black, ultra thick] (0,0) -- (0,4) -- (2,4) -- (2,6) -- (6,6);
        \end{tikzpicture}
    \end{center}
\end{example}

The fact that this is possible comes from the following proposition.
\begin{proposition}{\cite[Proposition 4]{BFP-noIntTam}}
    The poset $\mcT_{n,m}$ is poset isomorphic to an upper ideal in $\mcT_{mn}$ where the elements in $\mcD_{n,m}$ are mapped to their respective elements in $\mcD_{mn}$ under the conversion given above.
\end{proposition}

In the same paper, Bousquet-Mélou, Fusy and Préville-Ratelle \cite{BFP-noIntTam} define a distance function on Dyck paths which keeps track of the Tamari order.
We define this distance function next.

Given a Dyck path $D$ of size $mn$, let $\ell_D(r_i, t_i)$ denote the length of the path from $r_i$ to $t_i$ divided by two (so each pair of north and east steps count as $1$).
The \defn{touch distance function of $D$} is the function $d_D^t: [mn] \to [mn]$ such that $d_D^t(i) = \ell_D(r_i, t_i)$.
As we did with the horizontal distance, we will usually denote $d_D^t$ as a vector $d_D^t = (\ell_D(r_1, t_1) , \ell_D(r_2, t_2), \ldots, \ell_D(r_{mn}, t_{mn}))$.
Two distance functions are comparable by comparing each component.
In other words, $d_D^t \leq d_{D'}^t$ if and only if $d_D^t(i) \leq d_{D'}^t(i)$ for all $i$.
It was shown in \cite{BFP-noIntTam} that the touch distance function preserves the Tamari order.
\begin{proposition}[{\cite[Proposition 5]{BFP-noIntTam}}]
    \label{prop:tamari-iff-distance}
    Let $D$ and $D'$ be two Dyck paths of height $mn$.
    Then $D \leq_T D'$ if and only if $d_D^t \leq d_{D'}^t$.
\end{proposition}

Since the greedy order uses hit points and not touch points, we define a new distance function on Dyck paths and show how it is related to the greedy order.
The \defn{hit distance function of $D$} is the function $d_D^h: [mn] \to [mn]$ such that $d_D^h(i) = \ell_D(r_i, h_i)$.
As with the touch distance function, we will usually denote $d_D^h$ as the vector $d_D^h = \left( \ell_D(r_1, h_1), \ldots, \ell_D(r_{mn}, h_{mn}) \right)$.
\begin{example}
    As an example, suppose we have the following Dyck path $D$ from before.
    \begin{center}
        \begin{tikzpicture}[scale=0.5]
            \draw[dotted] (0, 0) grid (6, 6);
            \draw[rounded corners=1, color=BrightRed, ultra thick] (0, 0) -- (0, 1) -- (1, 1) -- (1,2) -- (2, 2) -- (2, 3) -- (3, 3) -- (3, 4) -- (4, 4) -- (4, 5) -- (5, 5) -- (5, 6) -- (6,6);
            \draw[rounded corners=1, color=Black, ultra thick] (0,0) -- (0,4) -- (2,4) -- (2,6) -- (6,6);
        \end{tikzpicture}
    \end{center}
    Then $d_D^t = (6, 5, 2, 1, 2, 1)$ and $d_D^h = (6, 5, 4, 1, 2, 1)$.
\end{example}

Although we would like to say that the hit distance function preserves the greedy order, it turns out to be slightly more complex than that.
\begin{proposition}
    \label{prop:greedy-distance}
    Let $D$ and $D'$ be Dyck paths of height $mn$.
    Then $D \leq_G D'$ if and only if $d_D^t \leq d_{D'}^t$ and $d_D^{h} \leq d_{D'}^h$
\end{proposition}

We prove \autoref{prop:greedy-distance} by showing each direction independently.
For the forward direction, we study what happens in cover relations which then extends naturally to the order itself.
\begin{lemma}
    \label{lem:greedy-distance-cover}
    Let $D$ be a Dyck path of height $mn$ and let $\Gup{D}{i}$ be the Dyck path of height $mn$ obtained from $D$ by moving the east step before $r_i^D$ to after $h_i^D$, \ie $D \cover_G \Gup{D}{i}$.
    Then
    \begin{align*}
        d_{\Gup{D}{i}}^t(j) &= \begin{cases}
            d_{D}^t(j)& \text{if } t_j^D \neq r_i^D,\\
            d_{D}^t(j) + \ell_D(r_i^D, h_i^D)& \text{if } t_j^D = r_i^D,
        \end{cases}\\
        d_{\Gup{D}{i}}^h(j) &= \begin{cases}
            d_{D}^h(j)& \text{if } h_j^D \neq x_i^D,\\
            d_{D}^h(j) + \ell_D(r_i^D, h_i^D)& \text{if } h_j^D = x_i^D,
        \end{cases}
    \end{align*}
    where $x_i^D$ is the point before $r_i^D$.
\end{lemma}

This lemma is proved using the following diagram showing the cover relation.
\begin{figure}[h]
    \begin{center}
        \begin{tikzpicture}[thick,scale=0.8
    ]
    \draw (0,0) to[out=90, in=180] (1,1);
    \draw (1,1) to[out=90, in=180] (2,3);
    \draw (2,3) to[out=90, in=180] (3,4);
    \draw (3,4) -- (4, 4);
    \node at (3.5, 4.25) {$E$};
    \draw (4,4) to[out=90, in=180] (5,5);
    \draw (5,5) to[out=90, in=180] (6,6);
    \draw (6,6) -- (7,6);
    \draw (7,6) to[out=90, in=180] (9,7);

    \draw[dashed] (1,1) -- (6,6);
    \draw[dashed] (2,3) -- (3,4);

    \draw (4.75, 3.25) -- (5, 3) -- (7, 5) -- (6.75, 5.25);
    \node at (6.5, 3.5) {$\ell_D(r_i^D, h_i^D)$};

    \draw[red] (3,4) to[out=90, in=180] (4,5);
    \draw[red] (4,5) to[out=90, in=180] (5,6);
    \draw[red] (5,6) -- (6,6);
    \node at (5.5, 6.25) {$E$};
    \draw[red, dashed] (3,4) -- (5,6);
    

    \fill (0,0) circle[radius=2pt];
    \fill (1,1) circle[radius=2pt];
    \fill (2,3) circle[radius=2pt];
    \fill (3,4) circle[radius=2pt] node[below] {$x_i^D$};
    \fill (4,4) circle[radius=2pt] node[below right] {$r_i^D$};
    \fill (5,5) circle[radius=2pt] node[below right] {$t_i^D$};
    \fill (6,6) circle[radius=2pt] node[below right] {$h_i^D$};
    \fill (7,6) circle[radius=2pt];
    \fill (9,7) circle[radius=2pt];
    \fill[red] (4,5) circle[radius=2pt];
    \fill[red] (5,6) circle[radius=2pt];

\end{tikzpicture}
        \caption{In this figure there are two paths, $D$ (below) and $\Gup{D}{i}$ (above) where $D \cover_G \Gup{D}{i}$. The two paths coincide on every point before $x_i^D$ and after $h_i^D$.}
        \label{fig:greedy-distance-cover}
    \end{center}
\end{figure}
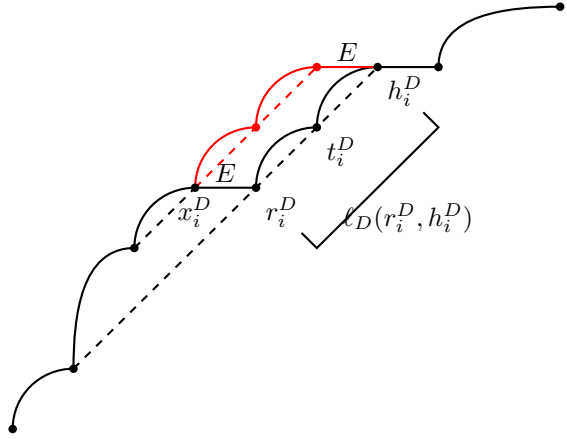

Since our cover relations always increase both distances, it is clear that this extends naturally to show the forward direction of our proposition.
We next show the reverse direction.
\begin{lemma}
    \label{lem:dist-imp-greedy}
    Let $D$ and $D'$ be Dyck paths of height $mn$.
    If $d_D^t \leq d_{D'}^t$ and $d_D^{h} \leq d_{D'}^h$ then $D \leq_G D'$.
\end{lemma}
\begin{proof}
    We first note that, since $d_D^t \leq d_{D'}^t$, then by \autoref{prop:tamari-iff-distance} $D \leq_T D'$ and $D$ is weakly below $D'$ in terms of paths, \ie the path of $D$ is weakly between $(NE)^{mn}$ and $D'$.
    We proceed by strong induction on $\norm{d_{D'}^t - d_{D}^t} + \norm{d_{D'}^h - d_D^h}$ where $\norm{(a_1, \ldots, a_k)} = \abs{a_1} + \cdots + \abs{a_k}$.
    First note that if $\norm{d_{D'}^t - d_{D}^t} = 0$ then $D' = D$ by \autoref{prop:tamari-iff-distance} and therefore the base case is handled.

    Let $i$ be the first row in which $D'$ and $D$ no longer coincide.
    In other words, $r_{i-1}^D = r_{i-1}^{D'}$ and $r_{i}^D \neq r_{i}^{D'}$.
    Let $\Gup{D}{i}$ be the Dyck path of height $mn$ obtained from $D$ by moving the east step before $r_{i}^D$ to after the hit point $h_{i}^D$, \ie we go up in the greedy order.
    By induction, it suffices to show $d_{\Gup{D}{i}}^t \leq d_{D'}^t$ and $d_{\Gup{D}{i}}^h \leq d_{D'}^h$; which then implies $D \leq_G \Gup{D}{i}\leq_G D'$ as desired.

    By \autoref{lem:greedy-distance-cover} we have
    \begin{align*}
        d_{\Gup{D}{i}}^t(j) &= \begin{cases}
            d_{D}^t(j)& \text{if } t_j^D \neq r_i^D,\\
            d_{D}^t(j) + \ell_D(r_i^D, h_i^D)& \text{if } t_j^D = r_i^D,
        \end{cases}\\
        d_{\Gup{D}{i}}^h(j) &= \begin{cases}
            d_{D}^h(j)& \text{if } h_j^D \neq x_i^D,\\
            d_{D}^h(j) + \ell_D(r_i^D, h_i^D)& \text{if } h_j^D = x_i^D,
        \end{cases}
    \end{align*}
    where $x_i^D$ is the point before $r_i^D$.
    Since $d_D^t \leq d_{D'}^t$ and $d_D^h \leq d_{D'}^h$ it suffices to show $d_{\Gup{D}{i}}^t(j) \leq d_{D'}^t (j)$ and $d_{\Gup{D}{i}}^h(j) \leq d_{D'}^h(j)$ for $j \in [mn]$ where the component changes.
    We handle each of the cases separately where in both cases we have $j < i$.

    \textbf{$\b{t_j^D = r_i^D}$:} In this case we have $d_{\Gup{D}{i}}^t(j) = d_{D}^t(j) + \ell_D(r_i^D, h_i^D)$.
        Let us suppose contrarily that $d_{D'}^t(j) < d_{\Gup{D}{i}}^t(j)$.
        Therefore $t_j^{D'}$ comes strictly (diagonally) before $t_j^{\Gup{D}{i}} = h_i^{D}$ as $r_j^D = r_j^{D'} = r_{j}^{\Gup{D}{i}}$.
        Furthermore, $r_i^{D} \neq r_i^{D'}$ and $D \leq_T D'$ implies that $\horiz[\nu]{} < \horiz[\nu]{r_i^{D'}}$ and thus $t_j^{D'}$ comes strictly (diagonally) after $r_i^D$.
        Since $\horiz[\nu]{r_j^{D'}} = \horiz[\nu]{r_i^D} < \horiz[\nu]{r_i^{D'}}$ and since $t_j^{D'}$ is strictly between $r_i^{D}$ and $h_i^{D}$, then $h_i^{D'}$ must be strictly (diagonally) between $r_i^D$ and $h_i^D$ (shifted over to the west).
        In other words, the number of north steps between $r_i^{D'}$ and $h_i^{D'}$ is strictly less than the number of north steps between $r_i^D$ and $h_i^D$.
        But since length is always equal to the number of north steps, this implies $d_{D'}^h(i) < d_{D}^h(i)$ contradicting our initial assumptions.

    \textbf{$\b{h_{j}^D = x_{i}^D}$:} In this case we have $d_{D_i}^h(j) = d_{D}^h(j) + \ell_{D}(r_i^D, h_i^D)$.
        We proceed by a similar argument as in the last case.
        Let us suppose contrarily that $d_{D'}^h(j) < d_{\Gup{D}{i}}^h(j)$.
        Therefore $h_j^{D'}$ comes strictly (diagonally) before $h_j^{\Gup{D}{i}} = h_i^{\Gup{D}{i}}$ as $r_j^D = r_j^{D'} = r_j^{\Gup{D}{i}}$.
        Since $r_i^{D'} \neq r_i^D$ then $\horiz[\nu]{r_i^{D'}} \geq \horiz[\nu]{} +1 = \horiz[\nu]{} = \horiz[\nu]{r_j^{D'}}$.
        In particular, these two facts imply that $h_i^{D'}$ is strictly (diagonally) before $h_i^{\Gup{D}{i}}$.
        Furthermore, since $d_D^h(i) = d_{\Gup{D}{i}}^h(i)$, then the number of north steps between $r_i^{D'}$ and $h_i^{D'}$ is strictly less than the number of north steps between $r_i^D$ and $h_i^D$.
        Therefore $d_{D'}^h(i) < d_{D}^h(i)$ which is a contradiction.

    Therefore we have that $d_{\Gup{D}{i}}^t \leq d_{D'}^t$ and $d_{\Gup{D}{i}}^h \leq d_{D'}^h$.
    Since at least one of $\norm{d_{D'}^t - d_{\Gup{D}{i}}^t}$ and $\norm{d_{D'}^h - d_{\Gup{D}{i}}^h}$ is strictly less than their counterparts, inductively, we know that $\Gup{D}{i} \leq_G D'$.
    Furthermore, since $D \cover_G \Gup{D}{i}$ we have $D \leq_G D'$ as desired.
\end{proof}

\begin{proof}[Proof of \autoref{prop:greedy-distance}]
    This is a direct result of \autoref{lem:greedy-distance-cover} and \autoref{lem:dist-imp-greedy}.
\end{proof}

Finally, before proving that $\Tinnm$ implies $\mbbG_{n, m-1}$, we define a way to describe $\phi$ in $\mcD_{mn}$.
Let $\pi$ be the poset isomorphism between $\mcT_{n,m}$ and the associated poset ideal in $\mcT_{mn}$ given at the beginning of the section.
Note that $\pi$ preserves touch and hit points.
Recall that $\phi$ is a (set) bijection between $\Dinnm$ and $\mcD_{n,m-1}$.

For $D$ a Dyck path of height $mn$ (in $\pi(\Dinnm)$), we let $\bar{\phi} = \pi \circ \phi \circ \pi^{-1}$.
In particular, $\pi^{-1}$ removes $m-1$ north steps for every string of $m$ north steps.
Then $\phi$ removes the east step after $h_i$ for all $i > 1$.
Finally, $\pi$ adds $m-2$ north steps for every north step.
Combining this together (and using the algorithm for obtaining $\phi$), we let $D^0 = D$ and let $D^{i+1}$ be the Dyck of height $mn$ obtained from $D^{i}$ by taking the east step after $h_{im + 1}^{D^{i}}$ (or the final east step) and placing them at the end of $D^{i}$ (north before east).
Removing the final $E^n$ at the end of $D^{n}$ and one $N$ step from each $m$ $N$ steps gives us $\bar{\phi}(D)$.

\begin{example}
    Let $D \in \pi(\Dinnm)$ be the following Dyck path of height $3 \cdot 4$ where $m = 3,\,n=4$.

    ~
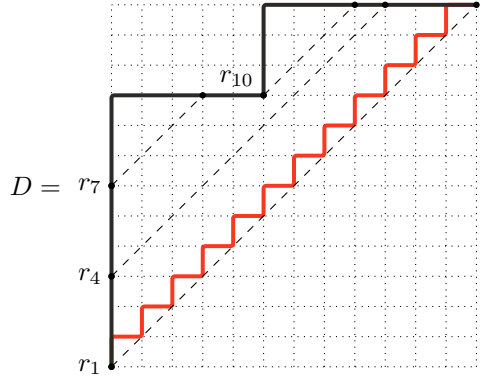
    \begin{center}
        \begin{tikzpicture}[scale=0.4]
            \node at (-2.5,6) {$D = $};
            \draw[dotted] (0, 0) grid (12, 12);
            \draw[rounded corners=1, color=BrightRed, line width=1.5] (0, 0) -- (0, 1) -- (1, 1) -- (1, 2)-- (2, 2) -- (2,3) -- (3,3) -- (3,4) -- (4, 4) -- (4,5) -- (5,5) -- (5,6) -- (6,6) -- (6,7) -- (7,7) -- (7,8) -- (8,8) -- (8,9) -- (9,9) -- (9,10) -- (10,10) -- (10,11) -- (11,11) -- (11, 12) -- (12,12);
            \draw[rounded corners=1, color=Black, line width=1.5] (0,0) -- (0,9) -- (5,9) -- (5, 12) -- (12,12);
            \fill (0,0) circle[radius=3pt];
            \fill (0,0) circle[radius=3pt] node[left] {$r_1$};
            \fill (0,3) circle[radius=3pt];
            \fill (0,3) circle[radius=3pt] node[left] {$r_4$};
            \fill (0,6) circle[radius=3pt];
            \fill (0,6) circle[radius=3pt] node[left] {$r_7$};
            \fill (5,9) circle[radius=3pt];
            \fill (5,9) circle[radius=3pt] node[above left] {$r_{10}$};
            \draw[dashed] (0,0) -- (12,12);
            \fill (12,12) circle[radius=3pt];
            \draw[dashed] (0,3) -- (9,12);
            \fill (9,12) circle[radius=3pt];
            \draw[dashed] (0,6) -- (3,9);
            \fill (3,9) circle[radius=3pt];
            \draw[dashed] (5,9) -- (8,12);
            \fill (8,12) circle[radius=3pt];
        \end{tikzpicture}
    \end{center}
    In the image, the dashed lines represent points with equal horizontal distance for $i \in [n]$ at the points $(i-1)m  + 1$.

    We next describe where $D$ is sent by $\pi \circ \phi \circ \pi^{-1}$.

    ~
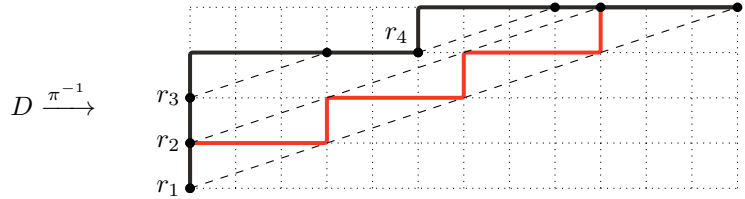
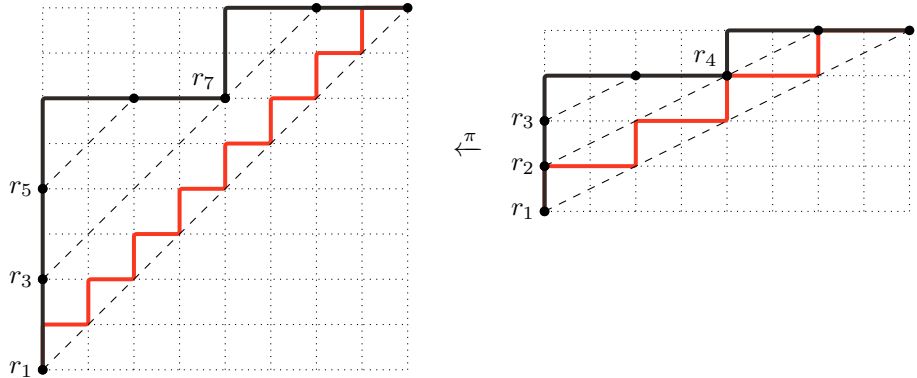
    \begin{center}
        \begin{tikzpicture}[scale=0.6]
            \begin{scope}[shift={(0,2)}]
                \node at (0,0) {$D \xrightarrow{\pi^{-1}}$};
            \end{scope}
            \begin{scope}[shift={(3,0)}]
                \draw[dotted] (0, 0) grid (12, 4);
                \draw[rounded corners=1, color=BrightRed, line width=1.5] (0, 0) -- (0, 1) -- (3, 1) -- (3, 2)-- (6, 2) -- (6,3) -- (9,3) -- (9,4) -- (12, 4);
                \draw[rounded corners=1, color=Black, line width=1.5] (0,0) -- (0,3) -- (5,3) -- (5, 4) -- (12,4);
                \fill (0,0) circle[radius=3pt];
                \fill (0,0) circle[radius=3pt] node[left] {$r_1$};
                \fill (0,1) circle[radius=3pt];
                \fill (0,1) circle[radius=3pt] node[left] {$r_2$};
                \fill (0,2) circle[radius=3pt];
                \fill (0,2) circle[radius=3pt] node[left] {$r_3$};
                \fill (5,3) circle[radius=3pt];
                \fill (5,3) circle[radius=3pt] node[above left] {$r_{4}$};
                \draw[dashed] (0,0) -- (12,4);
                \fill (12,4) circle[radius=3pt];
                \draw[dashed] (0,1) -- (9,4);
                \fill (9,4) circle[radius=3pt];
                \draw[dashed] (0,2) -- (3,3);
                \fill (3,3) circle[radius=3pt];
                \draw[dashed] (5,3) -- (8,4);
                \fill (8,4) circle[radius=3pt];
            \end{scope}
            \begin{scope}[shift={(10, -1)}]
                \node at (0,0) {$\downarrow$};
                \node at (0.6, 0) {$\phi$};
            \end{scope}
            \begin{scope}[shift={(7,-6.5)}]
                \draw[dotted] (0, 0) grid (8, 4);
                \draw[rounded corners=1, color=BrightRed, line width=1.5] (0, 0) -- (0, 1) -- (2, 1) -- (2, 2)-- (4, 2) -- (4,3) -- (6,3) -- (6,4) -- (8, 4);
                \draw[rounded corners=1, color=Black, line width=1.5] (0,0) -- (0,3) -- (4,3) -- (4, 4) -- (8,4);
                \fill (0,0) circle[radius=3pt];
                \fill (0,0) circle[radius=3pt] node[left] {$r_1$};
                \fill (0,1) circle[radius=3pt];
                \fill (0,1) circle[radius=3pt] node[left] {$r_2$};
                \fill (0,2) circle[radius=3pt];
                \fill (0,2) circle[radius=3pt] node[left] {$r_3$};
                \fill (4,3) circle[radius=3pt];
                \fill (4,3) circle[radius=3pt] node[above left] {$r_{4}$};
                \draw[dashed] (0,0) -- (8,4);
                \fill (8,4) circle[radius=3pt];
                \draw[dashed] (0,1) -- (6,4);
                \fill (6,4) circle[radius=3pt];
                \draw[dashed] (0,2) -- (2,3);
                \fill (2,3) circle[radius=3pt];
            \end{scope}
            \begin{scope}[shift={(5.33,-5)}]
                \node at (0,0) {$\xleftarrow{\pi}$};
            \end{scope}
            \begin{scope}[shift={(-4,-10)}]
                \draw[dotted] (0, 0) grid (8, 8);
                \draw[rounded corners=1, color=BrightRed, line width=1.5] (0, 0) -- (0, 1) -- (1, 1) -- (1, 2)-- (2, 2) -- (2,3) -- (3,3) -- (3,4) -- (4, 4) -- (4,5) -- (5,5) -- (5,6) -- (6,6) -- (6,7) -- (7,7) -- (7,8) -- (8,8);
                \draw[rounded corners=1, color=Black, line width=1.5] (0,0) -- (0,6) -- (4,6) -- (4, 8) -- (8,8);
                \fill (0,0) circle[radius=3pt];
                \fill (0,0) circle[radius=3pt] node[left] {$r_1$};
                \fill (0,2) circle[radius=3pt];
                \fill (0,2) circle[radius=3pt] node[left] {$r_3$};
                \fill (0,4) circle[radius=3pt];
                \fill (0,4) circle[radius=3pt] node[left] {$r_5$};
                \fill (4,6) circle[radius=3pt];
                \fill (4,6) circle[radius=3pt] node[above left] {$r_{7}$};
                \draw[dashed] (0,0) -- (8,8);
                \fill (8,8) circle[radius=3pt];
                \draw[dashed] (0,2) -- (6,8);
                \fill (6,8) circle[radius=3pt];
                \draw[dashed] (0,4) -- (2,6);
                \fill (2,6) circle[radius=3pt];
            \end{scope}
        \end{tikzpicture}
    \end{center}

    Let's describe how this works with our algorithms directly.
    On the left we are doing the algorithm we just described and on the right hand side we are doing the algorithm found in the proof of \autoref{lem:bij_nm}.
    \begin{figure}[H]
        \caption{An example of the algorithm for $\bar{\phi}$ on the left and for $\phi$ on the right.}
        \label{fig:phi-ex}
    \begin{center}
        \begin{tikzpicture}[scale=0.3]
            \begin{scope}[shift={(0,0)}]
                \node at (-4.5,6) {$D = D^0 = $};
                \draw[dotted] (0, 0) grid (12, 12);
                \draw[rounded corners=1, color=BrightRed, ultra thick] (0, 0) -- (0, 1) -- (1, 1) -- (1, 2)-- (2, 2) -- (2,3) -- (3,3) -- (3,4) -- (4, 4) -- (4,5) -- (5,5) -- (5,6) -- (6,6) -- (6,7) -- (7,7) -- (7,8) -- (8,8) -- (8,9) -- (9,9) -- (9,10) -- (10,10) -- (10,11) -- (11,11) -- (11, 12) -- (12,12);
                \draw[rounded corners=1, color=Black, ultra thick] (0,0) -- (0,9) -- (5,9) -- (5, 12) -- (12,12);
                \fill (0,0) circle[radius=5pt];
                \fill (0,0) circle[radius=5pt] node[left] {$r_1$};
                \fill (0,3) circle[radius=5pt];
                \fill (0,3) circle[radius=5pt] node[left] {$r_4$};
                \fill (0,6) circle[radius=5pt];
                \fill (0,6) circle[radius=5pt] node[left] {$r_7$};
                \fill (5,9) circle[radius=5pt];
                \fill (5,9) circle[radius=5pt] node[above left] {$r_{10}$};
                \draw[dashed] (0,0) -- (12,12);
                \fill (12,12) circle[radius=5pt];
                \draw[dashed] (0,3) -- (9,12);
                \fill (9,12) circle[radius=5pt];
                \draw[dashed] (0,6) -- (3,9);
                \fill (3,9) circle[radius=5pt];
                \draw[dashed] (5,9) -- (8,12);
                \fill (8,12) circle[radius=5pt];
                \draw[color=blue] (12,12) circle[radius=12pt];
                \draw[color=blue, line width=1.6] (11,12) -- (12,12);
            \end{scope}
            \begin{scope}[shift={(0,-13)}]
                \node at (-3.5,6) {$D^1 = $};
                \draw[dotted] (0, 0) grid (12, 12);
                \draw[rounded corners=1, color=BrightRed, ultra thick] (0, 0) -- (0, 1) -- (1, 1) -- (1, 2)-- (2, 2) -- (2,3) -- (3,3) -- (3,4) -- (4, 4) -- (4,5) -- (5,5) -- (5,6) -- (6,6) -- (6,7) -- (7,7) -- (7,8) -- (8,8) -- (8,9) -- (9,9) -- (9,10) -- (10,10) -- (10,11) -- (11,11) -- (11, 12) -- (12,12);
                \draw[rounded corners=1, color=Black, ultra thick] (0,0) -- (0,9) -- (5,9) -- (5, 12) -- (12,12);
                \fill (0,0) circle[radius=5pt];
                \fill (0,0) circle[radius=5pt] node[left] {$r_1$};
                \fill (0,3) circle[radius=5pt];
                \fill (0,3) circle[radius=5pt] node[left] {$r_4$};
                \fill (0,6) circle[radius=5pt];
                \fill (0,6) circle[radius=5pt] node[left] {$r_7$};
                \fill (5,9) circle[radius=5pt];
                \fill (5,9) circle[radius=5pt] node[above left] {$r_{10}$};
                \draw[dashed] (0,0) -- (12,12);
                \fill (12,12) circle[radius=5pt];
                \draw[dashed] (0,3) -- (9,12);
                \fill (9,12) circle[radius=5pt];
                \draw[dashed] (0,6) -- (3,9);
                \fill (3,9) circle[radius=5pt];
                \draw[dashed] (5,9) -- (8,12);
                \fill (8,12) circle[radius=5pt];
                \draw[color=blue] (8,12) circle[radius=12pt];
                \draw[color=blue, line width=1.6] (9,12) -- (10,12);
            \end{scope}
            \begin{scope}[shift={(0,-26)}]
                \node at (-3.5,6) {$D^2 = $};
                \draw[dotted] (0, 0) grid (12, 12);
                \draw[rounded corners=1, color=BrightRed, ultra thick] (0, 0) -- (0, 1) -- (1, 1) -- (1, 2)-- (2, 2) -- (2,3) -- (3,3) -- (3,4) -- (4, 4) -- (4,5) -- (5,5) -- (5,6) -- (6,6) -- (6,7) -- (7,7) -- (7,8) -- (8,8) -- (8,9) -- (9,9) -- (9,10) -- (10,10) -- (10,11) -- (11,11) -- (11, 12) -- (12,12);
                \draw[rounded corners=1, color=Black, ultra thick] (0,0) -- (0,9) -- (5,9) -- (5, 12) -- (12,12);
                \fill (0,0) circle[radius=5pt];
                \fill (0,0) circle[radius=5pt] node[left] {$r_1$};
                \fill (0,3) circle[radius=5pt];
                \fill (0,3) circle[radius=5pt] node[left] {$r_4$};
                \fill (0,6) circle[radius=5pt];
                \fill (0,6) circle[radius=5pt] node[left] {$r_7$};
                \fill (5,9) circle[radius=5pt];
                \fill (5,9) circle[radius=5pt] node[above left] {$r_{10}$};
                \draw[dashed] (0,0) -- (12,12);
                \fill (12,12) circle[radius=5pt];
                \draw[dashed] (0,3) -- (9,12);
                \fill (9,12) circle[radius=5pt];
                \draw[dashed] (0,6) -- (3,9);
                \fill (3,9) circle[radius=5pt];
                \draw[dashed] (5,9) -- (8,12);
                \fill (8,12) circle[radius=5pt];
                \draw[color=blue] (3,9) circle[radius=12pt];
                \draw[color=blue, line width=1.6] (3,9) -- (4,9);
            \end{scope}
            \begin{scope}[shift={(0,-39)}]
                \node at (-3.5,6) {$D^3 = $};
                \draw[dotted] (0, 0) grid (12, 12);
                \draw[rounded corners=1, color=BrightRed, ultra thick] (0, 0) -- (0, 1) -- (1, 1) -- (1, 2)-- (2, 2) -- (2,3) -- (3,3) -- (3,4) -- (4, 4) -- (4,5) -- (5,5) -- (5,6) -- (6,6) -- (6,7) -- (7,7) -- (7,8) -- (8,8) -- (8,9) -- (9,9) -- (9,10) -- (10,10) -- (10,11) -- (11,11) -- (11, 12) -- (12,12);
                \draw[rounded corners=1, color=Black, ultra thick] (0,0) -- (0,9) -- (4,9) -- (4, 12) -- (12,12);
                \fill (0,0) circle[radius=5pt];
                \fill (0,0) circle[radius=5pt] node[left] {$r_1$};
                \fill (0,3) circle[radius=5pt];
                \fill (0,3) circle[radius=5pt] node[left] {$r_4$};
                \fill (0,6) circle[radius=5pt];
                \fill (0,6) circle[radius=5pt] node[left] {$r_7$};
                \fill (4,9) circle[radius=5pt];
                \fill (4,9) circle[radius=5pt] node[above left] {$r_{10}$};
                \draw[dashed] (0,0) -- (12,12);
                \fill (12,12) circle[radius=5pt];
                \draw[dashed] (0,3) -- (9,12);
                \fill (9,12) circle[radius=5pt];
                \draw[dashed] (0,6) -- (3,9);
                \fill (3,9) circle[radius=5pt];
                \draw[dashed] (4,9) -- (7,12);
                \fill (7,12) circle[radius=5pt];
                \draw[color=blue] (7,12) circle[radius=12pt];
                \draw[color=blue, line width=1.6] (7,12) -- (8,12);
            \end{scope}
            \begin{scope}[shift={(0,-52)}]
                \node at (-3.5,6) {$D^4 = $};
                \draw[dotted] (0, 0) grid (12, 12);
                \draw[rounded corners=1, color=BrightRed, ultra thick] (0, 0) -- (0, 1) -- (1, 1) -- (1, 2)-- (2, 2) -- (2,3) -- (3,3) -- (3,4) -- (4, 4) -- (4,5) -- (5,5) -- (5,6) -- (6,6) -- (6,7) -- (7,7) -- (7,8) -- (8,8) -- (8,9) -- (9,9) -- (9,10) -- (10,10) -- (10,11) -- (11,11) -- (11, 12) -- (12,12);
                \draw[rounded corners=1, color=Black, ultra thick] (0,0) -- (0,9) -- (4,9) -- (4, 12) -- (12,12);
                \fill (0,0) circle[radius=5pt];
                \fill (0,0) circle[radius=5pt] node[left] {$r_1$};
                \fill (0,3) circle[radius=5pt];
                \fill (0,3) circle[radius=5pt] node[left] {$r_4$};
                \fill (0,6) circle[radius=5pt];
                \fill (0,6) circle[radius=5pt] node[left] {$r_7$};
                \fill (4,9) circle[radius=5pt];
                \fill (4,9) circle[radius=5pt] node[above left] {$r_{10}$};
                \draw[dashed] (0,0) -- (12,12);
                \fill (12,12) circle[radius=5pt];
                \draw[dashed] (0,3) -- (9,12);
                \fill (9,12) circle[radius=5pt];
                \draw[dashed] (0,6) -- (3,9);
                \fill (3,9) circle[radius=5pt];
                \draw[dashed] (4,9) -- (7,12);
                \fill (7,12) circle[radius=5pt];
            \end{scope}
            \begin{scope}[shift={(0,-61)}]
                \node at (-4.5,4) {$\bar{\phi}(D) = $};
                \draw[dotted] (0, 0) grid (8, 8);
                \draw[rounded corners=1, color=BrightRed, ultra thick] (0, 0) -- (0, 1) -- (1, 1) -- (1, 2)-- (2, 2) -- (2,3) -- (3,3) -- (3,4) -- (4, 4) -- (4,5) -- (5,5) -- (5,6) -- (6,6) -- (6,7) -- (7,7) -- (7,8) -- (8,8);
                \draw[rounded corners=1, color=Black, ultra thick] (0,0) -- (0,6) -- (4,6) -- (4, 8) -- (8,8);
                \fill (0,0) circle[radius=5pt];
                \fill (0,0) circle[radius=5pt] node[left] {$r_1$};
                \fill (0,2) circle[radius=5pt];
                \fill (0,2) circle[radius=5pt] node[left] {$r_3$};
                \fill (0,4) circle[radius=5pt];
                \fill (0,4) circle[radius=5pt] node[left] {$r_5$};
                \fill (4,6) circle[radius=5pt];
                \fill (4,6) circle[radius=5pt] node[above left] {$r_{7}$};
                \draw[dashed] (0,0) -- (8,8);
                \fill (8,8) circle[radius=5pt];
                \draw[dashed] (0,2) -- (6,8);
                \fill (6,8) circle[radius=5pt];
                \draw[dashed] (0,4) -- (2,6);
                \fill (2,6) circle[radius=5pt];
                \draw[dashed] (4,6) -- (6,8);
                \fill (6,8) circle[radius=5pt];
            \end{scope}
            \begin{scope}[shift={(22,5)}]
                \node at (-4.5,2) {\small$\pi^{-1}(D) = $};
                \draw[dotted] (0, 0) grid (12, 4);
                \draw[rounded corners=1, color=BrightRed, line width=1.5] (0, 0) -- (0, 1) -- (3, 1) -- (3, 2)-- (6, 2) -- (6,3) -- (9,3) -- (9,4) -- (12, 4);
                \draw[rounded corners=1, color=Black, line width=1.5] (0,0) -- (0,3) -- (5,3) -- (5, 4) -- (12,4);
                \fill (0,0) circle[radius=5pt];
                \fill (0,0) circle[radius=5pt] node[left] {$r_1$};
                \fill (0,1) circle[radius=5pt];
                \fill (0,1) circle[radius=5pt] node[left] {$r_2$};
                \fill (0,2) circle[radius=5pt];
                \fill (0,2) circle[radius=5pt] node[left] {$r_3$};
                \fill (5,3) circle[radius=5pt];
                \fill (5,3) circle[radius=5pt] node[above left] {$r_{4}$};
                \draw[dashed] (0,0) -- (12,4);
                \fill (12,4) circle[radius=5pt];
                \draw[dashed] (0,1) -- (9,4);
                \fill (9,4) circle[radius=5pt];
                \draw[dashed] (0,2) -- (3,3);
                \fill (3,3) circle[radius=5pt];
                \draw[dashed] (5,3) -- (8,4);
                \fill (8,4) circle[radius=5pt];
            \end{scope}
            \begin{scope}[shift={(22,-8)}]
                \node at (-4.5,2) {\small$\pi^{-1}(D)^1 = $};
                \draw[dotted] (0, 0) grid (12, 4);
                \draw[rounded corners=1, color=BrightRed, line width=1.5] (0, 0) -- (0, 1) -- (3, 1) -- (3, 2)-- (6, 2) -- (6,3) -- (9,3) -- (9,4) -- (12, 4);
                \draw[rounded corners=1, color=Black, line width=1.5] (0,0) -- (0,3) -- (5,3) -- (5, 4) -- (12,4);
                \fill (0,0) circle[radius=5pt];
                \fill (0,0) circle[radius=5pt] node[left] {$r_1$};
                \fill (0,1) circle[radius=5pt];
                \fill (0,1) circle[radius=5pt] node[left] {$r_2$};
                \fill (0,2) circle[radius=5pt];
                \fill (0,2) circle[radius=5pt] node[left] {$r_3$};
                \fill (5,3) circle[radius=5pt];
                \fill (5,3) circle[radius=5pt] node[above left] {$r_{4}$};
                \draw[dashed] (0,0) -- (12,4);
                \fill (12,4) circle[radius=5pt];
                \draw[dashed] (0,1) -- (9,4);
                \fill (9,4) circle[radius=5pt];
                \draw[dashed] (0,2) -- (3,3);
                \fill (3,3) circle[radius=5pt];
                \draw[dashed] (5,3) -- (8,4);
                \fill (8,4) circle[radius=5pt];
                \draw[color=blue] (9,4) circle[radius=12pt];
                \draw[color=blue, line width=1.6] (9,4) -- (10,4);
            \end{scope}
            \begin{scope}[shift={(22,-21)}]
                \node at (-4.5,2) {\small$\pi^{-1}(D)^2 = $};
                \draw[dotted] (0, 0) grid (12, 4);
                \draw[rounded corners=1, color=BrightRed, line width=1.5] (0, 0) -- (0, 1) -- (3, 1) -- (3, 2)-- (6, 2) -- (6,3) -- (9,3) -- (9,4) -- (12, 4);
                \draw[rounded corners=1, color=Black, line width=1.5] (0,0) -- (0,3) -- (5,3) -- (5, 4) -- (12,4);
                \fill (0,0) circle[radius=5pt];
                \fill (0,0) circle[radius=5pt] node[left] {$r_1$};
                \fill (0,1) circle[radius=5pt];
                \fill (0,1) circle[radius=5pt] node[left] {$r_2$};
                \fill (0,2) circle[radius=5pt];
                \fill (0,2) circle[radius=5pt] node[left] {$r_3$};
                \fill (5,3) circle[radius=5pt];
                \fill (5,3) circle[radius=5pt] node[above left] {$r_{4}$};
                \draw[dashed] (0,0) -- (12,4);
                \fill (12,4) circle[radius=5pt];
                \draw[dashed] (0,1) -- (9,4);
                \fill (9,4) circle[radius=5pt];
                \draw[dashed] (0,2) -- (3,3);
                \fill (3,3) circle[radius=5pt];
                \draw[dashed] (5,3) -- (8,4);
                \fill (8,4) circle[radius=5pt];
                \draw[color=blue] (3,3) circle[radius=12pt];
                \draw[color=blue, line width=1.6] (3,3) -- (4,3);
            \end{scope}
            \begin{scope}[shift={(22,-34)}]
                \node at (-4.5,2) {\small$\pi^{-1}(D)^3 = $};
                \draw[dotted] (0, 0) grid (12, 4);
                \draw[rounded corners=1, color=BrightRed, line width=1.5] (0, 0) -- (0, 1) -- (3, 1) -- (3, 2)-- (6, 2) -- (6,3) -- (9,3) -- (9,4) -- (12, 4);
                \draw[rounded corners=1, color=Black, line width=1.5] (0,0) -- (0,3) -- (4,3) -- (4, 4) -- (12,4);
                \fill (0,0) circle[radius=5pt];
                \fill (0,0) circle[radius=5pt] node[left] {$r_1$};
                \fill (0,1) circle[radius=5pt];
                \fill (0,1) circle[radius=5pt] node[left] {$r_2$};
                \fill (0,2) circle[radius=5pt];
                \fill (0,2) circle[radius=5pt] node[left] {$r_3$};
                \fill (4,3) circle[radius=5pt];
                \fill (4,3) circle[radius=5pt] node[above left] {$r_{4}$};
                \draw[dashed] (0,0) -- (12,4);
                \fill (12,4) circle[radius=5pt];
                \draw[dashed] (0,1) -- (9,4);
                \fill (9,4) circle[radius=5pt];
                \draw[dashed] (0,2) -- (3,3);
                \fill (3,3) circle[radius=5pt];
                \draw[dashed] (4,3) -- (7,4);
                \fill (7,4) circle[radius=5pt];
                \draw[color=blue] (7,4) circle[radius=12pt];
                \draw[color=blue, line width=1.6] (7,4) -- (8,4);
            \end{scope}
            \begin{scope}[shift={(22,-47)}]
                \node at (-4.5,2) {\small$\pi^{-1}(D)^4 = $};
                \draw[dotted] (0, 0) grid (12, 4);
                \draw[rounded corners=1, color=BrightRed, line width=1.5] (0, 0) -- (0, 1) -- (3, 1) -- (3, 2)-- (6, 2) -- (6,3) -- (9,3) -- (9,4) -- (12, 4);
                \draw[rounded corners=1, color=Black, line width=1.5] (0,0) -- (0,3) -- (4,3) -- (4, 4) -- (12,4);
                \fill (0,0) circle[radius=5pt];
                \fill (0,0) circle[radius=5pt] node[left] {$r_1$};
                \fill (0,1) circle[radius=5pt];
                \fill (0,1) circle[radius=5pt] node[left] {$r_2$};
                \fill (0,2) circle[radius=5pt];
                \fill (0,2) circle[radius=5pt] node[left] {$r_3$};
                \fill (4,3) circle[radius=5pt];
                \fill (4,3) circle[radius=5pt] node[above left] {$r_{4}$};
                \draw[dashed] (0,0) -- (12,4);
                \fill (12,4) circle[radius=5pt];
                \draw[dashed] (0,1) -- (9,4);
                \fill (9,4) circle[radius=5pt];
                \draw[dashed] (0,2) -- (3,3);
                \fill (3,3) circle[radius=5pt];
                \draw[dashed] (4,3) -- (7,4);
                \fill (7,4) circle[radius=5pt];
            \end{scope}
            \begin{scope}[shift={(22,-56)}]
                \node at (-5,2) {\small$\phi(\pi^{-1}(D)) = $};
                \draw[dotted] (0, 0) grid (8, 4);
                \draw[rounded corners=1, color=BrightRed, line width=1.5] (0, 0) -- (0, 1) -- (2, 1) -- (2, 2)-- (4, 2) -- (4,3) -- (6,3) -- (6,4) -- (8, 4);
                \draw[rounded corners=1, color=Black, line width=1.5] (0,0) -- (0,3) -- (4,3) -- (4, 4) -- (8,4);
                \fill (0,0) circle[radius=5pt];
                \fill (0,0) circle[radius=5pt] node[left] {$r_1$};
                \fill (0,1) circle[radius=5pt];
                \fill (0,1) circle[radius=5pt] node[left] {$r_2$};
                \fill (0,2) circle[radius=5pt];
                \fill (0,2) circle[radius=5pt] node[left] {$r_3$};
                \fill (4,3) circle[radius=5pt];
                \fill (4,3) circle[radius=5pt] node[above left] {$r_{4}$};
                \draw[dashed] (0,0) -- (8,4);
                \fill (8,4) circle[radius=5pt];
                \draw[dashed] (0,1) -- (6,4);
                \fill (6,4) circle[radius=5pt];
                \draw[dashed] (0,2) -- (2,3);
                \fill (2,3) circle[radius=5pt];
                \draw[dashed] (4,3) -- (6,4);
                \fill (6,4) circle[radius=5pt];
            \end{scope}
        \end{tikzpicture}
    \end{center}
    \end{figure}
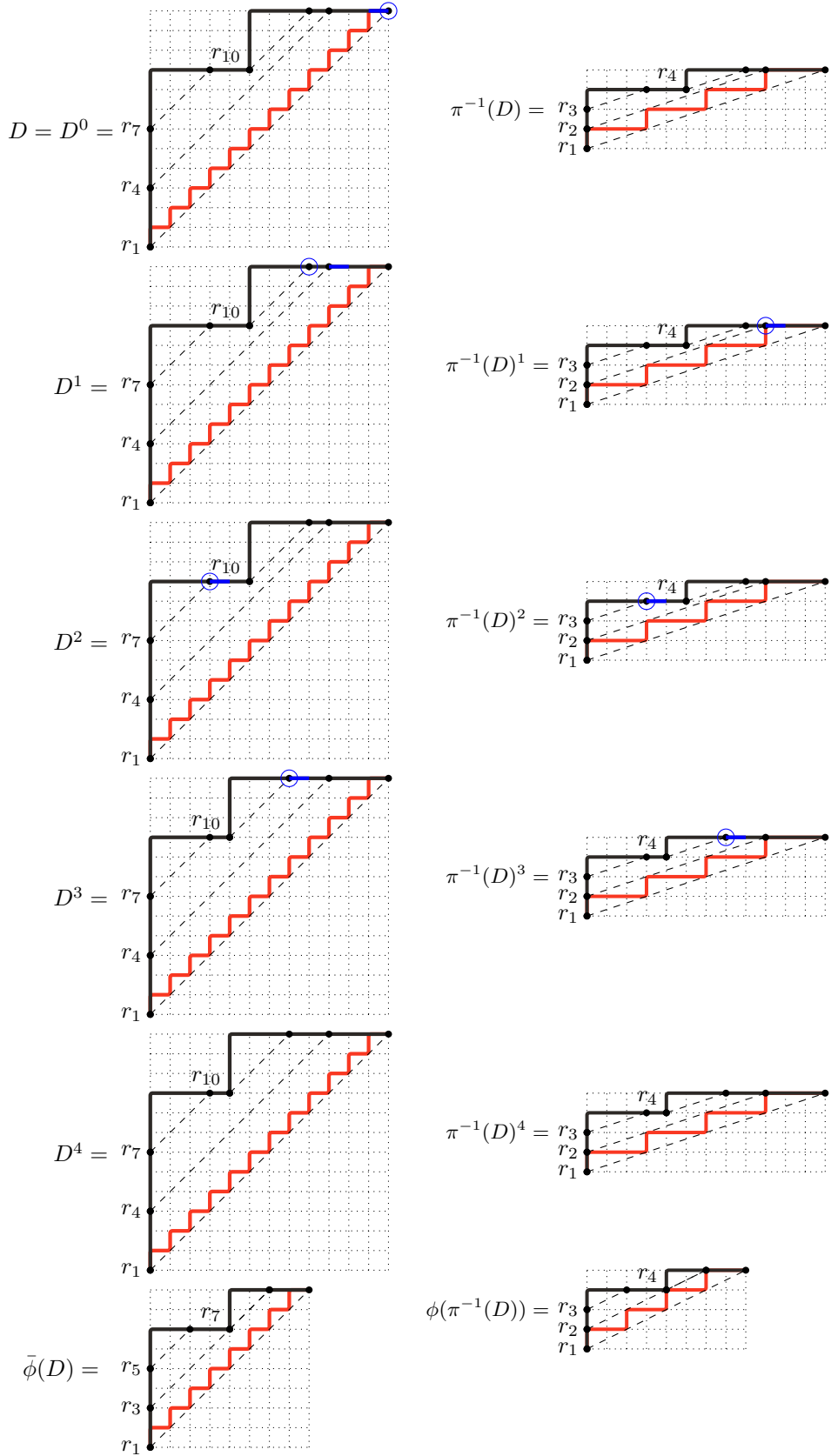
    Removing the final $n = 4$ $E$ steps and removing one north step for every $3$ north steps (in other words, above each $r_{im +1}$) gives us the same Dyck path as previously.
\end{example}

Notice that in the algorithm for $\bar{\phi}$ we can move the north step after the right hand point at the same time as the east step after the hit point and place the $NE$ at the end of Dyck path.
This would mean we would need to remove $(NE)^n$ at the end of the Dyck path to get $\bar{\phi}(D)$ and, in addition, since we are removing north steps, we would have to change our indices from $im + 1$ to $im - i + 1$.
We will use this alternative approach in the following lemmas.
We start by proving a technical lemma.

\begin{lemma}
    \label{lem:dist-fn-phibar}
    For $D \in \Dinnm$ let $D^i$ be the $i$-th iteration of the (alternative) algorithm for $\bar{\phi}(\pi(D))$.
    Then for $j \in [mn]$ we have the following distance functions,
    \begin{align*}
        d_{D^{i+1}}^t(j) &= \begin{cases}
            1 & \text{if } j \geq mn - i\\
            d_{D^{i}}^t(j+1) & \text{if }mn - i > j \geq im - i + 1\\
            d_{D^{i}}^t(j) - 1 & \text{if }j < im-i + 1\text{ and } t_j^{D^{i}} \text{ comes weakly after } t_{im-i+1}^{D^{i}}\\
            d_{D^{i}}^t(j) & \text{otherwise }
        \end{cases}\\
        d_{D^{i+1}}^h(j) &= \begin{cases}
            mn-j+1 & \text{if } j \geq mn - i\\
            d_{D^{i}}^h(j+1) & \text{if } mn - i > j \geq im - i + 1 \text{ and } j \neq 1\\
            mn & \text{if } j = 1\\
            d_{D^{i}}^h(j) - 1 & \text{if }j < im - i + 1 \text{ and } h_j^{D^{i}} \text{ comes weakly after } h_{im-i+1}^{D^{i}}\\
            d_{D^{i}}^h(j) & \text{otherwise }
        \end{cases}
    \end{align*}
\end{lemma}
\begin{proof}
    We show this case by case.

    If $j \geq mn - i$, then we have a chain of $NE$ steps by construction, giving us those entries.

    If $mn - i > j \geq im - i + 1$ (and $j \neq 1$ for the hit distance), then the subpath from the row $im - i + 1$ to the final point is pushed south by one step.
    If a point is between $im - i + 1$ and the row containing $t_{im - i + 1} = h_{im -i + 1}$, then the distance to the touch and hit points do not change.
    In particular, note that there must be an east step before $t_{im - i + 1}$ else we would have a contradiction for $t_{im - i + 1} = h_{im - i + 1}$.
    Furthermore, if a point is between the row containing $t_{im - i + 1}$ and the row $mn - i$, then the distances between all points don't change.
    Therefore, for every row we can just grab the entry of $d_{D^{i}}^t(j+1)$ and $d_{D^{i}}^h(j+1)$ respectively.

    If $j = 1$, then the hit distance never changes since the horizontal distance is always $0$, giving us $mn$ ofr the hit distance.

    If $j < im - i + 1$ then we have three subcases to consider.
    If the hit point of $j$ comes strictly before $t_{im - i + 1}^{D^{i}}$, then the horizontal distance of $r_j^{D^{i}}$ is greater than $r_{im - i + 1}^{D^{i}}$ and therefore the distances don't change.
    If the touch point of $j$ comes (weakly) after $t_{im - i + 1}^{D^{i}}$, then the touch and hit points both move south and to the west by one point and thus the ditsances are decreased by one.
    Finally, if $h_{j}^{D^{i}} = h_{im - i + 1}^{D^{i}}$ and $t_j^{D^{i}} = r_{im - i + 1}^{D^{i}}$ then the distance to the hit point decreases by one, but the distance to the touch point does not change.
\end{proof}

\begin{proposition}
    \label{prop:dex-to-greedy}
    Let $\phi$ be the map from \autoref{lem:bij_nm}.
    For $D, D' \in \Dinnm$ such that $D \leq_T D'$ then $\phi(D) \leq_G \phi(D')$.
\end{proposition}
\begin{proof}
    Let $\pi$, $\phi$, $\bar{\phi}$ be as defined above and $D^i$ as defined in \autoref{lem:dist-fn-phibar}.
    We restrict our attention to $d_{\bar{\phi}(\pi(D))}^t$ and $d_{\bar{\phi}(\pi(D))}^h$.
    In the rest of this proof, by ``comes after'', we mean comes weakly after instead of comes strictly after unless otherwise specified.
    
    It suffices to show $d_{D^{i+1}}^t \leq d_{D'^{i+1}}^t$ and $d_{D^{i+1}}^h \leq d_{D'^{i+1}}^h$ for all $i \in \left\{ 0, \ldots, n-1 \right\}$ where $D'^{i}$ is defined identically to $D^{i}$.
    By the distance equations in \autoref{lem:dist-fn-phibar}, the only time that $d_{D'^{i+1}}^t < d_{D^{i+1}}^t$ could happen is if there exists a $j < im - i + 1$ such that
    \begin{itemize}
        \item $d_{D^{i}}^t (j) = d_{{D'}^{i}}^t (j)$,
        \item $t_j^{D^{i}}$ does \emph{not} come after $t_{im-i+1}^{D^{i}}$, and
        \item $t_j^{{D'}^{i}}$ comes after $t_{im -i+ 1}^{{D'}^{i}}$.
    \end{itemize}
    Similarly, the only time that $d_{D'^{i+1}}^h < d_{D^{i+1}}^h$ is if there exists a $j < im - i + 1$ such that
    \begin{itemize}
        \item $d_{D^{i}}^h (j) = d_{{D'}^{i}}^h (j)$,
        \item $h_j^{D^{i}}$ does \emph{not} come after $h_{im-i+1}^{D^{i}}$, and
        \item $h_j^{{D'}^{i}}$ comes after $h_{im -i+ 1}^{{D'}^{i}}$.
    \end{itemize}

    We first show the first case is not possible through contradiction.
    Therefore, suppose that there exists (a minimal) $i$ and (a minimal) $j < im - i +1$ such that $d_{D^{i}}^t (j) = d_{{D'}^{i}}^t (j)$, $t_j^{D^{i}}$ does \emph{not} come after $t_{im-i+1}^{D^{i}}$, and $t_j^{{D'}^{i}}$ comes after $t_{im -i+ 1}^{{D'}^{i}}$.
    Then $d_{D^{i}}^t (j) = d_{{D'}^{i}}^t (j)$ implies that $t_j^{D^{i}}$ and $t_j^{{D'}^{i}}$ lie on the same row.
    Since $D \leq_T D'$ implies $d_D^t \leq d_D'^t$ by \autoref{prop:tamari-iff-distance} and by the minimality of $i$ and $j$, then $t_j^{D^{i}}$ comes after $t_{j}^{D'^{i}}$.
    Then the facts that $t_j^{D^{i}}$ does \emph{not} come after $t_{im-i+1}^{D^{i}}$ and $t_j^{{D'}^{i}}$ comes after $t_{im -i+ 1}^{{D'}^{i}}$ together with $t_j^{D^i}$ comes after $t_j^{D'^i}$ imply that $t_{im -i+ 1}^{D^{i-1}}$ comes strictly after $t_{im - i + 1}^{D'^{i-1}}$.
    This implies $d_{D^{i}}^t(im-i + 1) > d_{D'^{i}}^t(im -i+ 1)$ contradicting the fact that $d_{D}^t(im -i+1) \leq d_{D'}^{t}(im-i+1)$ as $D \leq_T D'$.

    We next show the second case is not possible (again) through contradiction.
    Suppose that there exists (a minimal) $i$ and (a minimal) $j < im - i + 1$ such that $d_{D^{i}}^h (j) = d_{{D'}^{i}}^h (j)$, $h_j^{D^{i}}$ does \emph{not} come after $h_{im-i+1}^{D^{i}}$, and $h_j^{{D'}^{i}}$ comes after $h_{im -i+ 1}^{{D'}^{i}}$.
    Then $d_{D^{i}}^h (j) = d_{{D'}^{i}}^h (j)$ implies that $h_j^{D^{i}}$ and $h_j^{{D'}^{i}}$ lie on the same row.
    Since $D \leq_T D'$ implies $d_D^t \leq d_D'^t$ by \autoref{prop:tamari-iff-distance} and by the minimality of $i$ and $j$, then $h_j^{D^{i}}$ comes after $h_{j}^{D'^{i}}$.
    Then the facts that $h_j^{D^{i}}$ does \emph{not} come after $h_{im-i+1}^{D^{i}}$ and $h_j^{{D'}^{i}}$ comes after $h_{im -i+ 1}^{{D'}^{i}}$ together with $h_j^{D^i}$ comes after $h_j^{D'^i}$ imply $h_{im -i+ 1}^{D^{i-1}}$ comes strictly after $h_{im - i + 1}^{D'^{i-1}}$.
    As, $D$ and $D'$ live in $\pi(\Dinnm)$, then $h_{im -i+ 1}^{D^{i}} = t_{im -i+ 1 }^{D^{i}}$ and $h_{im -i+ 1}^{D'^{i}} = t_{im -i+ 1 }^{D'^{i}}$.
    Therefore, $t_{im-i + 1}^{D^{i}}$ comes strictly after $t_{im -i+ 1}^{D'^{i}}$.
    This implies $d_{D^{i}}^t(im-i + 1) > d_{D'^{i}}^t(im -i+ 1)$ contradicting the fact that $d_{D}^t(im -i+1) \leq d_{D'}^{t}(im-i+1)$ as $D \leq_T D'$.

    Therefore, none of these two cases happen.
    In particular, removing the final $(NE)^n$ of $D^n$ and $D'^{n}$ gives us $\bar{\phi}(\pi(D))$ and $\bar{\phi}(\pi(D'))$ and thus $d_{\bar{\phi}(\pi(D))}^t\leq d_{\bar{\phi}(\pi(D'))}^t$ and $d_{\bar{\phi}(\pi(D))}^h \leq d_{\bar{\phi}(\pi(D'))}^h$.
    By \autoref{prop:greedy-distance}, we have $\bar{\phi}(\pi(D)) \leq_G \bar{\phi}(\pi(D'))$.
    Using the fact that $\bar{\phi} = \pi \circ \phi \circ \pi^{-1}$ and the fact that $\pi$ is a bijection (on the restriction) gives us $\phi(D) \leq_G \phi(D')$ as desired.
\end{proof}

Putting this all together, we can prove our final main theorem.

~
\begin{proof}[Proof of \autoref{thm:iso-dexter-greedy}]
    This is a direct result of \autoref{prop:greedy-to-dex} and \autoref{prop:dex-to-greedy}.
\end{proof}

\subsection{Arbitrary \texorpdfstring{$\nu$}{v}}
It might be tempting to look for ways to extend these result to arbitrary $\nu$.
It turns out that there are $\nu$ for which $\Dinnu$ is not in (set) bijection with any $\mcD_{\nu'}$ where $\nu'$ is a path weakly above $\nu$.
For paths from $(0, 0)$ to $(n, n)$ it can be shown that if $n \leq 3$ that there is always a set bijection, but when $n = 4$ there are two paths in which this fails.
In particular when $\nu$ is equal to either
\[
    ENNNEENE \quad \text{ or } \quad NENNEEEN
\]
then no such bijection exists.

In the following table we give the number of $\nu$ for which no such (set) bijection exists.
The symmetric nature of the table comes from the following fact.
\begin{theorem}{{\cite[Theorem 2]{PV-extTam}}}
    Given a path $\nu$, the $\nu$-Tamari lattice $\mbbT_\nu$ is poset isomorphic to $\mbbT_{\overleftarrow{\nu}}$ where $\overleftarrow{\nu}$ is the path obtained from ready $\nu$ backwards and replacing every north step by an east step and every east step by a north step.
\end{theorem}
Therefore, the top row can either be read as the number of north steps in $\nu$ and the left column as the number of east steps without loss of generality.
\begin{table}[h]
    \begin{center}
        \begin{tabular}{r || c | c | c |c|c|c|c|c|c|c|c }
                & 1     & 2     & 3     & 4     & 5     & 6     & 7     & 8     & 9     & 10    & 11\\
            \hline\hline
            1   & 0     & 0     & 0     & 0     & 0     & 0     & 0     & 0     & 0     & 0     & 0\\
            \hline
            2   & 0     & 0     & 0     & 0     & 0     & 0     & 0     & 0     & 0     & 0     & 0\\
            \hline
            3   & 0     & 0     & 0     & 1     & 2     & 2     & 3     & 4     & 4     & 5     & 6\\ 
            \hline
            4   & 0     & 0     & 1     & 2     & 4     & 6     & 9     & 10    & 15    & 19    & 26\\ 
            \hline
            5   & 0     & 0     & 2     & 4     & 16    & 21    & 28    & 35    & 46    & 54    & 67\\ 
            \hline
            6   & 0     & 0     & 2     & 6     & 21    & 38    & 51    & 71    & 97    & 150   & 189\\
            \hline
            7   & 0     & 0     & 3     & 9     & 28    & 51    & 77    & 121   & 187   & -     &-\\
            \hline
            8   & 0     & 0     & 4     & 10    & 35    & 71    & 121   & 210   & -     & -     &-\\
            \hline
            9   & 0     & 0     & 4     & 15    & 46    & 97    & 187   & -     & -     & -     &-\\
            \hline
            10  & 0     & 0     & 5     & 19    & 54    & 150   & -     & -     & -     & -     &-\\
            \hline
            11  & 0     & 0     & 6     & 26    & 67    & 189   & -     & -     & -     & -     &-\\
        \end{tabular}
    \end{center}
\end{table}
These results were obtained using sagemath \cite{sagemath}.
The OEIS does not seem to have any sequences which match the above set of numbers.
Spots with no numbers are places where our computer was not able to compute the value.

Continuing with our sagemath computations, we noticed that the number of elements with maximal in-degree is equal to the number of elements with maximal out-degree.
We conjecture this to be the case for arbitrary $\nu$.
\begin{conjecture}
    Let $\nu$ be an arbitrary path.
    Then $\order{\Dinnu} = \order{\Doutnu}$.
\end{conjecture}
This was verified for all $\nu$ with east and north steps as described in the table above and was verified for all $\nu = (NE^m)^n$ in \autoref{cor:in-eq-out}.

Some future directions in this area would be to study why there is a set bijection for some $\nu$, but not for others.
Is there some way to classify for which $\nu$ there is a $\nu'$ weakly above $\nu$ for which $\Dinnu$ has a set bijection?
Is there some (other) combinatorial object that gives us the numbers in the table?
Is there some nice formula in two variables which would give us the table above?

\section{Acknowledgements}
The author would like to thank Fr\'ed\'eric Chapoton for many fruitful conversations and discussions.

\nocite{*}
\printbibliography[title={References}]
\label{sec:biblio}

\end{document}